\documentclass{article}%
\usepackage{amsfonts}
\usepackage{graphicx}
\usepackage{amsmath}
\usepackage{amssymb}%
\providecommand{\U}[1]{\protect\rule{.1in}{.1in}}
\newtheorem{theorem}{Theorem}

\newtheorem{corollary}[theorem]{Corollary}

\newtheorem{definition}[theorem]{Definition}

\newtheorem{lemma}[theorem]{Lemma}
\newtheorem{notation}[theorem]{Notation}

\newtheorem{proposition}[theorem]{Proposition}
\newtheorem{remark}[theorem]{Remark}

\newenvironment{proof}[1][Proof]{\textbf{#1.} }{\ \rule{0.5em}{0.5em}}
\begin{document}

\title{Differential Geometry of Microlinear Fr\"{o}licher Spaces IV-1}
\author{Hirokazu Nishimura\\Institute of Mathematics\\University of Tsukuba\\Tsukuba, Ibaraki 305-8571\\Japan}
\maketitle

\begin{abstract}
The fourth paper of our series of papers entitled ''Differential Geometry of
Microlinear Fr\"{o}licher Spaces is concerned with jet bundles. We present
three distinct approaches together with transmogrifications of the first into
the second and of the second to the third. The affine bundle theorem and the
equivalence of the three approaches with coordinates are relegated to a
subsequent paper.

\end{abstract}

\section{Introduction}

As the fourth of our series of papers entitled ''Differential Geometry of
Microlinear Fr\"{o}licher Spaces'', this paper will discuss jet bundles. Since
the paper has become somewhat too long as a single paper, we have decided to
divide it into two parts. In this first part we will present three distinct
approaches to jet bundles in the general context of Weil exponentiable and
microlinear Fr\"{o}licher spaces. In the subsequent part, we will establish
the affine bundle theorem in the second and the third approaches, and we will
show that the three approaches are equivalent, as far as coordinates are
available (i.e., in the classical context).

This part consisits of 7 sections. The first section is this introduction,
while the second section is devoted to some preliminaries. We will present
three distinct approaches to jet bundles in Sections 3, 4 and 5. In Section 6
we will show how to translate the first approach into the second, while
Section 7 is devoted to the transmogrification of the second approach into the third.

We have already discussed these three approaches to jet bundles in the context
of synthetic differential geometry, for which the reader is referred to our
previous work \cite{ni1}, \cite{ni2}, \cite{ni3}, \cite{ni4}, \cite{ni5} and
\cite{ni6}. Now we have emancipated them to the real world of Fr\"{o}licher spaces.

\section{\label{s2}Preliminaries}

\subsection{\label{s2.1}Fr\"{o}licher Spaces}

Fr\"{o}licher and his followers have vigorously and consistently developed a
general theory of smooth spaces, often called \textit{Fr\"{o}licher spaces}
for his celebrity, which were intended to be the \textit{maximal class of
}spaces where smooth structures can live.\textit{\ }A Fr\"{o}licher space is
an underlying set endowed with a class of real-valued functions on it (simply
called \textit{structure} \textit{functions}) and a class of mappings from the
set $\mathbb{R}$ of real numbers to the underlying set (simply called
\textit{structure} \textit{curves}) subject to the condition that structure
curves and structure functions should compose so as to yield smooth mappings
from $\mathbb{R}$ to itself. It is required that the class of structure
functions and that of structure curves should determine each other so that
each of the two classes is maximal with respect to the other as far as they
abide by the above condition. What is most important among many nice
properties about the category $\mathbf{FS}$ of Fr\"{o}licher spaces and smooth
mappings is that it is cartesian closed, while neither the category of
finite-dimensional smooth manifolds nor that of infinite-dimensional smooth
manifolds modelled after any infinite-dimensional vector spaces such as
Hilbert spaces, Banach spaces, Fr\'{e}chet spaces or the like is so at all.
For a standard reference on Fr\"{o}licher spaces, the reader is referred to
\cite{fr3}.

\subsection{\label{s2.2}Weil Algebras and Infinitesimal Objects}

\subsubsection{\label{s2.2.1}The Category of Weil Algebras and the Category of
Infinitesimal Objects}

The notion of a \textit{Weil algebra} was introduced by Weil himself in
\cite{we}. We denote by $\mathbf{W}$ the category of Weil algebras, which is
well known to be left exact. Roughly speaking, each Weil algebra corresponds
to an infinitesimal object in the shade. By way of example, the Weil algebra
$\mathbb{R}[X]/(X^{2})$ (=the quotient ring of the polynomial ring
$\mathbb{R}[X]$\ of an indeterminate $X$\ over $\mathbb{R}$ modulo the ideal
$(X^{2})$\ generated by $X^{2}$) corresponds to the infinitesimal object of
first-order nilpotent infinitesimals, while the Weil algebra $\mathbb{R}%
[X]/(X^{3})$ corresponds to the infinitesimal object of second-order nilpotent
infinitesimals. Although an infinitesimal object is undoubtedly imaginary in
the real world, as has harassed both mathematicians and philosophers of the
17th and the 18th centuries such as philosopher Berkley (because
mathematicians at that time preferred to talk infinitesimal objects as if they
were real entities), each Weil algebra yields its corresponding \textit{Weil
functor} or \textit{Weil prolongation}\ on the category of smooth manifolds of
some kind to itself, which is no doubt a real entity. By way of example, the
Weil algebra $\mathbb{R}[X]/(X^{2})$ yields the tangent bundle functor as its
corresponding Weil functor. Intuitively speaking, the Weil functor
corresponding to a Weil algebra stands for the exponentiation by the
infinitesimal object corresponding to the Weil algebra at issue. For Weil
functors on the category of finite-dimensional smooth manifolds, the reader is
referred to \S 35 of \cite{kol}, while the reader can find a readable
treatment of Weil functors on the category of smooth manifolds modelled on
convenient vector spaces in \S 31 of \cite{kri}. In \cite{ni7} we have
discussed how to assign, to each pair $(X,W)$\ of a Fr\"{o}licher space $X$
and a Weil algebra $W$,\ another Fr\"{o}licher space $X\otimes W$\ called the
\textit{Weil prolongation of} $X$ \textit{with respect to} $W$, which is
naturally extended to a bifunctor $\mathbf{FS}\times\mathbf{W\rightarrow FS}
$. And we have shown that, given a Weil algebra $W$, the functor assigning
$X\otimes W$ to each object $X$ in $\mathbf{FS}$\ and $f\otimes\mathrm{id}%
_{W}$ to each morphism $f$\ in $\mathbf{FS}$, namely, the Weil functor on
$\mathbf{FS}$\ corresponding to $W$\ is product-preserving. The proof can
easily be strengthened to

\begin{theorem}
The Weil functor on the category $\mathbf{FS}$\ corresponding to any Weil
algebra is left exact.
\end{theorem}

There is a canonical projection $\pi:X\otimes W\rightarrow X$. Given $x\in X
$, we write $\left(  X\otimes W\right)  _{x}$ for the inverse image of
$x$\ under the mapping $\pi$. We denote by $\mathbf{S}_{n}$ the symmetric
group of the set $\{1,...,n\}$, which is well known to be generated by $n-1$
transpositions $<i,i+1>$ exchanging $i$ and $i+1(1\leq i\leq n-1)$ while
keeping the other elements fixed. Given $\sigma\in\mathbf{S}_{n}$ and
$\gamma\in X\otimes\mathcal{W}_{D^{n}}$, we define $\gamma^{\sigma}\in
X\otimes\mathcal{W}_{D^{n}}$ to be
\[
\gamma^{\sigma}=\left(  \mathrm{id}_{X}\otimes\mathcal{W}_{(d_{1}%
,...,d_{n})\in D^{n}\mapsto(d_{\sigma(1)},...,d_{\sigma(n)})\in D^{n}}\right)
\left(  \gamma\right)
\]
Given $\alpha\in\mathbb{R}$ and $\gamma\in X\otimes\mathcal{W}_{D^{n}}$, we
define $\alpha\underset{i}{\cdot}\gamma\in\gamma\in X\otimes\mathcal{W}%
_{D^{n}}\ (1\leq i\leq n)$ to be
\[
\alpha\underset{i}{\cdot}\gamma=\left(  \mathrm{id}_{X}\otimes\mathcal{W}%
_{(d_{1},...,d_{n})\in D^{n}\mapsto(d_{1},...,d_{i-1},\alpha d_{i}%
,d_{i+1},...,d_{n})\in D^{n}}\right)  \left(  \gamma\right)
\]
Given $\alpha\in\mathbb{R}$ and $\gamma\in X\otimes\mathcal{W}_{D_{n}}$, we
define $\alpha\gamma\in X\otimes\mathcal{W}_{D_{n}}\ (1\leq i\leq n)$ to be
\[
\alpha\gamma=\left(  \mathrm{id}_{X}\otimes\mathcal{W}_{d\in D_{n}%
\mapsto\alpha d\in D_{n}}\right)  \left(  \gamma\right)
\]
for any $d\in D_{n}$. The restriction mapping $\gamma\in\mathbf{T}%
_{x}^{D_{n+1}}(M)\mapsto\gamma|_{D_{n}}\in\mathbf{T}_{x}^{D_{n}}(M)$ is often
denoted by $\pi_{n+1,n}$.

Between $X\otimes\mathcal{W}_{D^{n}}$ and $X\otimes\mathcal{W}_{D^{n+1}}$
there are $2n+$ $2$ canonical mappings:
\[
X\otimes\mathcal{W}_{D^{n+1}}\ \
\begin{array}
[c]{c}%
\underrightarrow{\ \ \ \mathbf{d}_{i}\ \ \ }\\
\overleftarrow{\ \ \ \mathbf{s}_{i}\ \ \ }%
\end{array}
\ \ X\otimes\mathcal{W}_{D^{n}}\qquad(1\leq i\leq n+1)
\]
For any $\gamma\in X\otimes\mathcal{W}_{D^{n}}$, we define $\mathbf{s}%
_{i}(\gamma)\in X\otimes\mathcal{W}_{D^{n+1}}\ $to be
\[
\mathbf{s}_{i}(\gamma)=\left(  \mathrm{id}_{X}\otimes\mathcal{W}%
_{(d_{1},...,d_{n+1})\in D^{n+1}\mapsto(d_{1},...,d_{i-1},d_{i+1}%
,...,d_{n+1})\in D^{n}}\right)  \left(  \gamma\right)
\]
For any $\gamma\in X\otimes\mathcal{W}_{D^{n+1}}$, we define $\mathbf{d}%
_{i}(\gamma)\in X\otimes\mathcal{W}_{D^{n}}\ $to be
\[
\mathbf{d}_{i}(\gamma)=\left(  \mathrm{id}_{X}\otimes\mathcal{W}%
_{(d_{1},...,d_{n})\in D^{n}\mapsto(d_{1},...,d_{i-1},0,d_{i},...,d_{n})\in
D^{n+1}}\right)  \left(  \gamma\right)
\]
These operations satisfy the so-called simplicial identities (cf. Goerss and
Jardine \cite{go}), so that the family of $X\otimes\mathcal{W}_{D^{n}}$'s
together with mappings $\mathbf{s}_{i}$'s and $\mathbf{d}_{i}$'s form a
so-called simplicial set.

\textit{Synthetic differential geometry }(usually abbreviated to SDG), which
is a kind of differential geometry with a cornucopia of nilpotent
infinitesimals, was forced to invent its models, in which nilpotent
infinitesimals were visible. For a standard textbook on SDG, the reader is
referred to \cite{la}, while he or she is referred to \cite{ko} for the model
theory of SDG constructed vigorously by Dubuc \cite{du} and others. Although
we do not get involved in SDG herein, we will exploit locutions in terms of
infinitesimal objects so as to make the paper highly readable. Thus we prefer
to write $\mathcal{W}_{D}$\ and $\mathcal{W}_{D_{2}}$\ in place of
$\mathbb{R}[X]/(X^{2})$ and $\mathbb{R}[X]/(X^{3})$ respectively, where $D$
stands for the infinitesimal object of first-order nilpotent infinitesimals,
and $D_{2}$\ stands for the infinitesimal object of second-order nilpotent
infinitesimals. To Newton and Leibniz, $D$ stood for
\[
\{d\in\mathbb{R}\mid d^{2}=0\}
\]
while $D_{2}$\ stood for
\[
\{d\in\mathbb{R}\mid d^{3}=0\}
\]
More generally, given a natural number $n$, we denote by $D_{n}$ the set
\[
\{d\in\mathbb{R}|d^{n+1}=0\}\text{,}%
\]
which stands for the infinitesimal object corresponding to the Weil algebra
$\mathbb{R}[X]/(X^{n+1})$. Even more generally, given natural numbers $m,n$,
we denote by $D(m)_{n}$ the infinitesimal object
\[
\{(d_{1},...,d_{m})\in\mathbb{R}^{m}|d_{i_{1}}...d_{i_{n+1}}=0\}\text{,}%
\]
where $i_{1},...,i_{n+1}$ shall range over natural numbers between $1$ and $m
$ including both ends. It corresponds to the Weil algebra $\mathbb{R}%
[X_{1},...,X_{m}]/I$, where $I$ is the ideal generated by $X_{i_{1}%
}...X_{i_{n+1}}$'s. Therefore we have
\begin{align*}
D(1)_{n}  &  =D_{n}\\
D\left(  m\right)  _{1}  &  =D\left(  m\right)
\end{align*}
Trivially we have
\[
D(m)_{n}\subseteq D(m)_{n+1}%
\]
It is easy to see that
\begin{align*}
D(m_{1})_{n}\times D(m_{2})_{1}  &  \subseteq D(m_{1}+m_{2})_{n+1}\\
D(m_{1}+m_{2})_{n}  &  \subseteq D(m_{1})_{n}\times D(m_{2})_{n}%
\end{align*}
By convention, we have
\[
D^{0}=D_{0}=\{0\}=1
\]
A polynomial $\rho$ of $d\in D_{n}$ is called a \textit{simple} polynomial of
$d\in D_{n}$ if every coefficient of $\rho$ is either $1$ or $0$, and if the
constant term is $0$. A simple polynomial $\rho$ of $d\in D_{n}$ is said to be
of dimension $m$, in notation $\mathrm{\dim}(\rho)=m$, provided that $m$ is
the least integer with $\rho^{m+1}=0$. By way of example, letting $d\in D_{3}%
$, we have
\begin{align*}
\mathrm{\dim\,}(d)  &  =\mathrm{\dim\,}(d+d^{2})=\mathrm{\dim\,}(d+d^{3})=3\\
\mathrm{\dim\,}(d^{2})  &  =\mathrm{\dim\,}(d^{3})=\mathrm{\dim\,}(d^{2}%
+d^{3})=1
\end{align*}

We will write $\mathcal{W}_{d\in D_{2}\mapsto d^{2}\in D}$ for the homomorphim
of Weil algebras $\mathbb{R}[X]/(X^{2})\rightarrow\mathbb{R}[X]/(X^{3})$
induced by the homomorphism $X\rightarrow X^{2}$ of the polynomial ring
\ $\mathbb{R}[X]$ to itself. Such locutions are justifiable, because the
category $\mathbf{W}$ of Weil algebras in the real world and the category
$\mathbf{D}$ of infinitesimal objects in the shade are dual to each other in a
sense. Thus we have a contravariant functor $\mathcal{W}$\ from the category
of infinitesimal objects in the shade to the category of Weil algebras in the
real world. Its inverse contravariant functor from the category of Weil
algebras in the real world to the category of infinitesimal objects in the
shade is denoted by $\mathcal{D}$. By way of example, $\mathcal{D}%
_{\mathbb{R}[X]/(X^{2})}$ and $\mathcal{D}_{\mathbb{R}[X]/(X^{3})} $\ stand
for $D$ and $D_{2}$,\ respectively. Since the category $\mathbf{W}$\ is left
exact, the category $\mathbf{D}$\ is right exact, in which we write
$\mathbb{D\oplus D}^{\prime}$ for the coproduct of infinitesimal objects
$\mathbb{D}$ and $\mathbb{D}^{\prime}$. For any two infinitesimal objects
$\mathbb{D},\mathbb{D}^{\prime}$ with $\mathbb{D}\subseteq\mathbb{D}^{\prime}%
$, we write $i$ or $i_{\mathbb{D}\rightarrow\mathbb{D}^{\prime}} $ for its
natural injection of $\mathbb{D}$\ into $\mathbb{D}^{\prime}$. We write
$\mathbf{m}$ or $\mathbf{m}_{D_{n}\times D_{m}\rightarrow D_{n}}$ for the
mapping $\left(  d,d^{\prime}\right)  \in D_{n}\times D_{m}\mapsto dd^{\prime
}\in D_{n}$. Given $\alpha\in\mathbb{R} $, we write $\left(  \alpha
\underset{i}{\cdot}\right)  _{D^{n}}$ for the mapping
\[
(d_{1},...,d_{n})\in D^{n}\mapsto(d_{1},...d_{i-1},\alpha d_{i},d_{i+1}%
,...,d_{n})\in D^{n}%
\]
To familiarize himself or herself with such locutions, the reader is strongly
encouraged to read the first two chapters of \cite{la}, even if he or she is
not interested in SDG at all.

\subsubsection{\label{s2.2.2}Simplicial Infinitesimal Objects}

\begin{definition}
\label{d2.2.2.1}

\begin{enumerate}
\item \textit{Simplicial infinitesimal spaces} are objects of the form
\[
D\left\{  m;\mathcal{S}\right\}  =\{(d_{1},...,d_{m})\in D^{m}|d_{i_{1}%
}...d_{i_{k}}=0\text{ for any }(i_{1},...,i_{k})\in\mathcal{S}\}\text{,}%
\]
where $\mathcal{S}$ is a finite set of sequences $(i_{1},...,i_{k})$ of
natural numbers with $1\leq i_{1}<...<i_{k}\leq m$.

\item A simplicial infinitesimal object $D\left\{  m;\mathcal{S}\right\}  $ is
said to be \textit{symmetric }if $(d_{1},...,d_{m})\in D\left\{
m;\mathcal{S}\right\}  $ and $\sigma\in\mathbf{S}_{m}$ always imply
$(d_{\sigma(1)},...,d_{\sigma(m)})\in D\left\{  m;\mathcal{S}\right\}  $.
\end{enumerate}
\end{definition}

To give examples of simplicial infinitesimal spaces, we have
\begin{align*}
D(2)  &  =D\left\{  2;(1,2)\right\} \\
D(3)  &  =D\left\{  3;(1,2),(1,3),(2,3)\right\}  \text{,}%
\end{align*}
which are all symmetric.

\begin{definition}
\label{d2.2.2.2}

\begin{enumerate}
\item The number $m$ is called the \textit{degree} of $D\left\{
m;\mathcal{S}\right\}  $, in notation: $m=\mathrm{\deg\,}D\left\{
m;\mathcal{S}\right\}  $.

\item The maximum number $n$ such that there exists a sequence $(i_{1}%
,...,i_{n})$ of natural numbers of length $n$ with $1\leq i_{1}<...<i_{n}\leq
m$ containing no subsequence in $\mathcal{S}$ is called the \textit{dimension}
of $D\left\{  m;\mathcal{S}\right\}  $, in notation: $n=\mathrm{\dim
\,}D\left\{  m;\mathcal{S}\right\}  $.
\end{enumerate}
\end{definition}

By way of example, we have
\begin{align*}
\mathrm{\deg\,}D(3)  &  =\mathrm{\deg\,}D\left\{  3;(1,2)\right\}
=\mathrm{\deg\,}D\left\{  3;(1,2),(1,3)\right\}  =\mathrm{\deg\,}D^{3}=3\\
\mathrm{\dim\,}D(3)  &  =1\\
\mathrm{\dim\,}D\left\{  3;(1,2)\right\}   &  =\mathrm{\dim\,}D\left\{
3;(1,2),(1,3)\right\}  =2\\
\mathrm{\dim\,}D^{3}  &  =3
\end{align*}
It is easy to see that

\begin{proposition}
if $n=\mathrm{\dim\,}D\left\{  m;\mathcal{S}\right\}  $, then
\[
d_{1}+...+d_{m}\in D_{n}%
\]
for any $(d_{1},...,d_{m})\in D\left\{  m;\mathcal{S}\right\}  $, so that we
have the mapping
\[
+_{D\left\{  m;\mathcal{S}\right\}  \rightarrow D_{n}}:D\left\{
m;\mathcal{S}\right\}  \rightarrow D_{n}%
\]

\end{proposition}

\begin{definition}
\label{d2.2.2.3}Infinitesimal objects of the form $D^{m}$ are called
\textit{basic infinitesimal }objects.
\end{definition}

\begin{definition}
\label{d2.2.2.4}Given two simplicial infinitesimal objects $D\left\{
m;\mathcal{S}\right\}  $ and $D\left\{  m^{\prime};\mathcal{S}^{\prime
}\right\}  $, a mapping
\[
\varphi=(\varphi_{1},...,\varphi_{m^{\prime}}):D\left\{  m;\mathcal{S}%
\right\}  \rightarrow D\left\{  m^{\prime};\mathcal{S}^{\prime}\right\}
\]
is called a \textit{monomial mapping} if every $\varphi_{j}$ is a monomial in
$d_{1},...,d_{m}$ with coefficient $1$.
\end{definition}

\begin{notation}
We denote by $D\left\{  m\right\}  _{n}$ the infinitesimal object
\[
\{(d_{1},...,d_{m})\in D^{m}|d_{i_{1}}...d_{i_{n+1}}=0\}\text{,}%
\]
where $i_{1},...,i_{n+1}$ shall range over natural numbers between $1$ and $m
$ including both ends.
\end{notation}

\subsubsection{\label{s2.2.3}Quasi-Colimit Diagrams}

\begin{definition}
A diagram in the category $\mathbf{D}$\ is called a quasi-colimit diagram if
its dually corresponding diagram in the category $\mathbf{W}$\ is a limit diagram.
\end{definition}

\begin{theorem}
\label{t2.2.4.1}(The Fundamental Theorem on Simplicial Infinitesimal Objects)
Any simplicial infinitesimal object $\mathbb{D}$ of dimension $n$ is the
quasi-colimit of a finite diagram whose objects are of the form $D^{k} $'s
($0\leq k\leq n$) and whose arrows are natural injections.
\end{theorem}

\begin{proof}
Let $\mathbb{D}\mathcal{=}D(m;\mathcal{S})$. For any maximal sequence $1\leq
i_{1}<...<i_{k}\leq m$ of natural numbers containing no subsequence in
$\mathcal{S}$ (maximal in the sense that it is not a proper subsequence of
such a sequence), we have a natural injection of $D^{k}$ into $\mathbb{D}$. By
collecting all such $D^{k}$'s together with their natural injections into
$\mathbb{D}$, we have an overlapping representation of $\mathbb{D}$ in terms
of basic infinitesimal spaces. This representation is completed into a
quasi-colimit representation of $\mathbb{D}$ by taking $D^{l}$ together with
its natural injections into $D^{k_{1}}$ and $D^{k_{2}}$ for any two basic
infinitesimal spaces $D^{k_{1}}$ and $D^{k_{2}}$ in the overlapping
representation of $\mathbb{D}$, where if $D^{k_{1}}$ and $D^{k_{2}}$ come from
the sequences $1\leq i_{1}<...<i_{k_{1}}\leq m$ and $1\leq\overline{i}%
_{1}<...<\overline{i}_{k_{2}}\leq m$ in the above manner, then $D^{l}$
together with its natural injections into $D^{k_{1}}$ and $D^{k_{2}}$ comes
from the maximal common subsequence $1\leq\widetilde{i}_{1}<...<\widetilde
{i}_{l}\leq m$ of both the preceding sequences of natural numbers in the above
manner. By way of example, the above method leads to the following
quasi-colimit representation of $\mathbb{D}\mathcal{=}D\left\{  3\right\}
_{2} $:
\[%
\begin{array}
[c]{ccccc}
&  & D^{2} &  & \\
& i_{1}\nearrow &  & \nwarrow i_{2} & \\
D &  & \downarrow i_{12} &  & D\\
i_{1}\downarrow &  & D(3)_{2} &  & \downarrow i_{1}\\
D^{2} & i_{13}\nearrow &  & \nwarrow i_{23} & D^{2}\\
& i_{2}\nwarrow &  & \nearrow i_{2} & \\
&  & D &  &
\end{array}
\]
In the above representation $i_{jk}$'s and $i_{j}$'s are as follows:

\begin{enumerate}
\item the $j$-th and $k$-th components of $i_{jk}(d_{1},d_{2})\in D(3)_{2}$
are $d_{1}$ and $d_{2}$, respectively, while the remaining component is $0$;

\item the $j$-th component of $i_{j}(d)\in D^{2}$ is $d$, while the other
component is $0$.
\end{enumerate}
\end{proof}

\begin{definition}
The quasi-colimit representation of $\mathbb{D}$ depicted in the proof of the
above theorem is called \textit{standard}.
\end{definition}

\begin{remark}
Generally speaking, there are multiple ways of quasi-colimit representation of
a given simplicial infinitesimal space. By way of example, two quasi-colimit
representations of $D\left\{  3;(1,3),(2,3)\right\}  $ ($=(D\times D)\oplus
D$) were given in Lavendhomme \cite[pp.92-93]{la} (\S 3.4, pp.92-93), only the
second one being standard.
\end{remark}

\subsection{\label{s2.3}Weil-Exponentiability and Microlinearity}

\subsubsection{\label{s2.3.1}Weil-Exponentiability}

We have no reason to hold that all Fr\"{o}licher spaces credit Weil
prolongations as exponentiations by infinitesimal objects in the shade.
Therefore we need a notion which distinguishes Fr\"{o}licher spaces that do so
from those that do not.

\begin{definition}
A Fr\"{o}licher space $X$ is called \textit{Weil exponentiable }if
\begin{equation}
(X\otimes(W_{1}\otimes_{\infty}W_{2}))^{Y}=(X\otimes W_{1})^{Y}\otimes W_{2}
\label{2.3.1.1}%
\end{equation}
holds naturally for any Fr\"{o}licher space $Y$ and any Weil algebras $W_{1}$
and $W_{2}$.
\end{definition}

If $Y=1$, then (\ref{2.3.1.1}) degenerates into
\[
X\otimes(W_{1}\otimes_{\infty}W_{2})=(X\otimes W_{1})\otimes W_{2}%
\]
If $W_{1}=\mathbb{R}$, then (\ref{2.3.1.1}) degenerates into
\[
(X\otimes W_{2})^{Y}=X^{Y}\otimes W_{2}%
\]

The following three propositions have been established in our previous paper
\cite{ni7}.

\begin{proposition}
\label{t2.3.1.1}Convenient vector spaces are Weil exponentiable.
\end{proposition}

\begin{corollary}
$C^{\infty}$-manifolds in the sense of \cite{kri} (cf. Section 27) are Weil exponentiable.
\end{corollary}

\begin{proposition}
\label{t2.3.1.2}If $X$ is a Weil exponentiable Fr\"{o}licher space, then so is
$X\otimes W$ for any Weil algebra $W$.
\end{proposition}

\begin{proposition}
\label{t2.3.1.3}If $X$ and $Y$ are Weil exponentiable Fr\"{o}licher spaces,
then so is $X\times Y$.
\end{proposition}

The last proposition can be strengthened to

\begin{proposition}
\label{t2.3.1.4}The limit of a diagram in $\mathbf{FS}$\ whose objects are all
Weil-exponentiable is also Weil-exponentiable.
\end{proposition}

\begin{proof}
Let $\Gamma$\ be a diagram in $\mathbf{FS}$. Given a Weil algebra $W$, we
write $\Gamma\otimes W$ for the diagram obtained from $\Gamma$ by putting
$\otimes W$ to the right of every object in $\Gamma$\ and $\otimes
\mathrm{id}_{W}$ to the right of every morphism in $\Gamma$.We have
\begin{align*}
&  \left(  \left(  \mathrm{Lim\,}\Gamma\right)  \otimes\left(  W_{1}%
\otimes_{\infty}W_{2}\right)  \right)  ^{Y}\\
&  =\left(  \mathrm{Lim\,}\left(  \Gamma\otimes\left(  W_{1}\otimes_{\infty
}W_{2}\right)  \right)  \right)  ^{Y}\\
&  =\mathrm{Lim\,}\left(  \Gamma\otimes\left(  W_{1}\otimes_{\infty}%
W_{2}\right)  \right)  ^{Y}\\
&  =\mathrm{Lim\,}\left(  (\Gamma\otimes W_{1})^{Y}\otimes W_{2}\right) \\
&  =\left(  \mathrm{Lim\,}(\Gamma\otimes W_{1})^{Y}\right)  \otimes W_{2}\\
&  =\left(  \mathrm{Lim\,}\left(  \Gamma\otimes W_{1}\right)  \right)
^{Y}\otimes W_{2}\\
&  =\left(  \left(  \mathrm{Lim\,}\Gamma\right)  \otimes W_{1}\right)
^{Y}\otimes W_{2}%
\end{align*}
so that we have the coveted result.
\end{proof}

We have already established the following proposition and theorem in in our
previous paper \cite{ni7}.

\begin{proposition}
\label{t2.3.1.5}If $X$ is a Weil exponentiable Fr\"{o}licher space, then so is
$X^{Y}$ for any Fr\"{o}licher space $Y$.
\end{proposition}

\begin{theorem}
\label{t2.3.1.6}Weil exponentiable Fr\"{o}licher spaces, together with smooth
mappings among them, form a Cartesian closed subcategory $\mathbf{FS}%
_{\mathbf{WE}}$\ of the category $\mathbf{FS}$.
\end{theorem}

\subsubsection{\label{s2.3.2}Microlinearity}

The central object of study in SDG is \textit{microlinear} spaces. Although
the notion of a manifold (=a pasting of copies of a certain linear space) is
defined on the local level, the notion of microlinearity is defined on the
genuinely infinitesimal level. For the historical account of microlinearity,
the reader is referred to \S \S 2.4 of \cite{la} or Appendix D of \cite{ko}.
To get an adequately restricted cartesian closed subcategory of Fr\"{o}licher
spaces, we have emancipated microlinearity from within a well-adapted model of
SDG to Fr\"{o}licher spaces in the real world in \cite{ni8}. Recall that

\begin{definition}
A Fr\"{o}licher space $X$ is called \textit{microlinear} providing that any
finite limit diagram $\Gamma$ in $\mathbf{W}$ yields a limit diagram
$X\otimes\Gamma$ in $\mathbf{FS}$, where $X\otimes\Gamma$ is obtained from
$\Gamma$ by putting $X\otimes$ to the left of every object in $\Gamma$\ and
$\mathrm{id}_{X}\otimes$ to the left of every morphism in $\Gamma$.
\end{definition}

Generally speaking, limits in the category $\mathbf{FS}$ are bamboozling. The
notion of limit in $\mathbf{FS}$ should be elaborated geometrically.

\begin{definition}
A finite cone $\Gamma$ in $\mathbf{FS}$ is called a \textit{transversal limit}
\textit{diagram} providing that $\Gamma\otimes W$ is a limit diagram in
$\mathbf{FS}$ for any Weil algebra $W$, where the diagram $\Gamma\otimes W$ is
obtained from $\Gamma$ by putting $\otimes W$ to the right of every object in
$\Gamma$\ and $\otimes\mathrm{id}_{W}$ to the right of every morphism in
$\Gamma$. The limit of a finite diagram of Fr\"{o}licher spaces is said to be
\textit{transversal} providing that its limit diagram is a transversal limit diagram.
\end{definition}

\begin{remark}
By taking $W=\mathbb{R}$, we see that a transversal limit diagram in
$\mathbf{FS}$\ is always a limit diagram in $\mathbf{FS}$.
\end{remark}

We have already established the following two propositions in \ref{ni8}.

\begin{proposition}
\label{t2.3.2.1}If $\Gamma$ is a transversal limit diagram in $\mathbf{FS}%
$\ whose objects are all Weil exponentiable, then $\Gamma^{X}$ is also a
transversal limit diagram for any Fr\"{o}licher space $X$, where $\Gamma^{X} $
is obtained from $\Gamma$ by putting $X$ as the exponential over every object
in $\Gamma$\ and over every morphism in $\Gamma$.
\end{proposition}

\begin{proposition}
\label{t2.3.2.2}If $\Gamma$ is a transversal limit diagram in $\mathbf{FS}%
$\ whose objects are all Weil exponentiable, then $\Gamma\otimes W$ is also a
transversal limit diagram for any Weil algebra $W$.
\end{proposition}

The following results have been established in \cite{ni8}.

\begin{proposition}
\label{t2.3.2.3}Convenient vector spaces are microlinear.
\end{proposition}

\begin{corollary}
$C^{\infty}$-manifolds in the sense of \cite{kri} (cf. Section 27) are microlinear.
\end{corollary}

\begin{proposition}
\label{t2.3.2.4}If $X$ is a Weil exponentiable and microlinear Fr\"{o}licher
space, then so is $X\otimes W$ for any Weil algebra $W$.
\end{proposition}

\begin{proposition}
\label{t2.3.2.5}The class of microlinear Fr\"{o}licher spaces is closed under
transversal limits.
\end{proposition}

\begin{corollary}
Direct products are transversal limits, so that if $X$ and $Y$ are microlinear
Fr\"{o}licher spaces, then so is $X\times Y$.
\end{corollary}

\begin{proposition}
\label{t2.3.2.6}If $X$ is a Weil exponentiable and microlinear Fr\"{o}licher
space, then so is $X^{Y}$ for any Fr\"{o}licher space $Y$.
\end{proposition}

\begin{proposition}
\label{t2.3.2.7}If a Weil exponentiable Fr\"{o}licher space $X$ is
microlinear, then any finite limit diagram $\Gamma$ in $\mathbf{W}$ yields a
transversal limit diagram $X\otimes\Gamma$ in $\mathbf{FS}$.
\end{proposition}

\begin{theorem}
\label{t2.3.2.8}Weil exponentiable and microlinear Fr\"{o}licher spaces,
together with smooth mappings among them, form a cartesian closed subcategory
$\mathbf{FS}_{\mathbf{WE,ML}}$\ of the category $\mathbf{FS}$.
\end{theorem}

\subsection{\textbf{Convention}}

Unless stated to the contrary, every Fr\"{o}licher space occurring in the
sequel is assumed to be microlinear and Weil exponentiable. We will fix a
smooth mapping $\pi:E\rightarrow M$ arbitrarily. In this paper we will naively
speak of \textit{bundles} simply as smooth mappings of microlinear and Weil
exponentiable Fr\"{o}licher spaces, for which we will develop three theories
of jet bundles. We say that $t\in M\otimes\mathcal{W}_{D}$ is
\textit{degenerate} providing that
\[
t=\left(  i_{\left\{  x\right\}  \rightarrow M}\otimes\mathrm{id}%
_{\mathcal{W}_{D}}\right)  \left(  t^{\prime}\right)
\]
for some $x\in M$ and some $t^{\prime}\in\left\{  x\right\}  \otimes
\mathcal{W}_{D}$. We say that $t\in E\otimes\mathcal{W}_{D}$ is
\textit{vertical} provided that $\left(  \pi\otimes\mathrm{id}_{\mathcal{W}%
_{D}}\right)  \left(  t\right)  $ is degenerate. We write $\left(
E\otimes\mathcal{W}_{D}\right)  ^{\perp}$ for the totality of vertical $t\in
E\otimes\mathcal{W}_{D}$.

\section{\label{s3}The First Approach to Jets}

\begin{definition}
\label{d3.1}A $1$\textit{-tangential} \textit{over} the bundle $\pi
:E\rightarrow M$ \textit{at} $x\in E$ is a mapping $\nabla_{x}:\left(
M\otimes\mathcal{W}_{D}\right)  _{\pi(x)}\rightarrow\left(  E\otimes
\mathcal{W}_{D}\right)  _{x}$ subject to the following three conditions:

\begin{enumerate}
\item We have
\[
\left(  \pi\otimes\mathrm{id}_{\mathcal{W}_{D}}\right)  \left(  \nabla
_{x}(t)\right)  =t
\]
for any $t\in\left(  M\otimes\mathcal{W}_{D}\right)  _{\pi(x)}$.

\item We have
\[
\nabla_{x}(\alpha t)=\alpha\nabla_{x}(t)
\]
for any $t\in\left(  M\otimes\mathcal{W}_{D}\right)  _{\pi(x)}$ and any
$\alpha\in\mathbb{R}$.

\item The diagram
\[%
\begin{array}
[c]{ccc}%
\left(  M\otimes\mathcal{W}_{D}\right)  _{\pi\left(  x\right)  } &
\underrightarrow{\mathrm{id}_{M}\otimes\mathcal{W}_{\left(  d,e\right)  \in
D\times D_{m}\mapsto ed\in D}} & \left(  M\otimes\mathcal{W}_{D}\right)
_{\pi\left(  x\right)  }\otimes\mathcal{W}_{D_{m}}\\
\nabla_{x}\downarrow &  & \downarrow\nabla_{x}\otimes\mathrm{id}%
_{\mathcal{W}_{D_{m}}}\\
\left(  E\otimes\mathcal{W}_{D}\right)  _{x} & \overrightarrow{\mathrm{id}%
_{E}\otimes\mathcal{W}_{\left(  d,e\right)  \in D\times D_{m}\mapsto ed\in D}}
& \left(  E\otimes\mathcal{W}_{D}\right)  _{x}\otimes\mathcal{W}_{D_{m}}%
\end{array}
\]
is commutative, where $m$\ is an arbitrary natural number.
\end{enumerate}
\end{definition}

We note in passing that condition (1.2) implies that $\nabla_{x}$ is linear by
dint of Proposition 10 in \S 1.2 of \cite{la}.

\begin{notation}
\label{n3.1}We denote by $\mathbf{J}_{x}^{1}(\pi)$ the totality of
$1$\textit{-tangentials} $\nabla_{x}$ over the bundle $\pi:E\rightarrow M$ at
$x\in E$. We denote by $\mathbf{J}^{1}(\pi)$ the set-theoretic union of
$\mathbf{J}_{x}^{1}(\pi)$'s for all $x\in E$. The canonical projection
$\mathbf{J}^{1}(\pi)\rightarrow E$ is denoted by $\pi_{1,0}$ with
\[
\pi_{1}=\left(  \pi\otimes\mathrm{id}_{\mathcal{W}_{D}}\right)  \circ\pi
_{1,0}\text{.}%
\]

\end{notation}

\begin{definition}
\label{d3.2}Let $F$ be a morphism of bundles over $M$ from $\pi$ to
$\pi^{\prime}$ over the same base space $M$. We say that a $1$%
\textit{-tangential} $\nabla_{x}$ over $\pi$ at a point $x$ of $E$%
\textrm{\ }is $F$-related \textit{to} a $1$\textit{-tangential} $\nabla
_{F(x)}$ over $\pi^{\prime}$ at $F(x)$ of $E^{\prime}$ (\textit{in}
\textit{the} \textit{sense} \textit{of} \textit{Nishimura}) provided that
\[
\left(  F\otimes\mathrm{id}_{\mathcal{W}_{D}}\right)  \left(  \nabla
_{x}(t)\right)  =\nabla_{F(x)}(t)
\]
for any $t\in\left(  M\otimes\mathcal{W}_{D}\right)  _{\pi(x)}$.
\end{definition}

\begin{notation}
\label{n3.2}By convention, we let
\[
\tilde{\mathbf{J}}^{0}(\pi)=\hat{\mathbf{J}}^{0}(\pi)=\mathbf{J}^{0}(\pi)=E
\]
with
\[
\tilde{\pi}_{0,0}=\hat{\pi}_{0,0}=\pi_{0,0}=\mathrm{id}_{E}%
\]
and
\[
\tilde{\pi}_{0}=\hat{\pi}_{0}=\pi_{0}=\pi
\]
We let
\[
\tilde{\mathbf{J}}^{1}(\pi)=\hat{\mathbf{J}}^{1}(\pi)=\mathbf{J}^{1}(\pi)
\]
with
\[
\tilde{\pi}_{1,0}=\hat{\pi}_{1,0}=\pi_{1,0}%
\]
and
\[
\tilde{\pi}_{1}=\hat{\pi}_{1}=\pi_{1}%
\]

\end{notation}

\begin{notation}
\label{n3.3}Now we are going to define $\tilde{\mathbf{J}}^{k+1}(\pi)$,
$\hat{\mathbf{J}}^{k+1}(\pi)$ and $\mathbf{J}^{k+1}(\pi)$ together with
mappings $\tilde{\pi}_{k+1,k}:\tilde{\mathbf{J}}^{k+1}(\pi)\rightarrow
\tilde{\mathbf{J}}^{k}(\pi)$, $\hat{\pi}_{k+1,k}:\hat{\mathbf{J}}^{k+1}%
(\pi)\rightarrow\hat{\mathbf{J}}^{k}(\pi)$ and $\pi_{k+1,k}:\mathbf{J}%
^{k+1}(\pi)\rightarrow\mathbf{J}^{k}(\pi)$ by induction on $k\geq1$.
Intuitively speaking, these are intended for non-holonomic, semi-holonomic and
holonomic jet bundles in order. We let $\tilde{\pi}_{k+1}=\tilde{\pi}_{k}%
\circ\tilde{\pi}_{k+1,k}$ , $\hat{\pi}_{k+1}=\hat{\pi}_{k}\circ\hat{\pi
}_{k+1,k}$ and $\pi_{k+1}=\pi_{k}\circ\pi_{k+1,k}$.

\begin{enumerate}
\item First we deal with $\tilde{\mathbf{J}}^{k+1}(\pi)$, which is defined to
be $\mathbf{J}^{1}(\tilde{\pi}_{k})$ with $\tilde{\pi}_{k+1,k}=(\tilde{\pi
}_{k})_{1,0}$.

\item Next we deal with $\hat{\mathbf{J}}^{k+1}(\pi)$, which is defined to be
the subspace of $\mathbf{J}^{1}(\hat{\pi}_{k})$ consisting of $\nabla_{x} $'s
with $x=\nabla_{y}\in\hat{\mathbf{J}}^{k}(\pi)$ abiding by the condition that
$\nabla_{x}$ is $\hat{\pi}_{k,k-1}$-related to $\nabla_{y}$.

\item Finally we deal with $\mathbf{J}^{k+1}(\pi)$, which is defined to be the
subspace of $\mathbf{J}^{1}(\pi_{k})$ consisting of $\nabla_{x}$'s with
$x=\nabla_{y}\in\mathbf{J}^{k}(\pi)$ abiding by the conditions that
$\nabla_{x}$ is $\pi_{k,k-1}$-related to $\nabla_{y}$ and that the composition
of mappings
\begin{align*}
&  \left(  M\otimes\mathcal{W}_{D^{2}}\right)  _{\pi_{k}\left(  x\right)  }\\
&  \underrightarrow{\left\langle \mathrm{id}_{M}\otimes\mathcal{W}_{d\in
D\mapsto\left(  d,0\right)  \in D^{2}},\mathrm{id}_{M}\otimes\mathcal{W}%
_{\left(  d_{1}.d_{2}\right)  \in D^{2}\mapsto\left(  d_{2}.d_{1}\right)  \in
D^{2}}\right\rangle }\\
&  \left(  \left(  M\otimes\mathcal{W}_{D}\right)  \times_{M\otimes
\mathcal{W}_{D}}\left(  M\otimes\mathcal{W}_{D^{2}}\right)  \right)  _{\pi
_{k}\left(  x\right)  }\\
&  \underrightarrow{\nabla_{x}\times\mathrm{id}_{M\otimes\mathcal{W}_{D^{2}}}%
}\\
&  \left(  \left(  \mathbf{J}^{k}(\pi)\otimes\mathcal{W}_{D}\right)
\times_{M\otimes\mathcal{W}_{D}}\left(  M\otimes\mathcal{W}_{D^{2}}\right)
\right)  _{\pi_{k}\left(  x\right)  }\\
&  =\left(  \left(  \mathbf{J}^{k}(\pi)\otimes\mathcal{W}_{D}\right)
\times_{M\otimes\mathcal{W}_{D}}\left(  \left(  M\otimes\mathcal{W}%
_{D}\right)  \otimes\mathcal{W}_{D}\right)  \right)  _{\pi_{k}\left(
x\right)  }\\
&  =\left(  \left(  \mathbf{J}^{k}(\pi)\times_{M}\left(  M\otimes
\mathcal{W}_{D}\right)  \right)  \otimes\mathcal{W}_{D}\right)  _{\pi
_{k}\left(  x\right)  }\\
&  \underrightarrow{\left(  \left(  \nabla,t\right)  \in\mathbf{J}^{k}%
(\pi)\times_{M}\left(  M\otimes\mathcal{W}_{D}\right)  \mapsto\nabla\left(
t\right)  \in\left(  \mathbf{J}^{k-1}(\pi)\otimes\mathcal{W}_{D}\right)
\right)  \otimes\mathrm{id}_{\mathcal{W}_{D}}}\\
&  \left(  \mathbf{J}^{k-1}(\pi)\otimes\mathcal{W}_{D}\right)  \otimes
\mathcal{W}_{D}\\
&  =\mathbf{J}^{k-1}(\pi)\otimes\mathcal{W}_{D^{2}}%
\end{align*}
is equal to the composition of mappings
\begin{align*}
&  \left(  M\otimes\mathcal{W}_{D^{2}}\right)  _{\pi_{k}\left(  x\right)  }\\
&  \underrightarrow{\left\langle \mathrm{id}_{M}\otimes\mathcal{W}_{d\in
D\mapsto\left(  0,d\right)  \in D^{2}},\mathrm{id}_{M\otimes\mathcal{W}%
_{D^{2}}}\right\rangle }\\
&  \left(  \left(  M\otimes\mathcal{W}_{D}\right)  \times_{M\otimes
\mathcal{W}_{D}}\left(  M\otimes\mathcal{W}_{D^{2}}\right)  \right)  _{\pi
_{k}\left(  x\right)  }\\
&  \underrightarrow{\nabla_{x}\times\mathrm{id}_{M\otimes\mathcal{W}_{D^{2}}}%
}\\
&  \left(  \left(  \mathbf{J}^{k}(\pi)\otimes\mathcal{W}_{D}\right)
\times_{M\otimes\mathcal{W}_{D}}\left(  M\otimes\mathcal{W}_{D^{2}}\right)
\right)  _{\pi_{k}\left(  x\right)  }\\
&  =\left(  \left(  \mathbf{J}^{k}(\pi)\otimes\mathcal{W}_{D}\right)
\times_{M\otimes\mathcal{W}_{D}}\left(  \left(  M\otimes\mathcal{W}%
_{D}\right)  \otimes\mathcal{W}_{D}\right)  \right)  _{\pi_{k}\left(
x\right)  }\\
&  =\left(  \left(  \mathbf{J}^{k}(\pi)\times_{M}\left(  M\otimes
\mathcal{W}_{D}\right)  \right)  \otimes\mathcal{W}_{D}\right)  _{\pi
_{k}\left(  x\right)  }\\
&  \underrightarrow{\left(  \left(  \nabla,t\right)  \in\mathbf{J}^{k}%
(\pi)\times_{M}\left(  M\otimes\mathcal{W}_{D}\right)  \mapsto\nabla\left(
t\right)  \in\left(  \mathbf{J}^{k-1}(\pi)\otimes\mathcal{W}_{D}\right)
\right)  \otimes\mathrm{id}_{\mathcal{W}_{D}}}\\
&  \left(  \mathbf{J}^{k-1}(\pi)\otimes\mathcal{W}_{D}\right)  \otimes
\mathcal{W}_{D}\\
&  =\mathbf{J}^{k-1}(\pi)\otimes\mathcal{W}_{D^{2}}\\
&  \underrightarrow{\mathrm{id}_{\mathbf{J}^{k-1}(\pi)}\otimes\mathcal{W}%
_{\left(  d_{1},d_{2}\right)  \in D^{2}\mapsto\left(  d_{2},d_{1}\right)  \in
D^{2}}}\\
&  \mathbf{J}^{k-1}(\pi)\otimes\mathcal{W}_{D^{2}}%
\end{align*}

\end{enumerate}
\end{notation}

\begin{definition}
\label{d3.3}Elements of $\tilde{\mathbf{J}}^{n}(\pi)$\ are called
$n$-subtangentials, while elements of $\hat{\mathbf{J}}^{n}(\pi)$\ are called
$n$-quasitangentials. Elements of $\mathbf{J}^{n}(\pi)$\ are called $n$-tangentials.
\end{definition}

\section{\label{s4}The Second Approach to Jets}

\begin{definition}
\label{d4.1}Let $n$ be a natural number. A $D^{n}$-pseudotangential
\textit{over }the bundle $\pi:E\rightarrow M$ \textit{at} $x\in E$ is a
mapping $\nabla_{x}:\left(  M\otimes\mathcal{W}_{D^{n}}\right)  _{\pi
(x)}\rightarrow\left(  E\otimes\mathcal{W}_{D^{n}}\right)  _{x}\ $abiding by
the following conditions:

\begin{enumerate}
\item We have
\[
\left(  \pi\otimes\mathrm{id}_{\mathcal{W}_{D^{n}}}\right)  \left(  \nabla
_{x}(\gamma)\right)  =\gamma
\]
for any $\gamma\in\left(  M\otimes\mathcal{W}_{D^{n}}\right)  _{\pi(x)}$.

\item We have
\[
\nabla_{x}(\alpha\underset{i}{\cdot}\gamma)=\alpha\underset{i}{\cdot}%
\nabla_{x}(\gamma)\quad\quad(1\leq i\leq n)
\]
for any $\gamma\in\left(  M\otimes\mathcal{W}_{D^{n}}\right)  _{\pi(x)}$ and
any $\alpha\in\mathbb{R}$.

\item The diagram
\[%
\begin{array}
[c]{ccc}%
\left(  M\otimes\mathcal{W}_{D^{n}}\right)  _{\pi\left(  x\right)  } &
\rightarrow & \left(  M\otimes\mathcal{W}_{D^{n}}\right)  _{\pi\left(
x\right)  }\otimes\mathcal{W}_{D_{m}}\\
\nabla_{x}\downarrow &  & \downarrow\nabla_{x}\otimes\mathrm{id}%
_{\mathcal{W}_{D_{m}}}\\
\left(  E\otimes\mathcal{W}_{D^{n}}\right)  _{x} & \rightarrow & \left(
E\otimes\mathcal{W}_{D^{n}}\right)  _{x}\otimes\mathcal{W}_{D_{m}}%
\end{array}
\]
is commutative, where $m$\ is an arbitrary natural number, the upper
horizontal arrow is
\[
\mathrm{id}_{M}\otimes\mathcal{W}_{\left(  d_{1},...,d_{n},e\right)  \in
D^{n}\times D_{m}\mapsto\left(  d_{1},...,d_{i-1},ed_{i},d_{i+1}%
,...d_{n}\right)  \in D^{n}}\text{,}%
\]
and the lower horizontal arrow is
\[
\mathrm{id}_{E}\otimes\mathcal{W}_{\left(  d_{1},...,d_{n},e\right)  \in
D^{n}\times D_{m}\mapsto\left(  d_{1},...,d_{i-1},ed_{i},d_{i+1}%
,...d_{n}\right)  \in D^{n}}\text{.}%
\]

\item We have
\[
\nabla_{x}(\gamma^{\sigma})=(\nabla_{x}(\gamma))^{\sigma}%
\]
for any $\gamma\in\left(  M\otimes\mathcal{W}_{D^{n}}\right)  _{\pi(x)}$ and
for any $\sigma\in\mathbf{S}_{n}$.
\end{enumerate}
\end{definition}

\begin{remark}
The third condition in the above definition claims what is called
infinitesimal multilinearity, while the second claims what is authentic multilinearity.
\end{remark}

\begin{notation}
We denote by $\mathbb{\hat{J}}_{x}^{D^{n}}(\pi)$ the totality of $D^{n}%
$-pseudotangentials$\ \nabla_{x}\ $over the bundle $\pi:E\rightarrow M$ at
$x\in E$. We denote by $\mathbb{\hat{J}}^{D^{n}}(\pi)$ the set-theoretic union
of $\mathbb{\hat{J}}_{x}^{D^{n}}(\pi)$'s for all $x\in E$. In
particular,$\ \mathbb{\hat{J}}^{D^{0}}(\pi)=E$ by convention.
\end{notation}

\begin{lemma}
\label{t4.1.1}The diagram
\begin{align*}
&  E\otimes\mathcal{W}_{D^{n}}
\begin{array}
[c]{c}%
\underrightarrow{\mathrm{id}_{E}\otimes\mathcal{W}_{\left(  d_{1}%
,...,d_{n},d_{n+1}\right)  \in D^{n+1}\mapsto\left(  d_{1},...,d_{n}\right)
\in D^{n}}}\\
\,
\end{array}
E\otimes\mathcal{W}_{D^{n+1}}\\
&
\begin{array}
[c]{c}%
\underrightarrow{\qquad\qquad\;\;\;\;\;\;\;\mathrm{id}_{E\otimes
\mathcal{W}_{D^{n+1}}\qquad\qquad\qquad\;\;\;\;\;\;}}\\
\overrightarrow{\mathrm{id}_{E}\otimes\mathcal{W}_{\left(  d_{1}%
,...,d_{n},d_{n+1}\right)  \in D^{n+1}\mapsto\left(  d_{1},...,d_{n},0\right)
\in D^{n+1}}}%
\end{array}
E\otimes\mathcal{W}_{D^{n+1}}%
\end{align*}
is an equalizer.
\end{lemma}

\begin{proof}
It is well known that the diagram
\[
\mathcal{W}_{D^{n}}
\begin{array}
[c]{c}%
\underrightarrow{\mathcal{W}_{\left(  d_{1},...,d_{n},d_{n+1}\right)  \in
D^{n+1}\mapsto\left(  d_{1},...,d_{n}\right)  \in D^{n}}}\\
\,
\end{array}
\mathcal{W}_{D^{n+1}}
\begin{array}
[c]{c}%
\underrightarrow{\,\qquad\qquad\qquad\mathrm{id}_{\mathcal{W}_{D^{n+1}}%
\qquad\qquad\qquad\,\,\,}}\\
\overrightarrow{\mathcal{W}_{\left(  d_{1},...,d_{n},d_{n+1}\right)  \in
D^{n+1}\mapsto\left(  d_{1},...,d_{n},0\right)  \in D^{n+1}}}%
\end{array}
\mathcal{W}_{D^{n+1}}%
\]
is an equalizer in the category of Weil algebras, so that the desired result
follows from the microlinearity of $E$.
\end{proof}

\begin{corollary}
\label{t4.1.1.1}$\gamma\in E\otimes\mathcal{W}_{D^{n+1}}$ is in the equalizer
of
\[
E\otimes\mathcal{W}_{D^{n+1}}
\begin{array}
[c]{c}%
\underrightarrow{\qquad\qquad\;\;\;\;\;\;\;\mathrm{id}_{E\otimes
\mathcal{W}_{D^{n+1}}\qquad\qquad\qquad\;\;\;\;\;\;}}\\
\overrightarrow{\mathrm{id}_{E}\otimes\mathcal{W}_{\left(  d_{1}%
,...,d_{n},d_{n+1}\right)  \in D^{n+1}\mapsto\left(  d_{1},...,d_{n},0\right)
\in D^{n+1}}}%
\end{array}
E\otimes\mathcal{W}_{D^{n+1}}%
\]
iff
\[
\gamma=\left(  \mathbf{s}_{n+1}\circ\mathbf{d}_{n+1}\right)  (\gamma)
\]

\end{corollary}

\begin{proof}
This follows simply from
\[
\mathbf{s}_{n+1}\circ\mathbf{d}_{n+1}=\mathrm{id}_{E}\otimes\mathcal{W}%
_{\left(  d_{1},...,d_{n},d_{n+1}\right)  \in D^{n+1}\mapsto\left(
d_{1},...,d_{n},0\right)  \in D^{n+1}}%
\]

\end{proof}

\begin{proposition}
\label{t4.1.2}Let $\nabla_{x}$ be a $D^{n+1}$-pseudotangential over the bundle
$\pi:E\rightarrow M$ at $x\in E$. Let $\gamma\in\left(  M\otimes
\mathcal{W}_{D^{n}}\right)  _{\pi(x)}$.$\ $Then we have
\[
\nabla_{x}(\mathbf{s}_{n+1}(\gamma))=\left(  \mathrm{id}_{E}\otimes
\mathcal{W}_{\left(  d_{1},...,d_{n},d_{n+1}\right)  \in D^{n+1}\mapsto\left(
d_{1},...,d_{n},0\right)  \in D^{n+1}}\right)  \left(  \nabla_{x}%
(\mathbf{s}_{n+1}(\gamma))\right)
\]
so that
\[
\nabla_{x}(\mathbf{s}_{n+1}(\gamma))=\left(  \mathbf{s}_{n+1}\circ
\mathbf{d}_{n+1}\right)  \left(  \nabla_{x}(\mathbf{s}_{n+1}(\gamma))\right)
\]

\end{proposition}

\begin{proof}
For any $\alpha\in\mathbb{R}$, we have
\begin{align*}
&  \alpha\underset{n+1}{\cdot}\left(  \nabla_{x}(\mathbf{s}_{n+1}%
(\gamma))\right) \\
&  =\nabla_{x}(\alpha\underset{n+1}{\cdot}\left(  \mathbf{s}_{n+1}%
(\gamma)\right)  )\\
&  =\nabla_{x}(\mathbf{s}_{n+1}(\gamma))
\end{align*}
Therefore we have the desired result by letting $\alpha=0$ in the above calculation.
\end{proof}

\begin{corollary}
\label{t4.1.2.1}The assignment
\[
\gamma\in\left(  M\otimes\mathcal{W}_{D^{n}}\right)  _{\pi(x)}\longmapsto
\mathbf{d}_{n+1}\left(  \nabla_{x}(\mathbf{s}_{n+1}(\gamma))\right)
\in\left(  E\otimes\mathcal{W}_{D^{n}}\right)  _{x}%
\]
is an $n$-pseudotangential over the bundle $\pi:E\rightarrow M$ at $x$.
\end{corollary}

\begin{notation}
By this Corollary, we have canonical projections $\widehat{\pi}_{n+1,n}%
:\mathbb{\hat{J}}^{D^{n+1}}(\pi)\rightarrow\mathbb{\hat{J}}^{D^{n}}(\pi)$. By
assigning $\pi(x)\in M$ to each $n$-pseudotangential $\nabla_{x}\ $over the
bundle $\pi:E\rightarrow M$ at $x\in E$, we have the canonical projections
$\widehat{\pi}_{n}:\mathbb{\hat{J}}^{D^{n}}(\pi)\rightarrow$ $M$. Note that
$\widehat{\pi}_{n}\circ\widehat{\pi}_{n+1,n}=\widehat{\pi}_{n+1}$ For any
natural numbers $n$, $m$ with $m\leq n$, we define $\widehat{\pi}%
_{n,m}:\mathbb{\hat{J}}^{D^{n}}(\pi)\rightarrow\mathbb{\hat{J}}^{D^{m}}(\pi)$
to be $\widehat{\pi}_{m+1,m}\circ...\circ\widehat{\pi}_{n,n-1}$.
\end{notation}

Now we are going to show that

\begin{proposition}
\label{t4.1.3}Let $\nabla_{x}\in\mathbb{\hat{J}}^{D^{n+1}}(\pi)$. Then the
following diagrams are commutative:
\begin{align*}
&  \
\begin{array}
[c]{ccc}%
\quad\quad\left(  M\otimes\mathcal{W}_{D^{n+1}}\right)  _{\pi(x)} &
\underrightarrow{\ \ \ \ \ \ \ \ \ \ \ \ \nabla_{x}%
\ \ \ \ \ \ \ \ \ \ \ \ \ \ } & \left(  E\otimes\mathcal{W}_{D^{n+1}}\right)
_{x}\\
\mathbf{s}_{i}\quad\uparrow &  & \uparrow\quad\mathbf{s}_{i}\\
\quad\quad\left(  M\otimes\mathcal{W}_{D^{n}}\right)  _{\pi(x)} &
\overrightarrow{\ \ \ \ \ \ \ \widehat{\pi}_{n+1,n}(\nabla_{x})\ \ \ \ \ \ \ }
& \left(  E\otimes\mathcal{W}_{D^{n}}\right)  _{x}%
\end{array}
\\
&
\begin{array}
[c]{ccc}%
\quad\quad\left(  M\otimes\mathcal{W}_{D^{n+1}}\right)  _{\pi(x)} &
\underrightarrow{\ \ \ \ \ \ \ \ \ \ \ \ \nabla_{x}%
\ \ \ \ \ \ \ \ \ \ \ \ \ \ } & \left(  E\otimes\mathcal{W}_{D^{n+1}}\right)
_{x}\\
\mathbf{d}_{i}\quad\downarrow &  & \downarrow\quad\mathbf{d}_{i}\\
\quad\quad\left(  M\otimes\mathcal{W}_{D^{n}}\right)  _{\pi(x)} &
\overrightarrow{\ \ \ \ \ \ \ \widehat{\pi}_{n+1,n}(\nabla_{x})\ \ \ \ \ \ \ }
& \left(  E\otimes\mathcal{W}_{D^{n}}\right)  _{x}%
\end{array}
\end{align*}

\end{proposition}

\begin{proof}
By the very definition of $\widehat{\pi}_{n+1,n}$, we have
\[
\mathbf{s}_{n+1}(\widehat{\pi}_{n+1}(\nabla_{x})(\gamma))=\nabla
_{x}(\mathbf{s}_{n+1}(\gamma))
\]
for any $\gamma\in\left(  M\otimes\mathcal{W}_{D^{n}}\right)  _{\pi(x)}$. For
$i\neq n+1$, we have
\begin{align*}
&  \mathbf{s}_{i}(\widehat{\pi}_{n+1,n}(\nabla_{x})(\gamma))\\
&  =\left(  \left(  \mathbf{s}_{n+1}(\widehat{\pi}_{n+1,n}(\nabla_{x}%
)(\gamma))\right)  ^{<i,n+1>}\right)  ^{<i+1,i+2,...,n,n+1>}\\
&  =\left(  \left(  \nabla_{x}(\mathbf{s}_{n+1}(\gamma))\right)
^{<i,n+1>}\right)  ^{<i+1,i+2,...,n,n+1>}\\
&  =\left(  \nabla_{x}\left(  (\mathbf{s}_{n+1}(\gamma))^{<i,n+1>}\right)
\right)  ^{<i+1,i+2,...,n,n+1>}\\
&  =\nabla_{x}\left(  \left(  (\mathbf{s}_{n+1}(\gamma))^{<i,n+1>}\right)
^{<i+1,i+2,...,n,n+1>}\right) \\
&  =\nabla_{x}\left(  \mathbf{s}_{i}\left(  \gamma\right)  \right)
\end{align*}
Now we are going to show that
\[
\mathbf{d}_{i}(\nabla_{x}(\gamma))=(\widehat{\pi}_{n+1,n}(\nabla
_{x}))(\mathbf{d}_{i}(\gamma))
\]
for any $\gamma\in\left(  M\otimes\mathcal{W}_{D^{n+1}}\right)  _{\pi(x)}$.
First we deal with the case of $i=n+1$. We have
\begin{align*}
&  \mathbf{d}_{n+1}(\nabla_{x}(\gamma))\\
&  =\mathbf{d}_{n+1}(0\underset{n+1}{\cdot}\nabla_{x}(\gamma))\\
&  =\mathbf{d}_{n+1}(\nabla_{x}(0\underset{n+1}{\cdot}\gamma))\\
&  =\mathbf{d}_{n+1}(\nabla_{x}(\mathbf{s}_{n+1}(\mathbf{d}_{n+1}(\gamma))))\\
&  =(\widehat{\pi}_{n+1,n}(\nabla_{x}))(\mathbf{d}_{n+1}(\gamma))
\end{align*}
For $i\neq n+1$, we have
\begin{align*}
&  \mathbf{d}_{i}(\nabla_{x}(\gamma))\\
&  =\left(  \mathbf{d}_{n+1}\left(  (\nabla_{x}(\gamma))^{<i,n+1>}\right)
\right)  ^{<n,n-1,...,i+1,i>}\\
&  =\left(  \mathbf{d}_{n+1}(\nabla_{x}(\gamma^{<i,n+1>}))\right)
^{<n,n-1,...,i+1,i>}\\
&  =(((\widehat{\pi}_{n+1,n}(\nabla_{x}))\left(  \mathbf{d}_{n+1}%
(\gamma^{<i,n+1>}))\right)  ^{<n,n-1,...,i+1,i>}\\
&  =(\widehat{\pi}_{n+1,n}(\nabla_{x}))\left(  \left(  \mathbf{d}_{n+1}%
(\gamma^{<i,n+1>})\right)  ^{<n,n-1,...,i+1,i>}\right) \\
&  =(\widehat{\pi}_{n+1,n}(\nabla_{x}))\left(  \mathbf{d}_{i}(\gamma)\right)
\end{align*}
Thus we are done through.
\end{proof}

\begin{corollary}
Let $\nabla_{x}^{+}$, $\nabla_{x}^{-}\in\mathbb{\hat{J}}^{D^{n+1}}(\pi)$ with
\[
\widehat{\pi}_{n+1,n}(\nabla_{x}^{+})=\widehat{\pi}_{n+1,n}(\nabla_{x}^{-})
\]
Then
\[
\left(  \mathrm{id}_{E}\otimes\mathcal{W}_{i_{D\left\{  n+1\right\}
_{n}\rightarrow D^{n+1}}}\right)  \left(  \nabla_{x}^{+}(\gamma)\right)
=\left(  \mathrm{id}_{E}\otimes\mathcal{W}_{i_{D\left\{  n+1\right\}
_{n}\rightarrow D^{n+1}}}\right)  \left(  \nabla_{x}^{-}(\gamma)\right)
\]
for any $\gamma\in\left(  M\otimes\mathcal{W}_{D^{n+1}}\right)  _{\pi(x)}$.
\end{corollary}

\begin{definition}
\label{d4.2}The notion of a $D^{n}$-tangential \textit{over} the bundle
$\pi:E\rightarrow M$ \textit{at} $x$ is defined by induction on $n$. The
notion of a $D$-tangential \textit{over} the bundle $\pi:E\rightarrow M$
\textit{at} $x$ shall be identical with that of a $D$-pseudotangential
\textit{over} the bundle $\pi:E\rightarrow M$ \textit{at} $x$ . Now we proceed
inductively. A $D^{n+1}$-pseudotangential
\[
\nabla_{x}:\left(  M\otimes\mathcal{W}_{D^{n+1}}\right)  _{\pi(x)}%
\rightarrow\left(  E\otimes\mathcal{W}_{D^{n+1}}\right)  _{x}%
\]
$\ $ over the bundle $\pi:E\rightarrow M$ at $x\in E$ is called a $D^{n+1}%
$-tangential \textit{over} the bundle $\pi:E\rightarrow M$ \textit{at} $x$ if
it acquiesces in the following two conditions:

\begin{enumerate}
\item $\widehat{\pi}_{n+1,n}(\nabla_{x})$ is a $D^{n}$-tangential
\textit{over} the bundle $\pi:E\rightarrow M$ \textit{at} $x$.

\item For any $\gamma\in\left(  M\otimes\mathcal{W}_{D^{n}}\right)  _{\pi(x)}%
$, we have
\begin{align*}
&  \nabla_{x}\left(  \left(  \mathrm{id}_{M}\otimes\mathcal{W}_{\left(
d_{1},...,d_{n},d_{n+1}\right)  \in D^{n+1}\mapsto\left(  d_{1},...,d_{n}%
d_{n+1}\right)  \in D^{n}}\right)  \left(  \gamma\right)  \right) \\
&  =\left(  \mathrm{id}_{E}\otimes\mathcal{W}_{\left(  d_{1},...,d_{n}%
,d_{n+1}\right)  \in D^{n+1}\mapsto\left(  d_{1},...,d_{n}d_{n+1}\right)  \in
D^{n+1}}\right)  \left(  \left(  \widehat{\pi}_{n+1,n}(\nabla_{x})\right)
(\gamma)\right)
\end{align*}

\end{enumerate}
\end{definition}

\begin{notation}
We denote by $\mathbb{J}_{x}^{D^{n}}(\pi)$ the totality of $D^{n}%
$-tangentials$\ \nabla_{x}\ $over the bundle $\pi:E\rightarrow M$ at $x\in E
$. We denote by $\mathbb{J}^{D^{n}}(\pi)$ the set-theoretic union of
$\mathbb{J}_{x}^{D^{n}}(\pi)$'s for all $x\in E$. In particular,
$\mathbb{J}^{D^{0}}(\pi)=\mathbb{\hat{J}}^{D^{0}}(\pi)=E$ by convention and
$\mathbb{J}^{D}(\pi)=\mathbb{\hat{J}}^{D}(\pi)$ by definition. By the very
definition of $D^{n}$-tangential, the projections $\widehat{\pi}%
_{n+1,n}:\mathbb{\hat{J}}^{D^{n+1}}(\pi)\rightarrow\mathbb{\hat{J}}^{D^{n}%
}(\pi)$ are naturally restricted to mappings $\pi_{n+1,n}:\mathbb{J}^{D^{n+1}%
}(\pi)\rightarrow\mathbb{J}^{D^{n}}(\pi)$. Similarly for $\pi_{n}%
:\mathbb{J}^{D^{n}}(\pi)\rightarrow M$ and $\pi_{n,m}:\mathbb{J}^{D^{n}}%
(\pi)\rightarrow\mathbb{J}^{D^{m}}(\pi)$ with $m\leq n$.
\end{notation}

It is easy to see that

\begin{proposition}
\label{t4.1.4}Let $m,n$ be natural numbers with $m\leq n$. Let $k_{1}%
,...,k_{m}$ be positive integers with $k_{1}+...+k_{m}=n$. For any $\nabla
_{x}\in\mathbb{J}^{D^{n}}(\pi)$, any $\gamma\in\left(  M\otimes\mathcal{W}%
_{D^{m}}\right)  _{\pi(x)}$ and any $\sigma\in\mathbf{S}_{n}$, we have
\begin{align*}
&  \nabla_{x}\left(  \left(  \mathrm{id}_{M}\otimes\mathcal{W}_{(d_{1}%
,...,d_{n})\in D^{n}\mapsto\left(  d_{\sigma(1)}...d_{\sigma(k_{1})}%
,d_{\sigma(k_{1}+1)}...d_{\sigma(k_{1}+k_{2})},...,d_{\sigma(k_{1}%
+...+k_{m-1}+1)}...d_{\sigma(n)}\right)  }\right)  \left(  \gamma\right)
\right) \\
&  =\left(  \mathrm{id}_{E}\otimes\mathcal{W}_{(d_{1},...,d_{n})\in
D^{n}\mapsto\left(  d_{\sigma(1)}...d_{\sigma(k_{1})},d_{\sigma(k_{1}%
+1)}...d_{\sigma(k_{1}+k_{2})},...,d_{\sigma(k_{1}+...+k_{m-1}+1)}%
...d_{\sigma(n)}\right)  }\right)  \left(  \left(  \pi_{n,m}(\nabla
_{x})\right)  (\gamma)\right)
\end{align*}

\end{proposition}

Interestingly enough, any $D^{n}$-pseudotangential naturally gives rise to
what might be called a $\mathbb{D}$-\textit{pseudotangential} for any
simplicial infinitesimal space $\mathbb{D}$ of dimension less than or equal to
$n$.

\begin{theorem}
\label{t4.1.5}Let $n$ be a natural number. Let $\mathbb{D}$ be a simplicial
infinitesimal space of dimension less than or equal to $n$. Any $D^{n}%
$-pseudotangential $\nabla_{x}$ \textit{over }the bundle $\pi:E\rightarrow M $
\textit{at} $x\in E$ naturally induces a mapping $\nabla_{x}^{\mathbb{D}%
}:\left(  M\otimes\mathcal{W}_{\mathbb{D}}\right)  _{\pi\left(  x\right)
}\rightarrow\left(  E\otimes\mathcal{W}_{\mathbb{D}}\right)  _{x}$ abiding by
the following three conditions:

\begin{enumerate}
\item We have
\[
\left(  \pi\otimes\mathrm{id}_{\mathcal{W}_{\mathbb{D}}}\right)  \left(
\nabla_{x}^{\mathbb{D}}(\gamma)\right)  =\gamma
\]
for any $\gamma\in\left(  M\otimes\mathcal{W}_{\mathbb{D}}\right)
_{\pi\left(  x\right)  }$.

\item We have
\[
\nabla_{x}^{\mathbb{D}}(\alpha\underset{i}{\cdot}\gamma)=\alpha\underset
{i}{\cdot}\left(  \nabla_{x}^{\mathbb{D}}(\gamma)\right)
\]
for any $\alpha\in\mathbb{R}$ and any $\gamma\in\left(  M\otimes
\mathcal{W}_{\mathbb{D}}\right)  _{\pi\left(  x\right)  }$, where $i$ is a
natural number with $1\leq i\leq\mathrm{\deg}\,\mathbb{D}$.

\item The diagram
\[%
\begin{array}
[c]{ccc}%
\left(  M\otimes\mathcal{W}_{\mathbb{D}}\right)  _{\pi\left(  x\right)  } &
\rightarrow & \left(  M\otimes\mathcal{W}_{\mathbb{D}}\right)  _{\pi\left(
x\right)  }\otimes\mathcal{W}_{D_{m}}\\
\nabla_{x}\downarrow &  & \downarrow\nabla_{x}\otimes\mathrm{id}%
_{\mathcal{W}_{D_{m}}}\\
\left(  E\otimes\mathcal{W}_{\mathbb{D}}\right)  _{x} & \rightarrow & \left(
E\otimes\mathcal{W}_{\mathbb{D}}\right)  _{x}\otimes\mathcal{W}_{D_{m}}%
\end{array}
\]
is commutative, where $m$\ is an arbitrary natural number, the upper
horizontal arrow is
\[
\mathrm{id}_{M}\otimes\mathcal{W}_{\left(  d_{1},...,d_{k},e\right)
\in\mathbb{D}\times D_{m}\mapsto\left(  d_{1},...,d_{i-1},ed_{i}%
,d_{i+1},...d_{k}\right)  \in\mathbb{D}}\text{,}%
\]
and the lower horizontal arrow is
\[
\mathrm{id}_{E}\otimes\mathcal{W}_{\left(  d_{1},...,d_{k},e\right)
\in\mathbb{D}\times D_{m}\mapsto\left(  d_{1},...,d_{i-1},ed_{i}%
,d_{i+1},...d_{k}\right)  \in\mathbb{D}}%
\]
with $k=\mathrm{\deg}\,\mathbb{D}$ and $1\leq i\leq k$.
\end{enumerate}

If the simplicial infinitesimal space $\mathbb{D}$ is symmetric, the induced
mapping $\nabla_{x}^{\mathbb{D}}:\left(  M\otimes\mathcal{W}_{\mathbb{D}%
}\right)  _{\pi\left(  x\right)  }\rightarrow\left(  E\otimes\mathcal{W}%
_{\mathbb{D}}\right)  _{x}$ acquiesces in the following condition of symmetry
besides the above ones:

\begin{itemize}
\item We have
\[
\nabla_{x}^{\mathbb{D}}(\gamma^{\sigma})=(\nabla_{x}^{\mathbb{D}}%
(\gamma))^{\sigma}\text{ }%
\]
for any $\sigma\in\mathbf{S}_{k}$ and any $\gamma\in\left(  M\otimes
\mathcal{W}_{\mathbb{D}}\right)  _{\pi\left(  x\right)  }$.
\end{itemize}
\end{theorem}

\begin{proof}
For the sake of simplicity in description, we deal, by way of example, with
the case that $n=3$ and $\mathbb{D}=D\left\{  3\right\}  _{2}$, for which the
standard quasi-colimit representation was given in the proof of Theorem
\ref{t2.2.4.1}. Therefore, giving $\gamma\in\left(  M\otimes\mathcal{W}%
_{D\left\{  3\right\}  _{2}}\right)  _{\pi\left(  x\right)  }$ is equivalent
to giving $\gamma_{12},\gamma_{13},\gamma_{23}\in\left(  M\otimes
\mathcal{W}_{D^{2}}\right)  _{\pi\left(  x\right)  }$ with $\mathbf{d}%
_{2}(\gamma_{12})=\mathbf{d}_{2}(\gamma_{13})$, $\mathbf{d}_{1}(\gamma
_{12})=\mathbf{d}_{2}(\gamma_{23})$ and $\mathbf{d}_{1}(\gamma_{13}%
)=\mathbf{d}_{1}(\gamma_{23})$. By Proposition \ref{t4.1.3}, we have
\begin{align*}
\mathbf{d}_{2}(\widehat{\pi}_{3,2}\left(  \nabla_{x}\right)  (\gamma_{12}))
&  =\widehat{\pi}_{3,2}\left(  \nabla_{x}\right)  (\mathbf{d}_{2}(\gamma
_{12}))=\widehat{\pi}_{3,2}\left(  \nabla_{x}\right)  (\mathbf{d}_{2}%
(\gamma_{13}))=\mathbf{d}_{2}(\widehat{\pi}_{3,2}\left(  \nabla_{x}\right)
(\gamma_{13}))\\
\mathbf{d}_{1}(\widehat{\pi}_{3,2}\left(  \nabla_{x}\right)  (\gamma_{12}))
&  =\widehat{\pi}_{3,2}\left(  \nabla_{x}\right)  (\mathbf{d}_{1}(\gamma
_{12}))=\widehat{\pi}_{3,2}\left(  \nabla_{x}\right)  (\mathbf{d}_{2}%
(\gamma_{23}))=\mathbf{d}_{2}(\widehat{\pi}_{3,2}\left(  \nabla_{x}\right)
(\gamma_{23}))\\
\mathbf{d}_{1}(\widehat{\pi}_{3,2}\left(  \nabla_{x}\right)  (\gamma_{13}))
&  =\widehat{\pi}_{3,2}\left(  \nabla_{x}\right)  (\mathbf{d}_{1}(\gamma
_{13}))=\widehat{\pi}_{3,2}\left(  \nabla_{x}\right)  (\mathbf{d}_{1}%
(\gamma_{23}))=\mathbf{d}_{1}(\widehat{\pi}_{3,2}\left(  \nabla_{x}\right)
(\gamma_{23}))\text{,}%
\end{align*}
which determines a unique $\nabla_{x}^{D\{3\}_{2}}(\gamma)\in\left(
E\otimes\mathcal{W}_{D\left\{  3\right\}  _{2}}\right)  _{x}$ with
\begin{align*}
\mathbf{d}_{1}(\nabla_{x}^{D\{3\}_{2}}(\gamma))  &  =\widehat{\pi}%
_{3,2}\left(  \nabla_{x}\right)  (\gamma_{23})\\
\mathbf{d}_{2}(\nabla_{x}^{D\{3\}_{2}}(\gamma))  &  =\widehat{\pi}%
_{3,2}\left(  \nabla_{x}\right)  (\gamma_{13})\\
\mathbf{d}_{3}(\nabla_{x}^{D\{3\}_{2}}(\gamma))  &  =\widehat{\pi}%
_{3,2}\left(  \nabla_{x}\right)  (\gamma_{12})\text{.}%
\end{align*}
The proof that $\nabla_{x}^{D\{3\}_{2}}:\left(  M\otimes\mathcal{W}_{D\left\{
3\right\}  _{2}}\right)  _{\pi\left(  x\right)  }\rightarrow\left(
E\otimes\mathcal{W}_{D\left\{  3\right\}  _{2}}\right)  _{x}$ acquiesces in
the desired four properties is safely left to the reader.
\end{proof}

\begin{remark}
The reader should note that the induced mapping $\nabla_{x}^{\mathbb{D}}$ is
defined in terms of the standard quasi-colimit representation of $\mathbb{D}$.
The concluding corollary of this subsection will show that the induced mapping
$\nabla_{x}^{\mathbb{D}}$ is independent of our choice of a quasi-colimit
representation of $\mathbb{D}$ to a large extent, whether it is standard or
not, as long as $\nabla$ is not only a $D^{n}$-pseudotangential but also a
$D^{n}$-\textit{tangential.} We note in passing that $\hat{\pi}_{n,m}(\nabla)$
with $m\leq n$ is no other than $\nabla_{x}^{D^{m}}$.
\end{remark}

\begin{proposition}
\label{t4.1.6}Let $\pi^{\prime}:P\rightarrow E$ be another bundle with $x\in
P$. If $\nabla_{\pi^{\prime}(x)}$ is a $n$-tangential$_{2}$ over\textit{\ }the
bundle $\pi:E\rightarrow M$ \textit{at} $\pi^{\prime}(x)\in E$ and $\nabla
_{x}$ is a $n$-tangential$_{2}$ over\textit{\ }the bundle $\pi^{\prime
}:P\rightarrow E$ \textit{at} $x\in E$, then the composition $\nabla_{x}%
\circ\nabla_{\pi^{\prime}(x)}$ is a $n$-tangential$_{2}$ \textit{over }the
bundle $\pi\circ\pi^{\prime}:P\rightarrow M$ \textit{at} $x\in E$, and
$\pi_{n,n-1}(\nabla_{x}\circ\nabla_{\pi^{\prime}(x)})=\pi_{n,n-1}(\nabla
_{x})\circ\pi_{n,n-1}(\nabla_{\pi^{\prime}(x)})$ provided that $n\geq1$.
\end{proposition}

\begin{proof}
In case of $n=0$, there is nothing to prove. It is easy to see that
if$\ \nabla_{\pi^{\prime}(x)}$ is a $n$-tangential$_{2}$ over\textit{\ }the
bundle $\pi:E\rightarrow M$ \textit{at} $\pi^{\prime}(x)\in E$ and $\nabla
_{x}$ is a $n$-tangential$_{2}$ over\textit{\ }the bundle $\pi^{\prime
}:P\rightarrow E$ \textit{at} $x\in E$, then the composition $\nabla_{x}%
\circ\nabla_{\pi^{\prime}(x)}$ is an $n$-pseudoconnection \textit{over }the
bundle $\pi:E\rightarrow M$ \textit{at} $x\in P$. If $\nabla_{\pi^{\prime}%
(x)}$ is a $\left(  n+1\right)  $-tangential$_{2}$ over\textit{\ }the bundle
$\pi:E\rightarrow M$ \textit{at} $\pi^{\prime}(x)\in E$ and $\nabla_{x}$ is a
$\left(  n+1\right)  $-tangential$_{2}$ over\textit{\ }the bundle $\pi
^{\prime}:P\rightarrow E$ \textit{at} $x\in P$, then we have
\begin{align*}
\pi_{n+1,n}(\nabla_{x}\circ\nabla_{\pi^{\prime}(x)})  &  =\mathbf{d}%
_{n+1}\circ\nabla_{x}\circ\nabla_{\pi^{\prime}(x)}\circ\mathbf{s}_{n+1}\\
&  =\mathbf{d}_{n+1}\circ\nabla_{x}\circ\mathbf{s}_{n+1}\circ\mathbf{d}%
_{n+1}\circ\nabla_{\pi^{\prime}(x)}\circ\mathbf{s}_{n+1}\text{ \ \ \ }\\
&  \text{[By Proposition \ref{t4.1.2}]}\\
&  =\pi_{n+1,n}(\nabla_{x})\circ\pi_{n+1,n}(\nabla_{\pi^{\prime}(x)})
\end{align*}
Therefore we have
\begin{align*}
&  \nabla_{x}\circ\nabla_{\pi^{\prime}(x)}(\left(  \mathrm{id}_{M}%
\otimes\mathcal{W}_{\left(  d_{1},...,d_{n},d_{n+1}\right)  \in D^{n+1}%
\mapsto\left(  d_{1},...,d_{n}d_{n+1}\right)  \in D^{n}}\right)  \left(
\gamma\right)  )\\
&  =\nabla_{x}\left(  \nabla_{\pi^{\prime}(x)}\left(  \left(  \mathrm{id}%
_{M}\otimes\mathcal{W}_{\left(  d_{1},...,d_{n},d_{n+1}\right)  \in
D^{n+1}\mapsto\left(  d_{1},...,d_{n}d_{n+1}\right)  \in D^{n}}\right)
\left(  \gamma\right)  \right)  \right) \\
&  =\nabla_{x}\left(  \left(  \mathrm{id}_{E}\otimes\mathcal{W}_{\left(
d_{1},...,d_{n},d_{n+1}\right)  \in D^{n+1}\mapsto\left(  d_{1},...,d_{n}%
d_{n+1}\right)  \in D^{n}}\right)  \left(  \pi_{n+1,n}(\nabla_{\pi^{\prime
}(x)})(\gamma)\right)  \right) \\
&  =\left(  \mathrm{id}_{P}\otimes\mathcal{W}_{\left(  d_{1},...,d_{n}%
,d_{n+1}\right)  \in D^{n+1}\mapsto\left(  d_{1},...,d_{n}d_{n+1}\right)  \in
D^{n}}\right)  \left(  \pi_{n+1,n}\left(  \nabla_{x}\right)  \left(
\pi_{n+1,n}(\nabla_{\pi^{\prime}(x)})(\gamma)\right)  \right) \\
&  =\left(  \mathrm{id}_{P}\otimes\mathcal{W}_{\left(  d_{1},...,d_{n}%
,d_{n+1}\right)  \in D^{n+1}\mapsto\left(  d_{1},...,d_{n}d_{n+1}\right)  \in
D^{n}}\right)  \left(  \pi_{n+1,n}(\nabla_{x}\circ\nabla_{\pi^{\prime}%
(x)})\right)
\end{align*}
\bigskip Thus we can prove by induction on $n$ that if $\nabla_{\pi^{\prime
}(x)}$ is a $n$-tangential$_{2}$ over\textit{\ }the bundle $\pi:E\rightarrow
M$ \textit{at} $\pi^{\prime}(x)\in E$ and $\nabla_{x}$ is a $n$%
-tangential$_{2}$ $over$\textit{\ }the bundle $\pi^{\prime}:P\rightarrow E$
\textit{at} $x\in E$, then the composition $\nabla_{x}\circ\nabla_{\pi
^{\prime}(x)}$ is a $n$-tangential$_{2}$ \textit{over }the bundle $\pi\circ
\pi^{\prime}:P\rightarrow M$ \textit{at} $x\in E$.
\end{proof}

\begin{theorem}
\label{t4.1.8}Let $\nabla$ be a $D^{n}$-tangential \textit{over }the bundle
$\pi:E\rightarrow M$ \textit{at} $x\in E$. Let $\mathbb{D}$ and $\mathbb{D}%
^{\prime}$ be simplicial infinitesimal spaces of dimension less than or equal
to $n$. Let $\chi$ be a monomial mapping from $\mathbb{D}$ to $\mathbb{D}%
^{\prime}$. Let $\gamma\in\mathbf{T}_{x}^{\mathbb{D}^{\prime}}(M)$. Then we
have
\[
\nabla_{\mathbb{D}}(\left(  \mathrm{id}_{M}\otimes\mathcal{W}_{\chi}\right)
\left(  \gamma\right)  )=\left(  \mathrm{id}_{E}\otimes\mathcal{W}_{\chi
}\right)  \left(  \nabla_{\mathbb{D}^{\prime}}(\gamma)\right)
\]

\end{theorem}

\begin{remark}
The reader should note that the above far-flung generalization of Proposition
\ref{t4.1.4} subsumes Proposition \ref{t4.1.3}.
\end{remark}

\begin{proof}
In place of giving a general proof with formidable notation, we satisfy
ourselves with an illustration. Here we deal only with the case that
$\mathbb{D}=D^{3}$, $\mathbb{D}^{\prime}=D(3)$ and $\chi$ is
\[
\chi(d_{1},d_{2},d_{3})=(d_{1}d_{2},d_{1}d_{3},d_{2}d_{3})
\]
for any $(d_{1},d_{2},d_{3})\in D^{3}$. We assume that $n\geq3$. We note first
that the monomial mapping $\chi:D^{3}\rightarrow D(3)$ is the composition of
two monomial mappings
\begin{align*}
\chi_{1}  &  :D^{3}\rightarrow D\left\{  6;(1,2),(3,4),(5,6)\right\} \\
\chi_{2}  &  :D\left\{  6;(1,2),(3,4),(5,6)\right\}  \rightarrow D(3)
\end{align*}
with
\[
\chi_{1}(d_{1},d_{2},d_{3})=(d_{1},d_{1},d_{2},d_{2},d_{3},d_{3})
\]
for any $(d_{1},d_{2},d_{3})\in D^{3}$ and
\[
\chi_{2}(d_{1},d_{2},d_{3},d_{4},d_{5},d_{6})=(d_{1}d_{3},d_{2}d_{5}%
,d_{4}d_{6})
\]
for any $(d_{1},d_{2},d_{3},d_{4},d_{5},d_{6})\in D\left\{
6;(1,2),(3,4),(5,6)\right\}  $, while the former monomial mapping $\chi
_{1}:D^{3}\rightarrow D\left\{  6;(1,2),(3,4),(5,6)\right\}  $ is in turn the
composition of three monomial mappings
\begin{align*}
\chi_{1}^{1}  &  :D^{3}\rightarrow D\left\{  4;(1,2)\right\} \\
\chi_{1}^{2}  &  :D\left\{  4;(1,2)\right\}  \rightarrow D\left\{
5;(1,2),(3,4)\right\} \\
\chi_{1}^{3}  &  :D\left\{  5;(1,2),(3,4)\right\}  \rightarrow D\left\{
6;(1,2),(3,4),(5,6)\right\}
\end{align*}
with
\[
\chi_{1}^{1}(d_{1},d_{2},d_{3})=(d_{1},d_{1},d_{2},d_{3})
\]
for any $(d_{1},d_{2},d_{3})\in D^{3}$,
\[
\chi_{1}^{2}(d_{1},d_{2},d_{3},d_{4})=(d_{1},d_{2},d_{3},d_{3},d_{4})
\]
for any $(d_{1},d_{2},d_{3},d_{4})\in D\left\{  4;(1,2)\right\}  $ and
\[
\chi_{1}^{3}(d_{1},d_{2},d_{3},d_{4},d_{5})=(d_{1},d_{2},d_{3},d_{4}%
,d_{5},d_{5})
\]
for any $(d_{1},d_{2},d_{3},d_{4},d_{5})\in D\left\{  5;(1,2),(3,4)\right\}
$. Therefore it suffices to prove that
\begin{equation}
\nabla\left(  \left(  \mathrm{id}_{M}\otimes\mathcal{W}_{\chi_{1}^{1}}\right)
\left(  \gamma^{\prime}\right)  \right)  =\left(  \mathrm{id}_{E}%
\otimes\mathcal{W}_{\chi_{1}^{1}}\right)  \left(  \nabla_{D\left\{
4;(1,2)\right\}  }(\gamma^{\prime})\right)  \label{4.1.8.1}%
\end{equation}
for any $\gamma^{\prime}\in\left(  M\otimes\mathcal{W}_{D\left\{
4;(1,2)\right\}  }\right)  _{\pi\left(  x\right)  }$, that
\begin{equation}
\nabla_{D\left\{  4;(1,2)\right\}  }\left(  \left(  \mathrm{id}_{M}%
\otimes\mathcal{W}_{\chi_{1}^{2}}\right)  \left(  \gamma^{\prime\prime
}\right)  \right)  =\left(  \mathrm{id}_{E}\otimes\mathcal{W}_{\chi_{1}^{2}%
}\right)  \left(  \nabla_{D\left\{  5;(1,2),(3,4)\right\}  }(\gamma
^{\prime\prime})\right)  \label{4.1.8.2}%
\end{equation}
for any $\gamma^{\prime\prime}\in\left(  M\otimes\mathcal{W}_{D\left\{
5;(1,2),(3,4)\right\}  }\right)  _{\pi\left(  x\right)  }$, that
\begin{equation}
\nabla_{D\left\{  5;(1,2),(3,4)\right\}  }\left(  \left(  \mathrm{id}%
_{M}\otimes\mathcal{W}_{\chi_{1}^{3}}\right)  \left(  \gamma^{\prime
\prime\prime}\right)  \right)  =\left(  \mathrm{id}_{E}\otimes\mathcal{W}%
_{\chi_{1}^{3}}\right)  \left(  \nabla_{D\left\{  6;(1,2),(3,4),(5,6)\right\}
}(\gamma^{\prime\prime\prime})\right)  \label{4.1.8.3}%
\end{equation}
for any $\gamma^{\prime\prime\prime}\in\left(  M\otimes\mathcal{W}_{D\left\{
6;(1,2),(3,4),(5,6)\right\}  }\right)  _{\pi\left(  x\right)  }$, and that
\begin{equation}
\nabla_{D\left\{  6;(1,2),(3,4),(5,6)\right\}  }(\left(  \mathrm{id}%
_{M}\otimes\mathcal{W}_{\chi_{2}}\right)  \left(  \gamma^{\prime\prime
\prime\prime}\right)  )=\left(  \mathrm{id}_{E}\otimes\mathcal{W}_{\chi_{2}%
}\right)  \left(  \nabla_{D(3)}(\gamma^{\prime\prime\prime\prime})\right)
\label{4.1.8.4}%
\end{equation}
for any $\gamma^{\prime\prime\prime\prime}\in\left(  M\otimes\mathcal{W}%
_{D(3)}\right)  _{\pi\left(  x\right)  }\mathbf{T}_{x}^{D(3)}(M)$. Since
$D\left\{  4;(1,2)\right\}  =D(2)\times D^{2}$, it is easy to see that
\[
\nabla\left(  \left(  \mathrm{id}_{M}\otimes\mathcal{W}_{\chi_{1}^{1}}\right)
\left(  \gamma^{\prime}\right)  \right)  =\nabla(\gamma_{1}^{\prime}%
\underset{1}{+}\gamma_{2}^{\prime})=\nabla(\gamma_{1}^{\prime})+\nabla
(\gamma_{2}^{\prime})
\]
where $\gamma_{1}^{\prime}=\gamma^{\prime}\circ(i_{1}\times\mathrm{id}_{D^{2}%
})$ and $\gamma_{2}^{\prime}=\gamma^{\prime}\circ(i_{2}\times\mathrm{id}%
_{D^{2}})$ with $i_{1}(d)=(d,0)\in D(2)$ and $i_{2}(d)=(0,d)\in D(2)$ for any
$d\in D $. On the other hand, we have
\[
\left(  \mathrm{id}_{E}\otimes\mathcal{W}_{\chi_{1}^{1}}\right)  \left(
\nabla_{D(4;(1,2))}(\gamma^{\prime})\right)  =\left(  \mathrm{id}_{E}%
\otimes\mathcal{W}_{\chi_{1}^{1}}\right)  \left(  \mathbf{l}_{(\nabla
(\gamma_{1}^{\prime}),\nabla(\gamma_{2}^{\prime}))}\right)  =\nabla(\gamma
_{1}^{\prime})+\nabla(\gamma_{2}^{\prime})
\]
where $\mathbf{l}_{(\nabla(\gamma_{1}^{\prime}),\nabla(\gamma_{2}^{\prime}))}$
is the unique element of $E\otimes\mathcal{W}_{D(2)\times D^{2}}$ with
\[
\left(  \mathrm{id}_{E}\otimes\mathcal{W}_{i_{1}\times\mathrm{id}_{D^{2}}%
}\right)  \left(  \mathbf{l}_{(\nabla(\gamma_{1}^{\prime}),\nabla(\gamma
_{2}^{\prime}))}\right)  =\nabla(\gamma_{1}^{\prime})
\]
and
\[
\left(  \mathrm{id}_{E}\otimes\mathcal{W}_{i_{2}\times\mathrm{id}_{D^{2}}%
}\right)  \left(  \mathbf{l}_{(\nabla(\gamma_{1}^{\prime}),\nabla(\gamma
_{2}^{\prime}))}\right)  =\nabla(\gamma_{2}^{\prime})
\]
Thus we have established (\ref{4.1.8.1}). By the same token, we can establish
(\ref{4.1.8.2}) and (\ref{4.1.8.3}). In order to prove (\ref{4.1.8.4}), it
suffices to note that
\begin{align*}
&  \left(  \mathrm{id}_{E}\otimes\mathcal{W}_{i_{135}}\right)  \left(
\nabla_{D\left\{  6;(1,2),(3,4),(5,6)\right\}  }(\left(  \mathrm{id}%
_{M}\otimes\mathcal{W}_{\chi_{2}}\right)  \left(  \gamma^{\prime\prime
\prime\prime}\right)  )\right) \\
&  =\left(  \mathrm{id}_{E}\otimes\mathcal{W}_{\chi_{2}\circ i_{135}}\right)
\left(  \nabla_{D(3)}(\gamma^{\prime\prime\prime\prime})\right)
\end{align*}
together with the seven similar identities obtained from the above by
replacing $i_{135}$ by seven other $i_{jkl}:D^{3}\rightarrow D\left\{
6;(1,2),(3,4),(5,6)\right\}  $ in the standard quasi-colimit representation of
$D\left\{  6;(1,2),(3,4),(5,6)\right\}  $, where $i_{jkl}:D^{3}\rightarrow
D\left\{  6;(1,2),(3,4),(5,6)\right\}  $ ($1\leq j<k<l\leq6$) is a mapping
with $i_{jkl}(d_{1},d_{2},d_{3})=(...,\underset{j}{d_{1}},...,\underset
{k}{d_{2}},...,\underset{l}{d_{3}},...)$ \ ($d_{1}$, $d_{2}$ and $d_{3}$ are
inserted at the $j$-th, $k$-th and $l$-th positions respectively, while the
other components are fixed at $0$). Its proof goes as follows. Since
\begin{align*}
&  \left(  \mathrm{id}_{E}\otimes\mathcal{W}_{i_{135}}\right)  \left(
\nabla_{D\left\{  6;(1,2),(3,4),(5,6)\right\}  }(\left(  \mathrm{id}%
_{M}\otimes\mathcal{W}_{\chi_{2}}\right)  \left(  \gamma^{\prime\prime
\prime\prime}\right)  )\right) \\
&  =\nabla(\left(  \mathrm{id}_{M}\otimes\mathcal{W}_{\chi_{2}\circ i_{135}%
}\right)  \left(  \gamma^{\prime\prime\prime\prime}\right)  )\text{,}%
\end{align*}
it suffices to show that
\[
\nabla(\left(  \mathrm{id}_{M}\otimes\mathcal{W}_{\chi_{2}\circ i_{135}%
}\right)  \left(  \gamma^{\prime\prime\prime\prime}\right)  )=\left(
\mathrm{id}_{E}\otimes\mathcal{W}_{\chi_{2}\circ i_{135}}\right)
\nabla_{D(3)}(\gamma^{\prime\prime\prime\prime})
\]
However the last identity follows at once by simply observing that the mapping
$\chi_{2}\circ i_{135}:D^{3}\rightarrow D(3)$ is the mapping
\[
(d_{1},d_{2},d_{3})\in D^{3}\longmapsto(d_{1}d_{2},0,0)\in D(3)\text{,}%
\]
which is the successive composition of the following three mappings:
\begin{align*}
(d_{1},d_{2},d_{3})  &  \in D^{3}\longmapsto(d_{1},d_{2})\in D^{2}\\
(d_{1},d_{2})  &  \in D^{2}\longmapsto d_{1}d_{2}\in D\\
d  &  \in D\longmapsto(d,0,0)\in D(3)\text{.}%
\end{align*}

\end{proof}

\begin{corollary}
\label{2.12}Let $\nabla$ be a $D^{n}$-tangential \textit{over }the bundle
$\pi:E\rightarrow M$ \textit{at} $x\in E$. Let $\mathbb{D}$ be a simplicially
infinitesimal spaces of dimension less than or equal to $n$. Any nonstandard
quasi-colimit representation of $\mathbb{D}$,$\mathcal{\ }$if any mapping into
$\mathbb{D}$ in the representation is monomial, induces the same mapping as
$\nabla_{\mathbb{D}}$ (induced by the standard quasi-colimit representation of
$\mathbb{D}$) by the method in the proof of Theorem \ref{t4.1.5}.
\end{corollary}

\begin{proof}
It suffices to note that
\[
\nabla_{D^{m}}(\left(  \mathrm{id}_{M}\otimes\mathcal{W}_{\chi}\right)
\left(  \gamma\right)  )=\left(  \mathrm{id}_{E}\otimes\mathcal{W}_{\chi
}\right)  \left(  \nabla_{\mathbb{D}}(\gamma)\right)
\]
for any mapping $\chi:D^{m}\rightarrow\mathbb{D}$ in the given nonstandard
quasi-colimit representation of $\mathbb{D}$, which follows directly from the
above theorem.
\end{proof}

\section{\label{s5}The Third Approach to Jets}

\begin{definition}
\label{d5.1.1}Let $n$ be a natural number. A $D_{n}$-\textit{pseudotangential}
\textit{over }the bundle $\pi:E\rightarrow M$ \textit{at} $x\in E$ is a
mapping
\[
\nabla_{x}:\left(  M\otimes\mathcal{W}_{D_{n}}\right)  _{\pi(x)}%
\rightarrow\left(  E\otimes\mathcal{W}_{D_{n}}\right)  _{x}%
\]
abiding by the following two conditions:

\begin{enumerate}
\item We have
\[
\left(  \pi\otimes\mathrm{id}_{\mathcal{W}_{D_{n}}}\right)  \left(  \nabla
_{x}(\gamma)\right)  =\gamma
\]
for any $\gamma\in\left(  M\otimes\mathcal{W}_{D_{n}}\right)  _{\pi(x)}$.

\item For any $\gamma\in\left(  E\otimes\mathcal{W}_{D_{n}}\right)  _{x}$ and
any $\alpha\in\mathbb{R}$, we have
\[
\nabla_{x}(\alpha\gamma)=\alpha\nabla_{x}(\gamma)
\]

\item The diagram
\[%
\begin{array}
[c]{ccc}%
\left(  M\otimes\mathcal{W}_{D_{n}}\right)  _{\pi(x)} & \underrightarrow
{\mathrm{id}_{M}\otimes\mathcal{W}_{\left(  d_{1},d_{2}\right)  \in
D_{n}\times D_{m}\mapsto d_{1}d_{2}\in D_{n}}} & \left(  M\otimes
\mathcal{W}_{D_{n}}\right)  _{\pi(x)}\otimes\mathcal{W}_{D_{m}}\\%
\begin{array}
[c]{cc}%
\nabla_{x} & \downarrow
\end{array}
&  &
\begin{array}
[c]{cc}%
\downarrow & \nabla_{x}\otimes\mathrm{id}_{\mathcal{W}_{D_{m}}}%
\end{array}
\\
\left(  E\otimes\mathcal{W}_{D_{n}}\right)  _{x} & \underrightarrow
{\mathrm{id}_{E}\otimes\mathcal{W}_{\left(  d_{1},d_{2}\right)  \in
D_{n}\times D_{m}\mapsto d_{1}d_{2}\in D_{n}}} & \left(  E\otimes
\mathcal{W}_{D_{n}}\right)  _{x}\otimes\mathcal{W}_{D_{m}}%
\end{array}
\]
commutes, where $m$\ is an arbitrary natural number.
\end{enumerate}
\end{definition}

\begin{remark}
The third condition in the above definition claims what is called
infinitesimal linearity.
\end{remark}

\begin{notation}
We denote by $\widehat{\mathbb{J}}_{x}^{D_{n}}(\pi)$ the totality of $D_{n}%
$-pseudotangentials \textit{over }the bundle $\pi:E\rightarrow M$ \textit{at}
$x\in E$. We denote by $\widehat{\mathbb{J}}^{D_{n}}(\pi)$ the set-theoretic
union of $\widehat{\mathbb{J}}_{x}^{D_{n}}(\pi)$'s for all $x\in E$.
\end{notation}

It is easy to see that

\begin{lemma}
\label{t5.1.1}The following diagram is an equalizer in the category of Weil
algebras:
\begin{align*}
&  \mathcal{W}_{D_{n}}\underrightarrow{\,\mathcal{W}_{\left(  d_{1}%
,d_{2}\right)  \in D_{n+1}\times D_{n}\mapsto d_{1}d_{2}\in D_{n}}%
}\,\mathcal{W}_{D_{n+1}\times D_{n}}\\
&
\begin{array}
[c]{c}%
\underrightarrow{\mathcal{W}_{\left(  d_{1},d_{2},d_{3}\right)  \in
D_{n+1}\times D_{n+1}\times D_{n}\mapsto\left(  d_{1}d_{2},d_{3}\right)  \in
D_{n+1}\times D_{n}}}\\
\overrightarrow{\mathcal{W}_{\left(  d_{1},d_{2},d_{3}\right)  \in
D_{n+1}\times D_{n+1}\times D_{n}\mapsto\left(  d_{1},d_{2}d_{3}\right)  \in
D_{n+1}\times D_{n}}}%
\end{array}
\,\mathcal{W}_{D_{n+1}\times D_{n+1}\times D_{n}}%
\end{align*}

\end{lemma}

\begin{proposition}
\label{t5.1.2}Let $\nabla_{x}$ be a $D_{n+1}$-pseudotangential \textit{over
}the bundle $\pi:E\rightarrow M$ \textit{at} $x\in E$ and $\gamma\in\left(
M\otimes\mathcal{W}_{D_{n}}\right)  _{\pi\left(  x\right)  }$. Then there
exists a unique $\gamma^{\prime}\in\left(  E\otimes\mathcal{W}_{D_{n}}\right)
_{x}$ such that the composition of mappings
\begin{align}
&  \left(  M\otimes\mathcal{W}_{D_{n}}\right)  _{\pi\left(  x\right)
}\underrightarrow{\,\mathrm{id}_{M}\otimes\mathcal{W}_{\left(  d_{1}%
,d_{2}\right)  \in D_{n+1}\times D_{n}\mapsto d_{1}d_{2}\in D_{n}}}\,\left(
M\otimes\mathcal{W}_{D_{n+1}}\right)  _{\pi\left(  x\right)  }\otimes
\mathcal{W}_{D_{n}}\nonumber\\
&  \underrightarrow{\,\nabla_{x}\otimes\mathrm{id}_{\mathcal{W}_{D_{n}}}%
}\,\left(  E\otimes\mathcal{W}_{D_{n+1}}\right)  _{x}\otimes\mathcal{W}%
_{D_{n}} \label{5.1.2.a}%
\end{align}
applied to $\gamma$\ results in
\begin{equation}
\left(  \mathrm{id}_{E}\otimes\mathcal{W}_{\left(  d_{1},d_{2}\right)  \in
D_{n+1}\times D_{n}\mapsto d_{1}d_{2}\in D_{n}}\right)  \left(  \gamma
^{\prime}\right)  \label{5.1.2.b}%
\end{equation}

\end{proposition}

\begin{proof}
By dint of Lemma \ref{t5.1.1}, it suffices to show that the composition of
mappings
\begin{align}
&  \left(  M\otimes\mathcal{W}_{D_{n}}\right)  _{\pi\left(  x\right)
}\underrightarrow{\,\mathrm{id}_{M}\otimes\mathcal{W}_{\left(  d_{1}%
,d_{2}\right)  \in D_{n+1}\times D_{n}\mapsto d_{1}d_{2}\in D_{n}}}\,\left(
M\otimes\mathcal{W}_{D_{n+1}}\right)  _{\pi\left(  x\right)  }\otimes
\mathcal{W}_{D_{n}}\nonumber\\
&  \underrightarrow{\,\nabla_{x}\otimes\mathrm{id}_{\mathcal{W}_{D_{n}}}%
}\,\left(  E\otimes\mathcal{W}_{D_{n+1}}\right)  _{x}\otimes\mathcal{W}%
_{D_{n}}\nonumber\\
&  \underrightarrow{\,\mathrm{id}_{E}\otimes\mathcal{W}_{\left(  d_{1}%
,d_{2},d_{3}\right)  \in D_{n+1}\times D_{n+1}\times D_{n}\mapsto\left(
d_{1},d_{2}d_{3}\right)  \in D_{n+1}\times D_{n}}}\,\nonumber\\
&  \left(  E\otimes\mathcal{W}_{D_{n+1}}\right)  _{x}\otimes\mathcal{W}%
_{D_{n+1}\times D_{n}} \label{5.1.2.1}%
\end{align}
is equal to the composition of mappings
\begin{align}
&  \left(  M\otimes\mathcal{W}_{D_{n}}\right)  _{\pi\left(  x\right)
}\underrightarrow{\,\mathrm{id}_{M}\otimes\mathcal{W}_{\left(  d_{1}%
,d_{2}\right)  \in D_{n+1}\times D_{n}\mapsto d_{1}d_{2}\in D_{n}}}\,\left(
M\otimes\mathcal{W}_{D_{n+1}}\right)  _{\pi\left(  x\right)  }\otimes
\mathcal{W}_{D_{n}}\nonumber\\
&  \underrightarrow{\,\nabla_{x}\otimes\mathrm{id}_{\mathcal{W}_{D_{n}}}%
}\,\left(  E\otimes\mathcal{W}_{D_{n+1}}\right)  _{x}\otimes\mathcal{W}%
_{D_{n}}\nonumber\\
&  \underrightarrow{\,\mathrm{id}_{E}\otimes\mathcal{W}_{\left(  d_{1}%
,d_{2},d_{3}\right)  \in D_{n+1}\times D_{n+1}\times D_{n}\mapsto\left(
d_{1}d_{2},d_{3}\right)  \in D_{n+1}\times D_{n}}}\,\nonumber\\
&  \left(  E\otimes\mathcal{W}_{D_{n+1}}\right)  _{x}\otimes\mathcal{W}%
_{D_{n+1}\times D_{n}} \label{5.1.2.2}%
\end{align}
Since $\otimes$\ is a bifunctor, the diagram
\[%
\begin{array}
[c]{ccc}%
\left(  M\otimes\mathcal{W}_{D_{n+1}}\right)  _{\pi\left(  x\right)  }%
\otimes\mathcal{W}_{D_{n}} & \rightarrow & \left(  M\otimes\mathcal{W}%
_{D_{n+1}}\right)  _{\pi\left(  x\right)  }\otimes\mathcal{W}_{D_{n+1}\times
D_{n}}\\%
\begin{array}
[c]{cc}%
\nabla_{x}\otimes\mathrm{id}_{\mathcal{W}_{D_{n}}} & \downarrow
\end{array}
&  &
\begin{array}
[c]{cc}%
\downarrow & \nabla_{x}\otimes\mathrm{id}_{\mathcal{W}_{D_{n+1}\times D_{n}}}%
\end{array}
\\
\left(  E\otimes\mathcal{W}_{D_{n+1}}\right)  _{x}\otimes\mathcal{W}_{D_{n}} &
\rightarrow & \left(  E\otimes\mathcal{W}_{D_{n+1}}\right)  _{x}%
\otimes\mathcal{W}_{D_{n+1}\times D_{n}}%
\end{array}
\]
commutes, where the upper horizontal arrow is
\[
\mathrm{id}_{M}\otimes\mathcal{W}_{\left(  d_{1},d_{2},d_{3}\right)  \in
D_{n+1}\times D_{n+1}\times D_{n}\mapsto\left(  d_{1},d_{2}d_{3}\right)  \in
D_{n+1}\times D_{n}}\text{,}%
\]
while the lower horizontal arrow is
\[
\mathrm{id}_{E}\otimes\mathcal{W}_{\left(  d_{1},d_{2},d_{3}\right)  \in
D_{n+1}\times D_{n+1}\times D_{n}\mapsto\left(  d_{1},d_{2}d_{3}\right)  \in
D_{n+1}\times D_{n}}\text{.}%
\]
Therefore the composition of mappings in (\ref{5.1.2.1}) is equal to the
composition of mappings
\begin{align}
&  \left(  M\otimes\mathcal{W}_{D_{n}}\right)  _{\pi\left(  x\right)
}\underrightarrow{\,\mathrm{id}_{M}\otimes\mathcal{W}_{\left(  d_{1}%
,d_{2}\right)  \in D_{n+1}\times D_{n}\mapsto d_{1}d_{2}\in D_{n}}}\,\left(
M\otimes\mathcal{W}_{D_{n+1}}\right)  _{\pi\left(  x\right)  }\otimes
\mathcal{W}_{D_{n}}\nonumber\\
&  \underrightarrow{\mathrm{id}_{M}\otimes\mathcal{W}_{\left(  d_{1}%
,d_{2},d_{3}\right)  \in D_{n+1}\times D_{n+1}\times D_{n}\mapsto\left(
d_{1},d_{2}d_{3}\right)  \in D_{n+1}\times D_{n}}}\,\left(  M\otimes
\mathcal{W}_{D_{n+1}}\right)  _{\pi\left(  x\right)  }\otimes\mathcal{W}%
_{D_{n+1}\times D_{n}}\nonumber\\
&  \underrightarrow{\,\nabla_{x}\otimes\mathrm{id}_{\mathcal{W}_{D_{n+1}\times
D_{n}}}}\,\left(  E\otimes\mathcal{W}_{D_{n+1}}\right)  _{x}\otimes
\mathcal{W}_{D_{n+1}\times D_{n}} \label{5.1.2.3}%
\end{align}
Since the composition of mappings
\begin{align}
&  M\otimes\mathcal{W}_{D_{n}}\underrightarrow{\,\mathrm{id}_{M}%
\otimes\mathcal{W}_{\left(  d_{1},d_{2}\right)  \in D_{n+1}\times D_{n}\mapsto
d_{1}d_{2}\in D_{n}}}\,M\otimes\mathcal{W}_{D_{n+1}\times D_{n}}\nonumber\\
&  \underrightarrow{\mathrm{id}_{M}\otimes\mathcal{W}_{\left(  d_{1}%
,d_{2},d_{3}\right)  \in D_{n+1}\times D_{n+1}\times D_{n}\mapsto\left(
d_{1},d_{2}d_{3}\right)  \in D_{n+1}\times D_{n}}}M\otimes\mathcal{W}%
_{D_{n+1}\times D_{n+1}\times D_{n}}\nonumber
\end{align}
is trivially equal to the composition of mappings
\begin{align}
&  M\otimes\mathcal{W}_{D_{n}}\underrightarrow{\mathrm{id}_{M}\otimes
\mathcal{W}_{\left(  d_{1},d_{2}\right)  \in D_{n+1}\times D_{n}\mapsto
d_{1}d_{2}\in D_{n}}}M\otimes\mathcal{W}_{D_{n+1}\times D_{n}}\nonumber\\
&  \underrightarrow{\mathrm{id}_{M}\otimes\mathcal{W}_{\left(  d_{1}%
,d_{2},d_{3}\right)  \in D_{n+1}\times D_{n+1}\times D_{n}\mapsto\left(
d_{1}d_{2},d_{3}\right)  \in D_{n+1}\times D_{n}}}\,M\otimes\mathcal{W}%
_{D_{n+1}\times D_{n+1}\times D_{n}}\text{,}\nonumber
\end{align}
the composition of mappings in (\ref{5.1.2.3}) is equal to the composition of
mappings
\begin{align}
&  \left(  M\otimes\mathcal{W}_{D_{n}}\right)  _{\pi\left(  x\right)
}\underrightarrow{\,\mathrm{id}_{M}\otimes\mathcal{W}_{\left(  d_{1}%
,d_{2}\right)  \in D_{n+1}\times D_{n}\mapsto d_{1}d_{2}\in D_{n}}}\,\left(
M\otimes\mathcal{W}_{D_{n+1}}\right)  _{\pi\left(  x\right)  }\otimes
\mathcal{W}_{D_{n}}\nonumber\\
&  \underrightarrow{\mathrm{id}_{M}\otimes\mathcal{W}_{\left(  d_{1}%
,d_{2},d_{3}\right)  \in D_{n+1}\times D_{n+1}\times D_{n}\mapsto\left(
d_{1}d_{2},d_{3}\right)  \in D_{n+1}\times D_{n}}}\,\left(  M\otimes
\mathcal{W}_{D_{n+1}}\right)  _{\pi\left(  x\right)  }\otimes\mathcal{W}%
_{D_{n+1}\times D_{n}}\nonumber\\
&  \underrightarrow{\,\nabla_{x}\otimes\mathrm{id}_{\mathcal{W}_{D_{n+1}\times
D_{n}}}}\,\left(  E\otimes\mathcal{W}_{D_{n+1}}\right)  _{x}\otimes
\mathcal{W}_{D_{n+1}\times D_{n}} \label{5.1.2.4}%
\end{align}
By dint of the third condition in Definition \ref{d5.1.1}, the diagram
\[%
\begin{array}
[c]{ccc}%
\left(  M\otimes\mathcal{W}_{D_{n+1}}\right)  _{\pi\left(  x\right)  }%
\otimes\mathcal{W}_{D_{n}} & \rightarrow & \left(  M\otimes\mathcal{W}%
_{D_{n+1}}\right)  _{\pi\left(  x\right)  }\otimes\mathcal{W}_{D_{n+1}\times
D_{n}}\\%
\begin{array}
[c]{cc}%
\nabla_{x}\otimes\mathrm{id}_{\mathcal{W}_{D_{n}}} & \downarrow
\end{array}
&  &
\begin{array}
[c]{cc}%
\downarrow & \nabla_{x}\otimes\mathrm{id}_{\mathcal{W}_{D_{n+1}\times D_{n}}}%
\end{array}
\\
\left(  E\otimes\mathcal{W}_{D_{n+1}}\right)  _{x}\otimes\mathcal{W}_{D_{n}} &
\rightarrow & \left(  E\otimes\mathcal{W}_{D_{n+1}}\right)  _{x}%
\otimes\mathcal{W}_{D_{n+1}\times D_{n}}%
\end{array}
\]
commutes, where the upper horizontal arrow is
\[
\mathrm{id}_{M}\otimes\mathcal{W}_{\left(  d_{1},d_{2},d_{3}\right)  \in
D_{n+1}\times D_{n+1}\times D_{n}\mapsto\left(  d_{1}d_{2},d_{3}\right)  \in
D_{n+1}\times D_{n}}\text{,}%
\]
and the lower horizontal arrow is
\[
\mathrm{id}_{E}\otimes\mathcal{W}_{\left(  d_{1},d_{2},d_{3}\right)  \in
D_{n+1}\times D_{n+1}\times D_{n}\mapsto\left(  d_{1}d_{2},d_{3}\right)  \in
D_{n+1}\times D_{n}}\text{.}%
\]
Therefore the composition of mappings in (\ref{5.1.2.4}) is equal to the
composition of mappings in (\ref{5.1.2.2}), which completes the proof.
\end{proof}

It is not difficult to see that

\begin{proposition}
\label{t5.1.3}Given a $D_{n+1}$-pseudotangential $\nabla_{x}$\ \textit{over
}the bundle $\pi:E\rightarrow M$ \textit{at} $x\in E$, the assignment
$\gamma\in\left(  M\otimes\mathcal{W}_{D_{n}}\right)  _{\pi\left(  x\right)
}\mapsto\gamma^{\prime}\in\left(  E\otimes\mathcal{W}_{D_{n}}\right)  _{x} $
in the above proposition, denoted by $\hat{\pi}_{n+1,n}(\nabla_{x})$, is a
$D_{n}$-pseudotangential \textit{over }the bundle $\pi:E\rightarrow M $
\textit{at} $x\in E$.
\end{proposition}

\begin{proof}
We have to verify the three conditions in Definition \ref{d5.1.1} concerning
the mapping $\hat{\pi}_{n+1,n}(\nabla_{x}):\left(  M\otimes\mathcal{W}_{D_{n}%
}\right)  _{\pi\left(  x\right)  }\rightarrow\left(  E\otimes\mathcal{W}%
_{D_{n}}\right)  _{x}$.

\begin{enumerate}
\item To see the first condition, it suffices to show that
\[
\left(  \mathrm{id}_{M}\otimes\mathcal{W}_{\left(  d_{1},d_{2}\right)  \in
D_{n+1}\times D_{n}\mapsto d_{1}d_{2}\in D_{n}}\right)  \circ\left(
\pi\otimes\mathrm{id}_{\mathcal{W}_{D_{n}}}\right)  \left(  \left(  \hat{\pi
}_{n+1,n}(\nabla_{x})\right)  \left(  \gamma\right)  \right)  =\gamma\text{,}%
\]
which is equivalent to
\[
\left(  \pi\otimes\mathrm{id}_{\mathcal{W}_{D_{n+1}\times D_{n}}}\right)
\circ\left(  \mathrm{id}_{E}\otimes\mathcal{W}_{\left(  d_{1},d_{2}\right)
\in D_{n+1}\times D_{n}\mapsto d_{1}d_{2}\in D_{n}}\right)  \left(  \left(
\hat{\pi}_{n+1,n}(\nabla_{x})\right)  \left(  \gamma\right)  \right)
=\gamma\text{,}%
\]
since $\otimes$\ is a bifunctor. Therefore it suffices to show that the
composition of mappings
\begin{align*}
&  \left(  M\otimes\mathcal{W}_{D_{n}}\right)  _{\pi\left(  x\right)
}\underrightarrow{\,\mathrm{id}_{M}\otimes\mathcal{W}_{\left(  d_{1}%
,d_{2}\right)  \in D_{n+1}\times D_{n}\mapsto d_{1}d_{2}\in D_{n}}}\,\left(
M\otimes\mathcal{W}_{D_{n+1}}\right)  _{\pi\left(  x\right)  }\otimes
\mathcal{W}_{D_{n}}\\
&  \underrightarrow{\,\nabla_{x}\otimes\mathrm{id}_{\mathcal{W}_{D_{n}}}%
}\,\left(  E\otimes\mathcal{W}_{D_{n+1}}\right)  _{x}\otimes\mathcal{W}%
_{D_{n}}\underrightarrow{\,\pi\otimes\mathrm{id}_{\mathcal{W}_{D_{n+1}\times
D_{n}}}}\,\left(  M\otimes\mathcal{W}_{D_{n+1}}\right)  _{x}\otimes
\mathcal{W}_{D_{n}}%
\end{align*}
applied to $\gamma$\ results in
\[
\left(  \mathrm{id}_{M}\otimes\mathcal{W}_{\left(  d_{1},d_{2}\right)  \in
D_{n+1}\times D_{n}\mapsto d_{1}d_{2}\in D_{n}}\right)  \left(  \gamma\right)
\text{,}%
\]
which follows directly from the first condition in Definition \ref{d5.1.1}.

\item To see the second, let us note first that the composition of mappings
\begin{align*}
&  \left(  M\otimes\mathcal{W}_{D_{n}}\right)  _{\pi\left(  x\right)
}\underrightarrow{\,\mathrm{id}_{M}\otimes\mathcal{W}_{d\in D_{n}\mapsto\alpha
d\in D_{n}}}\,\left(  M\otimes\mathcal{W}_{D_{n}}\right)  _{\pi\left(
x\right)  }\underrightarrow{\,\mathrm{id}_{M}\otimes\mathcal{W}_{\left(
d_{1},d_{2}\right)  \in D_{n+1}\times D_{n}\mapsto d_{1}d_{2}\in D_{n}}}\,\\
&  \left(  M\otimes\mathcal{W}_{D_{n+1}}\right)  _{\pi\left(  x\right)
}\otimes\mathcal{W}_{D_{n}}%
\end{align*}
is equal to the composition of mappings
\begin{align*}
&  \left(  M\otimes\mathcal{W}_{D_{n}}\right)  _{\pi\left(  x\right)
}\underrightarrow{\,\mathrm{id}_{M}\otimes\mathcal{W}_{\left(  d_{1}%
,d_{2}\right)  \in D_{n+1}\times D_{n}\mapsto d_{1}d_{2}\in D_{n}}}\,\left(
M\otimes\mathcal{W}_{D_{n+1}}\right)  _{\pi\left(  x\right)  }\otimes
\mathcal{W}_{D_{n}}\\
&  \underrightarrow{\mathrm{id}_{M}\otimes\mathcal{W}_{\left(  d_{1}%
,d_{2}\right)  \in D_{n+1}\times D_{n}\mapsto\left(  \alpha d_{1}%
,d_{2}\right)  \in D_{n+1}\times D_{n}}}\,\left(  M\otimes\mathcal{W}%
_{D_{n+1}}\right)  _{\pi\left(  x\right)  }\otimes\mathcal{W}_{D_{n}}%
\end{align*}
Since $\nabla_{x}$ is a $D_{n+1}$-pseudotangential over\textit{\ }the bundle
$\pi:E\rightarrow M$ \textit{at} $x\in E$, the diagram
\[%
\begin{array}
[c]{ccc}%
\left(  M\otimes\mathcal{W}_{D_{n+1}}\right)  _{\pi\left(  x\right)  }%
\otimes\mathcal{W}_{D_{n}} & \rightarrow & \left(  M\otimes\mathcal{W}%
_{D_{n+1}}\right)  _{\pi\left(  x\right)  }\otimes\mathcal{W}_{D_{n}}\\%
\begin{array}
[c]{cc}%
\nabla_{x}\otimes\mathrm{id}_{\mathcal{W}_{D_{n}}} & \downarrow
\end{array}
&  &
\begin{array}
[c]{cc}%
\downarrow & \nabla_{x}\otimes\mathrm{id}_{\mathcal{W}_{D_{n}}}%
\end{array}
\\
\left(  E\otimes\mathcal{W}_{D_{n+1}}\right)  _{x}\otimes\mathcal{W}_{D_{n}} &
\rightarrow & \left(  E\otimes\mathcal{W}_{D_{n+1}}\right)  _{x}%
\otimes\mathcal{W}_{D_{n}}%
\end{array}
\]
commutes, where the upper horizontal arrow is
\[
\mathrm{id}_{M}\otimes\mathcal{W}_{\left(  d_{1},d_{2}\right)  \in
D_{n+1}\times D_{n}\mapsto\left(  \alpha d_{1},d_{2}\right)  \in D_{n+1}\times
D_{n}}\text{,}%
\]
while the lower horizontal arrow is
\[
\mathrm{id}_{E}\otimes\mathcal{W}_{\left(  d_{1},d_{2}\right)  \in
D_{n+1}\times D_{n}\mapsto\left(  \alpha d_{1},d_{2}\right)  \in D_{n+1}\times
D_{n}}\text{.}%
\]
Therefore the composition of mappings
\begin{align*}
&  \left(  M\otimes\mathcal{W}_{D_{n}}\right)  _{\pi\left(  x\right)
}\underrightarrow{\,\mathrm{id}_{M}\otimes\mathcal{W}_{d\in D_{n}\mapsto\alpha
d\in D_{n}}}\,\left(  M\otimes\mathcal{W}_{D_{n}}\right)  _{\pi\left(
x\right)  }\\
&  \underrightarrow{\,\mathrm{id}_{M}\otimes\mathcal{W}_{\left(  d_{1}%
,d_{2}\right)  \in D_{n+1}\times D_{n}\mapsto d_{1}d_{2}\in D_{n}}}\,\left(
M\otimes\mathcal{W}_{D_{n+1}}\right)  _{\pi\left(  x\right)  }\otimes
\mathcal{W}_{D_{n}}\\
&  \underrightarrow{\,\nabla_{x}\otimes\mathrm{id}_{\mathcal{W}_{D_{n}}}%
}\,\left(  E\otimes\mathcal{W}_{D_{n+1}}\right)  _{x}\otimes\mathcal{W}%
_{D_{n}}%
\end{align*}
is equal to the composition of mappings
\begin{align*}
&  \left(  M\otimes\mathcal{W}_{D_{n}}\right)  _{\pi\left(  x\right)
}\underrightarrow{\,\mathrm{id}_{M}\otimes\mathcal{W}_{\left(  d_{1}%
,d_{2}\right)  \in D_{n+1}\times D_{n}\mapsto d_{1}d_{2}\in D_{n}}}\,\left(
M\otimes\mathcal{W}_{D_{n+1}}\right)  _{\pi\left(  x\right)  }\otimes
\mathcal{W}_{D_{n}}\\
&  \underrightarrow{\,\nabla_{x}\otimes\mathrm{id}_{\mathcal{W}_{D_{n}}}%
}\,\left(  E\otimes\mathcal{W}_{D_{n+1}}\right)  _{x}\otimes\mathcal{W}%
_{D_{n}}\underrightarrow{\,\mathrm{id}_{E}\otimes\mathcal{W}_{\left(
d_{1},d_{2}\right)  \in D_{n+1}\times D_{n}\mapsto\left(  \alpha d_{1}%
,d_{2}\right)  \in D_{n+1}\times D_{n}}}\,\\
&  \left(  E\otimes\mathcal{W}_{D_{n+1}}\right)  _{x}\otimes\mathcal{W}%
_{D_{n}}%
\end{align*}
The former composition of mappings applied to $\gamma\in\left(  M\otimes
\mathcal{W}_{D_{n}}\right)  _{\pi\left(  x\right)  }$ results in
\[
\left(  \mathrm{id}_{E}\otimes\mathcal{W}_{\left(  d_{1},d_{2}\right)  \in
D_{n+1}\times D_{n}\mapsto d_{1}d_{2}\in D_{n}}\right)  \left(  \hat{\pi
}_{n+1,n}(\nabla_{x})(\alpha\gamma)\right)  \text{,}%
\]
while the latter composition of mappings applied to $\gamma$ results in
\begin{align*}
&  \left(  \mathrm{id}_{E}\otimes\mathcal{W}_{\left(  d_{1},d_{2}\right)  \in
D_{n+1}\times D_{n}\mapsto\left(  \alpha d_{1},d_{2}\right)  \in D_{n+1}\times
D_{n}}\right)  \circ\\
&  \left(  \mathrm{id}_{E}\otimes\mathcal{W}_{\left(  d_{1},d_{2}\right)  \in
D_{n+1}\times D_{n}\mapsto d_{1}d_{2}\in D_{n}}\right)  \left(  \hat{\pi
}_{n+1,n}(\nabla_{x})(\gamma)\right) \\
&  =\left(  \mathrm{id}_{E}\otimes\mathcal{W}_{\left(  d_{1},d_{2}\right)  \in
D_{n+1}\times D_{n}\mapsto d_{1}d_{2}\in D_{n}}\right)  \left(  \alpha\left(
\hat{\pi}_{n+1,n}(\nabla_{x})(\gamma)\right)  \right)  \text{.}%
\end{align*}
Therefore we have
\[
\hat{\pi}_{n+1,n}(\nabla_{x})(\alpha\gamma)=\alpha\left(  \hat{\pi}%
_{n+1,n}(\nabla_{x})(\gamma)\right)
\]

\item To see the third, we have to show that the diagram
\begin{equation}%
\begin{array}
[c]{ccc}%
\left(  M\otimes\mathcal{W}_{D_{n}}\right)  _{\pi\left(  x\right)  } &
\underrightarrow{\mathrm{id}_{M}\otimes\mathcal{W}_{\mathbf{m}_{D_{n}\times
D_{m}\rightarrow D_{n}}}} & \left(  M\otimes\mathcal{W}_{D_{n}}\right)
_{\pi\left(  x\right)  }\otimes\mathcal{W}_{D_{m}}\\%
\begin{array}
[c]{cc}%
\hat{\pi}_{n+1,n}(\nabla_{x}) & \downarrow
\end{array}
&  &
\begin{array}
[c]{cc}%
\downarrow & \hat{\pi}_{n+1,n}(\nabla_{x})\otimes\mathrm{id}_{\mathcal{W}%
_{D_{m}}}%
\end{array}
\\
\left(  E\otimes\mathcal{W}_{D_{n}}\right)  _{x} & \underrightarrow
{\mathrm{id}_{E}\otimes\mathcal{W}_{\mathbf{m}_{D_{n}\times D_{m}\rightarrow
D_{n}}}} & \left(  E\otimes\mathcal{W}_{D_{n}}\right)  _{x}\otimes
\mathcal{W}_{D_{m}}%
\end{array}
\label{5.1.3.1}%
\end{equation}
commutes, where $m$\ is an arbitrary natural number. Since the lower square of
the diagram
\begin{equation}%
\begin{array}
[c]{ccc}%
\left(  M\otimes\mathcal{W}_{D_{n}}\right)  _{\pi\left(  x\right)  } &
\underrightarrow{\mathrm{id}_{M}\otimes\mathcal{W}_{\mathbf{m}_{D_{n}\times
D_{m}\rightarrow D_{n}}}} & \left(  M\otimes\mathcal{W}_{D_{n}}\right)
_{\pi\left(  x\right)  }\otimes\mathcal{W}_{D_{m}}\\%
\begin{array}
[c]{cc}%
\hat{\pi}_{n+1,n}(\nabla_{x}) & \downarrow
\end{array}
&  &
\begin{array}
[c]{cc}%
\downarrow & \hat{\pi}_{n+1,n}(\nabla_{x})\otimes\mathrm{id}_{\mathcal{W}%
_{D_{n}}}%
\end{array}
\\
\left(  E\otimes\mathcal{W}_{D_{n}}\right)  _{x} & \underrightarrow
{\mathrm{id}_{E}\otimes\mathcal{W}_{\mathbf{m}_{D_{n}\times D_{m}\rightarrow
D_{n}}}} & \left(  E\otimes\mathcal{W}_{D_{n}}\right)  _{x}\otimes
\mathcal{W}_{D_{m}}\\%
\begin{array}
[c]{cc}%
\mathrm{id}_{E}\otimes\mathcal{W}_{\mathbf{m}_{D_{n+1}\times D_{n}\rightarrow
D_{n}}} & \downarrow
\end{array}
&  &
\begin{array}
[c]{cc}%
\downarrow & \mathrm{id}_{E}\otimes\mathcal{W}_{\mathbf{m}_{D_{n+1}\times
D_{n}\rightarrow D_{n}}\times\mathrm{id}_{D_{m}}}%
\end{array}
\\
\left(  E\otimes\mathcal{W}_{D_{n+1}}\right)  _{x}\otimes\mathcal{W}_{D_{n}} &
\underrightarrow{\mathrm{id}_{E}\otimes\mathcal{W}_{\mathrm{id}_{D_{n+1}%
}\times\mathbf{m}_{D_{n}\times D_{m}\rightarrow D_{n}}}} & \left(
E\otimes\mathcal{W}_{D_{n+1}}\right)  _{x}\otimes\mathcal{W}_{D_{n}\times
D_{m}}%
\end{array}
\label{5.1.3.2}%
\end{equation}
commutes, so that the commutativity of the diagram in (\ref{5.1.3.1}) is
equivalent to the commutativity of the outer square of the diagram in
(\ref{5.1.3.2}). .The composition of mappings
\[
\left(  M\otimes\mathcal{W}_{D_{n}}\right)  _{\pi\left(  x\right)
}\underrightarrow{\,\hat{\pi}_{n+1,n}(\nabla_{x})}\,\left(  E\otimes
\mathcal{W}_{D_{n}}\right)  _{x}\underrightarrow{\,\mathrm{id}_{E}%
\otimes\mathcal{W}_{\mathbf{m}_{D_{n+1}\times D_{n}\rightarrow D_{n}}}%
}\,\left(  E\otimes\mathcal{W}_{D_{n+1}}\right)  _{x}\otimes\mathcal{W}%
_{D_{n}}%
\]
is equal to the composition of mappings
\begin{align*}
&  \left(  M\otimes\mathcal{W}_{D_{n}}\right)  _{\pi\left(  x\right)
}\underrightarrow{\,\mathrm{id}_{M}\otimes\mathcal{W}_{\mathbf{m}%
_{D_{n+1}\times D_{n}\rightarrow D_{n}}}}\,\left(  M\otimes\mathcal{W}%
_{D_{n+1}}\right)  _{\pi\left(  x\right)  }\otimes\mathcal{W}_{D_{n}}\\
&  \underrightarrow{\,\nabla_{x}\otimes\mathrm{id}_{\mathcal{W}_{D_{n}}}%
}\,\left(  E\otimes\mathcal{W}_{D_{n+1}}\right)  _{x}\otimes\mathcal{W}%
_{D_{n}}\text{,}%
\end{align*}
while the composition of mappings
\begin{align*}
&  \left(  M\otimes\mathcal{W}_{D_{n}}\right)  _{\pi\left(  x\right)  }%
\otimes\mathcal{W}_{D_{m}}\underrightarrow{\,\hat{\pi}_{n+1,n}(\nabla
_{x})\otimes\mathrm{id}_{\mathcal{W}_{D_{m}}}}\,\left(  E\otimes
\mathcal{W}_{D_{n}}\right)  _{x}\otimes\mathcal{W}_{D_{m}}\\
&  \underrightarrow{\,\mathrm{id}_{E}\otimes\mathcal{W}_{\mathbf{m}%
_{D_{n+1}\times D_{n}\rightarrow D_{n}}\times\mathrm{id}_{D_{m}}}}\,\left(
E\otimes\mathcal{W}_{D_{n+1}}\right)  _{x}\otimes\mathcal{W}_{D_{n}\times
D_{m}}%
\end{align*}
is equal to the composition of mappings
\begin{align*}
&  \left(  M\otimes\mathcal{W}_{D_{n}}\right)  _{\pi\left(  x\right)  }%
\otimes\mathcal{W}_{D_{m}}\underrightarrow{\,\mathrm{id}_{M}\otimes
\mathcal{W}_{\mathbf{m}_{D_{n+1}\times D_{n}\rightarrow D_{n}}\times
\mathrm{id}_{D_{m}}}}\,\left(  M\otimes\mathcal{W}_{D_{n+1}}\right)
_{\pi\left(  x\right)  }\otimes\mathcal{W}_{D_{n}\times D_{m}}\\
&  \underrightarrow{\nabla_{x}\otimes\mathrm{id}_{\mathcal{W}_{D_{n}\times
D_{m}}}}\,\left(  E\otimes\mathcal{W}_{D_{n+1}}\right)  _{x}\otimes
\mathcal{W}_{D_{n}\times D_{m}}%
\end{align*}
It is easy to see that the diagram
\[%
\begin{array}
[c]{ccc}%
\left(  M\otimes\mathcal{W}_{D_{n}}\right)  _{\pi\left(  x\right)  } &
\underrightarrow{\mathrm{id}_{M}\otimes\mathcal{W}_{\mathbf{m}_{D_{n}\times
D_{m}\rightarrow D_{n}}}} & \left(  M\otimes\mathcal{W}_{D_{n}}\right)
_{\pi\left(  x\right)  }\otimes\mathcal{W}_{D_{m}}\\%
\begin{array}
[c]{cc}%
\mathrm{id}_{M}\otimes\mathcal{W}_{\mathbf{m}_{D_{n+1}\times D_{n}\rightarrow
D_{n}}} & \downarrow
\end{array}
&  &
\begin{array}
[c]{cc}%
\downarrow & \mathrm{id}_{M}\otimes\mathcal{W}_{\mathbf{m}_{D_{n+1}\times
D_{n}\rightarrow D_{n}}\times\mathrm{id}_{D_{m}}}%
\end{array}
\\
\left(  M\otimes\mathcal{W}_{D_{n+1}}\right)  _{\pi\left(  x\right)  }%
\otimes\mathcal{W}_{D_{n}} & \underrightarrow{\mathrm{id}_{M}\otimes
\mathcal{W}_{\mathrm{id}_{D_{n+1}}\times\mathbf{m}_{D_{n}\times D_{m}%
\rightarrow D_{n}}}} & \left(  M\otimes\mathcal{W}_{D_{n+1}}\right)
_{\pi\left(  x\right)  }\otimes\mathcal{W}_{D_{n}\times D_{m}}\\%
\begin{array}
[c]{cc}%
\nabla_{x}\otimes\mathrm{id}_{\mathcal{W}_{D_{n}}} & \downarrow
\end{array}
&  &
\begin{array}
[c]{cc}%
\downarrow & \nabla_{x}\otimes\mathrm{id}_{\mathcal{W}_{D_{n}\times D_{m}}}%
\end{array}
\\
\left(  E\otimes\mathcal{W}_{D_{n+1}}\right)  _{x}\otimes\mathcal{W}_{D_{n}} &
\underrightarrow{\mathrm{id}_{E}\otimes\mathcal{W}_{\mathrm{id}_{D_{n+1}%
}\times\mathbf{m}_{D_{n}\times D_{m}\rightarrow D_{n}}}} & \left(
E\otimes\mathcal{W}_{D_{n+1}}\right)  _{x}\otimes\mathcal{W}_{D_{n}\times
D_{m}}%
\end{array}
\]
commutes, which implies that the outer square of the diagram in (\ref{5.1.3.2}%
) commutes. This completes the proof.
\end{enumerate}
\end{proof}

\begin{notation}
By the above proposition, we have the canonical projection $\hat{\pi}%
_{n+1,n}:\widehat{\mathbb{J}}^{D_{n+1}}(\pi)\rightarrow\widehat{\mathbb{J}%
}^{D_{n}}(\pi)$ so that, given $\nabla_{x}\in\widehat{\mathbb{J}}_{x}%
^{D_{n+1}}(\pi)$\ and $\gamma\in\left(  M\otimes\mathcal{W}_{D_{n}}\right)
_{\pi\left(  x\right)  }$, the composition of mappings in (\ref{5.1.2.a})
applied to $\gamma$\ results in
\[
\left(  \mathrm{id}_{E}\otimes\mathcal{W}_{\left(  d_{1},d_{2}\right)  \in
D_{n+1}\times D_{n}\mapsto d_{1}d_{2}\in D_{n}}\right)  \left(  \hat{\pi
}_{n+1,n}(\nabla_{x})(\gamma)\right)
\]
For any natural numbers $n$, $m$ with $m\leq n$, we define $\hat{\pi}%
_{n,m}:\widehat{\mathbb{J}}^{D_{n}}(\pi)\rightarrow\widehat{\mathbb{J}}%
^{D_{m}}(\pi)$ to be $\hat{\pi}_{m+1,m}\circ...\circ\hat{\pi}_{n,n-1}$.
\end{notation}

\begin{proposition}
\label{t5.1.4}Let $\nabla_{x}$ be a $D_{n+1}$-pseudotangential \textit{over
}the bundle $\pi:E\rightarrow M$ \textit{at} $x\in E$. Then the diagram
\[%
\begin{array}
[c]{ccc}%
\quad\quad\left(  M\otimes\mathcal{W}_{D_{n+1}}\right)  _{\pi\left(  x\right)
} & \underrightarrow{\ \ \ \ \ \ \ \ \ \ \ \ \nabla_{x}%
\ \ \ \ \ \ \ \ \ \ \ \ \ \ } & \left(  E\otimes\mathcal{W}_{D_{n+1}}\right)
_{x}\\
\widehat{\mathbf{\pi}}_{n+1,n}\quad\downarrow &  & \downarrow\quad
\widehat{\mathbf{\pi}}_{n+1,n}\\
\quad\quad\left(  M\otimes\mathcal{W}_{D_{n}}\right)  _{\pi\left(  x\right)  }
& \overrightarrow{\ \ \ \ \ \ \ \hat{\pi}_{n+1,n}(\nabla_{x})\ \ \ \ \ \ \ } &
\left(  E\otimes\mathcal{W}_{D_{n}}\right)  _{x}%
\end{array}
\]
is commutative.
\end{proposition}

\begin{proof}
It is easy to see that the following four diagrams are commutative:
\[%
\begin{array}
[c]{ccc}%
M\otimes\mathcal{W}_{D_{n+1}} & \underrightarrow{\mathrm{id}_{M}%
\otimes\mathcal{W}_{\left(  d_{1},d_{2}\right)  \in D_{n+1}\times
D_{n+1}\mapsto d_{1}d_{2}\in D_{n+1}}} & M\otimes\mathcal{W}_{D_{n+1}\times
D_{n+1}}\\%
\begin{array}
[c]{cc}%
\mathrm{id}_{M}\otimes\mathcal{W}_{\mathbf{i}_{D_{n}\subseteq D_{n+1}}} &
\downarrow
\end{array}
&  &
\begin{array}
[c]{cc}%
\downarrow & \mathrm{id}_{M}\otimes\mathcal{W}_{\mathbf{i}_{D_{n+1}\times
D_{n}\subseteq D_{n+1}\times D_{n+1}}}%
\end{array}
\\
M\otimes\mathcal{W}_{D_{n}} & \overrightarrow{\mathrm{id}_{M}\otimes
\mathcal{W}_{\left(  d_{1},d_{2}\right)  \in D_{n+1}\times D_{n}\mapsto
d_{1}d_{2}\in D_{n}}} & M\otimes\mathcal{W}_{D_{n+1}\times D_{n}}%
\end{array}
\]
\[%
\begin{array}
[c]{ccc}%
M\otimes\mathcal{W}_{D_{n+1}\times D_{n+1}} & \underrightarrow{\nabla
_{x}\otimes\mathrm{id}_{\mathcal{W}_{D_{n+1}}}} & E\otimes\mathcal{W}%
_{D_{n+1}\times D_{n+1}}\\%
\begin{array}
[c]{cc}%
\mathrm{id}_{M}\otimes\mathcal{W}_{\mathbf{i}_{D_{n+1}\times D_{n}\subseteq
D_{n+1}\times D_{n+1}}} & \downarrow
\end{array}
&  &
\begin{array}
[c]{cc}%
\downarrow & \mathrm{id}_{E}\otimes\mathcal{W}_{\mathbf{i}_{D_{n+1}\times
D_{n}\subseteq D_{n+1}\times D_{n+1}}}%
\end{array}
\\
M\otimes\mathcal{W}_{D_{n+1}\times D_{n}} & \nabla_{x}\otimes\mathrm{id}%
_{\mathcal{W}_{D_{n}}} & E\otimes\mathcal{W}_{D_{n+1}\times D_{n}}%
\end{array}
\]
\begin{align*}
&
\begin{array}
[c]{ccc}%
M\otimes\mathcal{W}_{D_{n+1}} & \underrightarrow{\mathrm{id}_{M}%
\otimes\mathcal{W}_{\left(  d_{1},d_{2}\right)  \in D_{n+1}\times
D_{n+1}\mapsto d_{1}d_{2}\in D_{n+1}}} & M\otimes\mathcal{W}_{D_{n+1}\times
D_{n+1}}\\%
\begin{array}
[c]{cc}%
\nabla_{x} & \downarrow
\end{array}
&  &
\begin{array}
[c]{cc}%
\downarrow & \nabla_{x}\otimes\mathrm{id}_{\mathcal{W}_{D_{n+1}}}%
\end{array}
\\
E\otimes\mathcal{W}_{D_{n+1}} & \overrightarrow{\mathrm{id}_{E}\otimes
\mathcal{W}_{\left(  d_{1},d_{2}\right)  \in D_{n+1}\times D_{n+1}\mapsto
d_{1}d_{2}\in D_{n+1}}} & E\otimes\mathcal{W}_{D_{n+1}\times D_{n+1}}%
\end{array}
\\
&  \text{[By the second condition in Definition \ref{d5.1.1}]}%
\end{align*}
\[%
\begin{array}
[c]{ccc}%
E\otimes\mathcal{W}_{D_{n+1}} & \underrightarrow{\mathrm{id}_{E}%
\otimes\mathcal{W}_{\left(  d_{1},d_{2}\right)  \in D_{n+1}\times
D_{n+1}\mapsto d_{1}d_{2}\in D_{n+1}}} & E\otimes\mathcal{W}_{D_{n+1}\times
D_{n+1}}\\%
\begin{array}
[c]{cc}%
\mathrm{id}_{E}\otimes\mathcal{W}_{\mathbf{i}_{D_{n}\subseteq D_{n+1}}} &
\downarrow
\end{array}
&  &
\begin{array}
[c]{cc}%
\downarrow & \mathrm{id}_{E}\otimes\mathcal{W}_{\mathbf{i}_{D_{n+1}\times
D_{n}\subseteq D_{n+1}\times D_{n+1}}}%
\end{array}
\\
E\otimes\mathcal{W}_{D_{n}} & \overrightarrow{\mathrm{id}_{E}\otimes
\mathcal{W}_{\left(  d_{1},d_{2}\right)  \in D_{n+1}\times D_{n}\mapsto
d_{1}d_{2}\in D_{n}}} & E\otimes\mathcal{W}_{D_{n+1}\times D_{n}}%
\end{array}
\]
Therefore the composition of mappings
\begin{align*}
&  M\otimes\mathcal{W}_{D_{n+1}}\underrightarrow{\,\mathrm{id}_{M}%
\otimes\mathcal{W}_{\mathbf{i}_{D_{n}\subseteq D_{n+1}}}}\,M\otimes
\mathcal{W}_{D_{n}}\\
&  \underrightarrow{\mathrm{id}_{M}\otimes\mathcal{W}_{\left(  d_{1}%
,d_{2}\right)  \in D_{n+1}\times D_{n}\mapsto d_{1}d_{2}\in D_{n}}}%
\,M\otimes\mathcal{W}_{D_{n+1}\times D_{n}}\\
&  =\left(  M\otimes\mathcal{W}_{D_{n+1}}\right)  \otimes\mathcal{W}_{D_{n}%
}\underrightarrow{\,\nabla_{x}\otimes\mathrm{id}_{\mathcal{W}_{D_{n}}}%
}\,\left(  E\otimes\mathcal{W}_{D_{n+1}}\right)  \otimes\mathcal{W}_{D_{n}}\\
&  =E\otimes\mathcal{W}_{D_{n+1}\times D_{n}}%
\end{align*}
is equal to the composition of mappings
\begin{align*}
&  M\otimes\mathcal{W}_{D_{n+1}}\underrightarrow{\,\nabla_{x}}\,E\otimes
\mathcal{W}_{D_{n+1}}\underrightarrow{\,\mathrm{id}_{E}\otimes\mathcal{W}%
_{\mathbf{i}_{D_{n}\rightarrow D_{n+1}}}}\,E\otimes\mathcal{W}_{D_{n}}\\
&  \underrightarrow{\mathrm{id}_{E}\otimes\mathcal{W}_{\left(  d_{1}%
,d_{2}\right)  \in D_{n+1}\times D_{n}\mapsto d_{1}d_{2}\in D_{n}}}%
\,E\otimes\mathcal{W}_{D_{n+1}\times D_{n}}%
\end{align*}
which yields the coveted result.
\end{proof}

\begin{corollary}
\label{t5.1.4.1}Let $\nabla_{x}$ be a $D_{n+1}$-pseudotangential \textit{over
}the bundle $\pi:E\rightarrow M$ \textit{at} $x\in E$. For any $\gamma
,\gamma^{\prime}\in\left(  M\otimes\mathcal{W}_{D_{n+1}}\right)  _{\pi\left(
x\right)  }$, if
\[
\mathbf{\pi}_{n+1,n}\left(  \gamma\right)  =\mathbf{\pi}_{n+1,n}\left(
\gamma^{\prime}\right)
\]
then
\[
\mathbf{\pi}_{n+1,n}\left(  \nabla_{x}(\gamma)\right)  =\mathbf{\pi}%
_{n+1,n}\left(  \nabla_{x}(\gamma^{\prime})\right)
\]

\end{corollary}

\begin{proof}
By the above proposition, we have
\begin{align*}
&  \mathbf{\pi}_{n+1,n}(\nabla_{x}(\gamma))=\hat{\pi}_{n+1,n}(\nabla
_{x})(\mathbf{\pi}_{n+1,n}(\gamma))\\
&  =\hat{\pi}_{n+1,n}(\nabla_{x})(\mathbf{\pi}_{n+1,n}(\gamma^{\prime
}))=\mathbf{\pi}_{n+1,n}(\nabla_{x}(\gamma^{\prime}))\text{,}%
\end{align*}
which establishes the coveted proposition.
\end{proof}

\begin{definition}
The notion of a $D_{n}$-\textit{tangential} \textit{over }the bundle
$\pi:E\rightarrow M$ \textit{at} $x\in E$ is defined inductively on $n$. The
notion of a $D_{0}$-\textit{tangential} \textit{over }the bundle
$\pi:E\rightarrow M$ \textit{at} $x\in E$ and that of a $D_{1}$%
-\textit{tangential} \textit{over }the bundle $\pi:E\rightarrow M$ \textit{at}
$x\in E$ shall be identical with that of a $D_{0}$-pseudotangential
\textit{over }the bundle $\pi:E\rightarrow M$ \textit{at} $x\in E$ and that of
a $D_{1}$-pseudotangential \textit{over }the bundle $\pi:E\rightarrow M$
\textit{at} $x\in E$ respectively. Now we proceed by induction on $n$. A
$D_{n+1}$-pseudotangential $\nabla_{x}:\left(  M\otimes\mathcal{W}_{D_{n+1}%
}\right)  _{\pi\left(  x\right)  }\rightarrow\left(  E\otimes\mathcal{W}%
_{D_{n+1}}\right)  _{x}$ \textit{over }the bundle $\pi:E\rightarrow M$
\textit{at} $x\in E$ is called a $D_{n+1}$-\textit{tangential} \textit{over
}the bundle $\pi:E\rightarrow M$ \textit{at} $x\in E$ if it acquiesces in the
following two conditions:

\begin{enumerate}
\item $\hat{\pi}_{n+1,n}(\nabla_{x})$ is a $D_{n}$-tangential \textit{over
}the bundle $\pi:E\rightarrow M$ \textit{at} $x\in E$.

\item For any simple polynomial $\rho$ of $d\in D_{n+1}$ with $l=\mathrm{\dim
\,}\rho$ and any $\gamma\in\left(  M\otimes\mathcal{W}_{D_{l}}\right)
_{\pi\left(  x\right)  }$, we have
\[
\nabla_{x}(\gamma\circ\rho)=(\pi_{n+1,l}(\nabla_{x})(\gamma))\circ\rho
\]

\end{enumerate}
\end{definition}

\begin{notation}
We denote by $\mathbb{J}_{x}^{D_{n}}(\pi)$ the totality of $D_{n}$-tangentials
\textit{over }the bundle $\pi:E\rightarrow M$ \textit{at} $x\in E$, while we
denote by $\mathbb{J}^{D_{n}}(\pi)$\ the totality of $D_{n}$-tangentials
\textit{over }the bundle $\pi:E\rightarrow M$. By the very definition of a
$D_{n}$-tangential, the projection $\hat{\pi}_{n+1,n}:\widehat{\mathbb{J}%
}^{D_{n+1}}(\pi)\rightarrow\widehat{\mathbb{J}}^{D_{n}}(\pi)$ is naturally
restricted to a mapping $\pi_{n+1,n}:\mathbb{J}^{D_{n+1}}(\pi)\rightarrow
\mathbb{J}^{D_{n}}(\pi)$. Similarly for $\pi_{n,m}:\mathbb{J}^{D_{n}}%
(\pi)\rightarrow\mathbb{J}^{D_{m}}(\pi)$ with $m\leq n$.
\end{notation}

\section{\label{s6}From the First Approach to the Second}

\begin{definition}
Mappings $\varphi_{n}:\mathbf{J}^{n}(\pi)\rightarrow\mathbb{J}^{D^{n}}(\pi)$
$(n=0,1)$ shall be the identity mappings. We are going to define $\varphi
_{n}:\mathbf{J}^{n}(\pi)\rightarrow\mathbb{J}^{D^{n}}(\pi)\ $for any natural
number $n$ by induction on $n$. Let $x_{n}=\nabla_{x_{n-1}}\in\mathbf{J}%
^{n}(\pi)$ and $\nabla_{x_{n}}\in\mathbf{J}^{n+1}(\pi)$. We define
$\varphi_{n+1}(\nabla_{x_{n}})$ as the composition of mappings
\begin{align*}
&  \left(  M\otimes\mathcal{W}_{D^{n+1}}\right)  _{\pi\left(  x_{n}\right)
}\\
&  =\left(  \left(  M\otimes\mathcal{W}_{D^{n}}\right)  \otimes\mathcal{W}%
_{D}\right)  _{\left(  M\otimes\mathcal{W}_{D^{n}}\right)  _{\pi\left(
x_{n}\right)  }}\\
&  \underrightarrow{\left\langle \pi_{M}^{M\otimes\mathcal{W}_{D^{n}}}%
\otimes\mathrm{id}_{\mathcal{W}_{D}},\mathrm{id}_{\left(  M\otimes
\mathcal{W}_{D^{n}}\right)  \otimes\mathcal{W}_{D}}\right\rangle }\,\\
&  \left(  M\otimes\mathcal{W}_{D}\right)  _{\pi\left(  x_{n}\right)
}\underset{M\otimes\mathcal{W}_{D}}{\times}\left(  \left(  M\otimes
\mathcal{W}_{D^{n}}\right)  \otimes\mathcal{W}_{D}\right)  _{\left(
M\otimes\mathcal{W}_{D^{n}}\right)  _{\pi\left(  x_{n}\right)  }}\\
&  \underrightarrow{\nabla_{x_{n}}\times\mathrm{id}_{\left(  M\otimes
\mathcal{W}_{D^{n}}\right)  \otimes\mathcal{W}_{D}}}\,\\
&  \left(  \mathbf{J}^{n}(\pi)\otimes\mathcal{W}_{D}\right)  _{x_{n}}%
\underset{M\otimes\mathcal{W}_{D}}{\times}\left(  \left(  M\otimes
\mathcal{W}_{D^{n}}\right)  \otimes\mathcal{W}_{D}\right)  _{\left(
M\otimes\mathcal{W}_{D^{n}}\right)  _{\pi\left(  x_{n}\right)  }}\\
&  \underrightarrow{\left(  \varphi_{n}\otimes\mathrm{id}_{\mathcal{W}_{D}%
}\right)  \times\mathrm{id}_{\left(  M\otimes\mathcal{W}_{D^{n}}\right)
\otimes\mathcal{W}_{D}}}\,\\
&  \left(  \mathbb{J}^{D^{n}}(\pi)\otimes\mathcal{W}_{D}\right)  _{\varphi
_{n}\left(  x_{n}\right)  }\underset{M\otimes\mathcal{W}_{D}}{\times}\left(
\left(  M\otimes\mathcal{W}_{D^{n}}\right)  \otimes\mathcal{W}_{D}\right)
_{\left(  M\otimes\mathcal{W}_{D^{n}}\right)  _{\pi\left(  x_{n}\right)  }}\\
&  =\left(  \left(  \mathbb{J}^{D^{n}}(\pi)\underset{M}{\times}\left(
M\otimes\mathcal{W}_{D^{n}}\right)  \right)  \otimes\mathcal{W}_{D}\right)
_{\varphi_{n}\left(  x_{n}\right)  \times\left(  M\otimes\mathcal{W}_{D^{n}%
}\right)  _{\pi\left(  x_{n}\right)  }}\\
&  \underrightarrow{\left(  \left(  \nabla,\gamma\right)  \in\mathbb{J}%
^{D^{n}}(\pi)\underset{M}{\times}\left(  M\otimes\mathcal{W}_{D^{n}}\right)
\mapsto\nabla\left(  \gamma\right)  \in E\otimes\mathcal{W}_{D^{n}}\right)
\otimes\mathrm{id}_{\mathcal{W}_{D}}}\\
&  \left(  \left(  E\otimes\mathcal{W}_{D^{n}}\right)  \otimes\mathcal{W}%
_{D}\right)  _{\left(  E\otimes\mathcal{W}_{D^{n}}\right)  _{\pi_{0}\left(
x_{n}\right)  }}\\
&  =\left(  E\otimes\mathcal{W}_{D^{n+1}}\right)  _{\pi_{0}\left(
x_{n}\right)  }%
\end{align*}

\end{definition}

Surely we have to show that

\begin{lemma}
\label{t6.1.1}We have
\[
\varphi_{n+1}(\nabla_{x_{n}})\in\mathbb{\hat{J}}^{n+1}(\pi)
\]

\end{lemma}

\begin{proof}
We have to show that for any $\gamma\in\mathbf{T}_{\pi_{n}(x_{n})}^{n+1}%
(M)$,\ any $\alpha\in\mathbb{R}$ and any $\sigma\in\mathbf{S}_{n+1}$, we have
\begin{align}
\gamma &  =\left(  \pi\otimes\mathrm{id}_{\mathcal{W}_{D^{n+1}}}\right)
\circ\left(  \varphi_{n+1}(\nabla_{x_{n}})\right)  (\gamma)\label{6.1.1.1}\\
\varphi_{n+1}(\nabla_{x_{n}})(\alpha\underset{i}{\cdot}\gamma)  &
=\alpha\underset{i}{\cdot}\varphi_{n+1}(\nabla_{x_{n}})(\gamma)\quad(1\leq
i\leq n+1)\label{6.1.1.2}\\
\varphi_{n+1}(\nabla_{x_{n}})(\gamma^{\sigma})  &  =(\varphi_{n+1}%
(\nabla_{x_{n}})(\gamma))^{\sigma} \label{6.1.1.3}%
\end{align}
We proceed by induction on $n$.

\begin{enumerate}
\item First we deal with (\ref{6.1.1.1}). The mapping
\[
\left(  \pi\otimes\mathrm{id}_{\mathcal{W}_{D^{n+1}}}\right)  \left(
\varphi_{n+1}(\nabla_{x_{n}})\right)
\]
is the composition of mappings
\begin{align*}
&  \left(  M\otimes\mathcal{W}_{D^{n+1}}\right)  _{\pi\left(  x_{n}\right)
}\\
&  =\left(  \left(  M\otimes\mathcal{W}_{D^{n}}\right)  \otimes\mathcal{W}%
_{D}\right)  _{\left(  M\otimes\mathcal{W}_{D^{n}}\right)  _{\pi\left(
x_{n}\right)  }}\\
&  \underrightarrow{\left\langle \pi_{M}^{M\otimes\mathcal{W}_{D^{n}}}%
\otimes\mathrm{id}_{\mathcal{W}_{D}},\mathrm{id}_{\left(  M\otimes
\mathcal{W}_{D^{n}}\right)  \otimes\mathcal{W}_{D}}\right\rangle }\,\\
&  \left(  M\otimes\mathcal{W}_{D}\right)  _{\pi\left(  x_{n}\right)
}\underset{M\otimes\mathcal{W}_{D}}{\times}\left(  \left(  M\otimes
\mathcal{W}_{D^{n}}\right)  \otimes\mathcal{W}_{D}\right)  _{\left(
M\otimes\mathcal{W}_{D^{n}}\right)  _{\pi\left(  x_{n}\right)  }}\\
&  \underrightarrow{\nabla_{x_{n}}\times\mathrm{id}_{\left(  M\otimes
\mathcal{W}_{D^{n}}\right)  \otimes\mathcal{W}_{D}}}\,\\
&  \left(  \mathbf{J}^{n}(\pi)\otimes\mathcal{W}_{D}\right)  _{x_{n}}%
\underset{M\otimes\mathcal{W}_{D}}{\times}\left(  \left(  M\otimes
\mathcal{W}_{D^{n}}\right)  \otimes\mathcal{W}_{D}\right)  _{\left(
M\otimes\mathcal{W}_{D^{n}}\right)  _{\pi\left(  x_{n}\right)  }}\\
&  \underrightarrow{\left(  \varphi_{n}\otimes\mathrm{id}_{\mathcal{W}_{D}%
}\right)  \times\mathrm{id}_{\left(  M\otimes\mathcal{W}_{D^{n}}\right)
\otimes\mathcal{W}_{D}}}\,\\
&  \left(  \mathbb{J}^{D^{n}}(\pi)\otimes\mathcal{W}_{D}\right)  _{\varphi
_{n}\left(  x_{n}\right)  }\underset{M\otimes\mathcal{W}_{D}}{\times}\left(
\left(  M\otimes\mathcal{W}_{D^{n}}\right)  \otimes\mathcal{W}_{D}\right)
_{\left(  M\otimes\mathcal{W}_{D^{n}}\right)  _{\pi\left(  x_{n}\right)  }}\\
&  =\left(  \left(  \mathbb{J}^{D^{n}}(\pi)\underset{M}{\times}\left(
M\otimes\mathcal{W}_{D^{n}}\right)  \right)  \otimes\mathcal{W}_{D}\right)
_{\varphi_{n}\left(  x_{n}\right)  \times\left(  M\otimes\mathcal{W}_{D^{n}%
}\right)  _{\pi\left(  x_{n}\right)  }}\\
&  \underrightarrow{\left(  \left(  \nabla,\gamma\right)  \in\mathbb{J}%
^{D^{n}}(\pi)\underset{M}{\times}\left(  M\otimes\mathcal{W}_{D^{n}}\right)
\mapsto\nabla\left(  \gamma\right)  \in E\otimes\mathcal{W}_{D^{n}}\right)
\otimes\mathrm{id}_{\mathcal{W}_{D}}}\\
&  \left(  \left(  E\otimes\mathcal{W}_{D^{n}}\right)  \otimes\mathcal{W}%
_{D}\right)  _{\left(  E\otimes\mathcal{W}_{D^{n}}\right)  _{\pi_{0}\left(
x_{n}\right)  }}\\
&  =\left(  E\otimes\mathcal{W}_{D^{n+1}}\right)  _{\pi_{0}\left(
x_{n}\right)  }\,\underrightarrow{\,\pi\otimes\mathrm{id}_{\mathcal{W}%
_{D^{n+1}}}}\,\left(  M\otimes\mathcal{W}_{D^{n+1}}\right)  _{\pi\left(
x_{n}\right)  }%
\end{align*}
It is easy to see that the composition of mappings
\begin{align*}
&  \left(  \mathbf{J}^{n}(\pi)\otimes\mathcal{W}_{D}\right)  _{x_{n}}%
\underset{M\otimes\mathcal{W}_{D}}{\times}\left(  \left(  M\otimes
\mathcal{W}_{D^{n}}\right)  \otimes\mathcal{W}_{D}\right)  _{\left(
M\otimes\mathcal{W}_{D^{n}}\right)  _{\pi\left(  x_{n}\right)  }}\\
&  \underrightarrow{\left(  \varphi_{n}\otimes\mathrm{id}_{\mathcal{W}_{D}%
}\right)  \times\mathrm{id}_{\left(  M\otimes\mathcal{W}_{D^{n}}\right)
\otimes\mathcal{W}_{D}}}\,\\
&  \left(  \mathbb{J}^{D^{n}}(\pi)\otimes\mathcal{W}_{D}\right)  _{\varphi
_{n}\left(  x_{n}\right)  }\underset{M\otimes\mathcal{W}_{D}}{\times}\left(
\left(  M\otimes\mathcal{W}_{D^{n}}\right)  \otimes\mathcal{W}_{D}\right)
_{\left(  M\otimes\mathcal{W}_{D^{n}}\right)  _{\pi\left(  x_{n}\right)  }}\\
&  =\left(  \left(  \mathbb{J}^{D^{n}}(\pi)\underset{M}{\times}\left(
M\otimes\mathcal{W}_{D^{n}}\right)  \right)  \otimes\mathcal{W}_{D}\right)
_{\left\{  \varphi_{n}\left(  x_{n}\right)  \right\}  \times\left(
M\otimes\mathcal{W}_{D^{n}}\right)  _{\pi\left(  x_{n}\right)  }}\\
&  \underrightarrow{\left(  \left(  \nabla,\gamma\right)  \in\mathbb{J}%
^{D^{n}}(\pi)\underset{M}{\times}\left(  M\otimes\mathcal{W}_{D^{n}}\right)
\mapsto\nabla\left(  \gamma\right)  \in E\otimes\mathcal{W}_{D^{n}}\right)
\otimes\mathrm{id}_{\mathcal{W}_{D}}}\\
&  \left(  \left(  E\otimes\mathcal{W}_{D^{n}}\right)  \otimes\mathcal{W}%
_{D}\right)  _{\left(  E\otimes\mathcal{W}_{D^{n}}\right)  _{\pi_{0}\left(
x_{n}\right)  }}\\
&  =\left(  E\otimes\mathcal{W}_{D^{n+1}}\right)  _{\pi_{0}\left(
x_{n}\right)  }\,\underrightarrow{\,\pi\otimes\mathrm{id}_{\mathcal{W}%
_{D^{n+1}}}}\,\left(  M\otimes\mathcal{W}_{D^{n+1}}\right)  _{\pi\left(
x_{n}\right)  }%
\end{align*}
is no other than the canonical projection of
\[
\left(  \mathbf{J}^{n}(\pi)\otimes\mathcal{W}_{D}\right)  _{x_{n}}%
\underset{M\otimes\mathcal{W}_{D}}{\times}\left(  \left(  M\otimes
\mathcal{W}_{D^{n}}\right)  \otimes\mathcal{W}_{D}\right)  _{\left(
M\otimes\mathcal{W}_{D^{n}}\right)  _{\pi\left(  x_{n}\right)  }}%
\]
to the second factor $\left(  \left(  M\otimes\mathcal{W}_{D^{n}}\right)
\otimes\mathcal{W}_{D}\right)  _{\left(  M\otimes\mathcal{W}_{D^{n}}\right)
_{\pi\left(  x_{n}\right)  }}$. It is also easy to see that the composition of
mappings
\begin{align*}
&  \left(  \left(  M\otimes\mathcal{W}_{D^{n}}\right)  \otimes\mathcal{W}%
_{D}\right)  _{\left(  M\otimes\mathcal{W}_{D^{n}}\right)  _{\pi\left(
x_{n}\right)  }}\\
&  \underrightarrow{\left\langle \pi_{M}^{M\otimes\mathcal{W}_{D^{n}}}%
\otimes\mathrm{id}_{\mathcal{W}_{D}},\mathrm{id}_{\left(  M\otimes
\mathcal{W}_{D^{n}}\right)  \otimes\mathcal{W}_{D}}\right\rangle }\,\\
&  \left(  M\otimes\mathcal{W}_{D}\right)  _{\pi\left(  x_{n}\right)
}\underset{M\otimes\mathcal{W}_{D}}{\times}\left(  \left(  M\otimes
\mathcal{W}_{D^{n}}\right)  \otimes\mathcal{W}_{D}\right)  _{\left(
M\otimes\mathcal{W}_{D^{n}}\right)  _{\pi\left(  x_{n}\right)  }}\\
&  \underrightarrow{\nabla_{x_{n}}\times\mathrm{id}_{\left(  M\otimes
\mathcal{W}_{D^{n}}\right)  \otimes\mathcal{W}_{D}}}\,\\
&  \left(  \mathbf{J}^{n}(\pi)\otimes\mathcal{W}_{D}\right)  _{x_{n}}%
\underset{M\otimes\mathcal{W}_{D}}{\times}\left(  \left(  M\otimes
\mathcal{W}_{D^{n}}\right)  \otimes\mathcal{W}_{D}\right)  _{\left(
M\otimes\mathcal{W}_{D^{n}}\right)  _{\pi\left(  x_{n}\right)  }}\\
&  \underrightarrow{\left(  \varphi_{n}\otimes\mathrm{id}_{\mathcal{W}_{D}%
}\right)  \times\mathrm{id}_{\left(  M\otimes\mathcal{W}_{D^{n}}\right)
\otimes\mathcal{W}_{D}}}\,\\
&  \left(  \mathbb{J}^{D^{n}}(\pi)\otimes\mathcal{W}_{D}\right)  _{\varphi
_{n}\left(  x_{n}\right)  }\underset{M\otimes\mathcal{W}_{D}}{\times}\left(
\left(  M\otimes\mathcal{W}_{D^{n}}\right)  \otimes\mathcal{W}_{D}\right)
_{\left(  M\otimes\mathcal{W}_{D^{n}}\right)  _{\pi\left(  x_{n}\right)  }}%
\end{align*}
is
\begin{align*}
&  \left(  \left(  M\otimes\mathcal{W}_{D^{n}}\right)  \otimes\mathcal{W}%
_{D}\right)  _{\left(  M\otimes\mathcal{W}_{D^{n}}\right)  _{\pi\left(
x_{n}\right)  }}\\
&  \underrightarrow{\left\langle \left(  \varphi_{n}\otimes\mathrm{id}%
_{\mathcal{W}_{D}}\right)  \circ\nabla_{x_{n}}\circ\left(  \pi_{M}%
^{M\otimes\mathcal{W}_{D^{n}}}\otimes\mathrm{id}_{\mathcal{W}_{D}}\right)
,\mathrm{id}_{\left(  M\otimes\mathcal{W}_{D^{n}}\right)  \otimes
\mathcal{W}_{D}}\right\rangle }\\
&  \left(  \mathbb{J}^{D^{n}}(\pi)\otimes\mathcal{W}_{D}\right)  _{\varphi
_{n}\left(  x_{n}\right)  }\underset{M\otimes\mathcal{W}_{D}}{\times}\left(
\left(  M\otimes\mathcal{W}_{D^{n}}\right)  \otimes\mathcal{W}_{D}\right)
_{\left(  M\otimes\mathcal{W}_{D^{n}}\right)  _{\pi\left(  x_{n}\right)  }%
}\text{.}%
\end{align*}
Therefore (\ref{6.1.1.1}) follows at once.

\item Now we deal with (\ref{6.1.1.2}), the treatment of which is divided into
two cases, namely, $i\leq n$ and $i=n+1$. Since both of them are almost
trivial, they can safely be left to the reader.

\item Finally we must deal with (\ref{6.1.1.3}), for which it suffices to
consider only transpositions $\sigma=\left\langle i,i+1\right\rangle \;\left(
1\leq i\leq n\right)  $. Here we deal only with the most difficult case of
$\sigma=\left\langle n,n+1\right\rangle $. We consider the composition of
mappings
\begin{align}
&  \left(  M\otimes\mathcal{W}_{D^{n+1}}\right)  _{\pi\left(  x_{n}\right)
}\,\underrightarrow{\,\gamma\in\left(  M\otimes\mathcal{W}_{D^{n+1}}\right)
_{\pi\left(  x_{n}\right)  }\mapsto\gamma^{\left\langle n,n+1\right\rangle
}\in\left(  M\otimes\mathcal{W}_{D^{n+1}}\right)  _{\pi\left(  x_{n}\right)
}}\,\nonumber\\
&  \left(  M\otimes\mathcal{W}_{D^{n+1}}\right)  _{\pi\left(  x_{n}\right)
}\nonumber\\
&  =\left(  \left(  M\otimes\mathcal{W}_{D^{n}}\right)  \otimes\mathcal{W}%
_{D}\right)  _{\left(  M\otimes\mathcal{W}_{D^{n}}\right)  _{\pi\left(
x_{n}\right)  }}\nonumber\\
&  \underrightarrow{\left\langle \pi_{M}^{M\otimes\mathcal{W}_{D^{n}}}%
\otimes\mathrm{id}_{\mathcal{W}_{D}},\mathrm{id}_{\left(  M\otimes
\mathcal{W}_{D^{n}}\right)  \otimes\mathcal{W}_{D}}\right\rangle
}\,\nonumber\\
&  \left(  M\otimes\mathcal{W}_{D}\right)  _{\pi\left(  x_{n}\right)
}\underset{M\otimes\mathcal{W}_{D}}{\times}\left(  \left(  M\otimes
\mathcal{W}_{D^{n}}\right)  \otimes\mathcal{W}_{D}\right)  _{\left(
M\otimes\mathcal{W}_{D^{n}}\right)  _{\pi\left(  x_{n}\right)  }}\nonumber\\
&  \underrightarrow{\nabla_{x_{n}}\times\mathrm{id}_{\left(  M\otimes
\mathcal{W}_{D^{n}}\right)  \otimes\mathcal{W}_{D}}}\,\nonumber\\
&  \left(  \mathbf{J}^{n}(\pi)\otimes\mathcal{W}_{D}\right)  _{x_{n}}%
\underset{M\otimes\mathcal{W}_{D}}{\times}\left(  \left(  M\otimes
\mathcal{W}_{D^{n}}\right)  \otimes\mathcal{W}_{D}\right)  _{\left(
M\otimes\mathcal{W}_{D^{n}}\right)  _{\pi\left(  x_{n}\right)  }}\nonumber\\
&  \underrightarrow{\left(  \varphi_{n}\otimes\mathrm{id}_{\mathcal{W}_{D}%
}\right)  \times\mathrm{id}_{\left(  M\otimes\mathcal{W}_{D^{n}}\right)
\otimes\mathcal{W}_{D}}}\,\nonumber\\
&  \left(  \mathbb{J}^{D^{n}}(\pi)\otimes\mathcal{W}_{D}\right)  _{\varphi
_{n}\left(  x_{n}\right)  }\underset{M\otimes\mathcal{W}_{D}}{\times}\left(
\left(  M\otimes\mathcal{W}_{D^{n}}\right)  \otimes\mathcal{W}_{D}\right)
_{\left(  M\otimes\mathcal{W}_{D^{n}}\right)  _{\pi\left(  x_{n}\right)  }%
}\nonumber\\
&  =\left(  \left(  \mathbb{J}^{D^{n}}(\pi)\underset{M}{\times}\left(
M\otimes\mathcal{W}_{D^{n}}\right)  \right)  \otimes\mathcal{W}_{D}\right)
_{\varphi_{n}\left(  x_{n}\right)  \times\left(  M\otimes\mathcal{W}_{D^{n}%
}\right)  _{\pi\left(  x_{n}\right)  }}\nonumber\\
&  \underrightarrow{\left(  \left(  \nabla,\gamma\right)  \in\mathbb{J}%
^{D^{n}}(\pi)\underset{M}{\times}\left(  M\otimes\mathcal{W}_{D^{n}}\right)
\mapsto\nabla\left(  \gamma\right)  \in E\otimes\mathcal{W}_{D^{n}}\right)
\otimes\mathrm{id}_{\mathcal{W}_{D}}}\nonumber\\
&  \left(  \left(  E\otimes\mathcal{W}_{D^{n}}\right)  \otimes\mathcal{W}%
_{D}\right)  _{\left(  E\otimes\mathcal{W}_{D^{n}}\right)  _{\pi_{0}\left(
x_{n}\right)  }}\nonumber\\
&  =\left(  E\otimes\mathcal{W}_{D^{n+1}}\right)  _{\pi_{0}\left(
x_{n}\right)  } \label{6.1.1.4'}%
\end{align}
By the very definition of $\varphi_{n}$, the composition of mappings
\begin{align*}
&  \left(  \mathbf{J}^{n}(\pi)\otimes\mathcal{W}_{D}\right)  _{x_{n}}%
\underset{M\otimes\mathcal{W}_{D}}{\times}\left(  \left(  M\otimes
\mathcal{W}_{D^{n}}\right)  \otimes\mathcal{W}_{D}\right)  _{\left(
M\otimes\mathcal{W}_{D^{n}}\right)  _{\pi\left(  x_{n}\right)  }}\\
&  \underrightarrow{\left(  \varphi_{n}\otimes\mathrm{id}_{\mathcal{W}_{D}%
}\right)  \times\mathrm{id}_{\left(  M\otimes\mathcal{W}_{D^{n}}\right)
\otimes\mathcal{W}_{D}}}\,\\
&  \left(  \mathbb{J}^{D^{n}}(\pi)\otimes\mathcal{W}_{D}\right)  _{\varphi
_{n}\left(  x_{n}\right)  }\underset{M\otimes\mathcal{W}_{D}}{\times}\left(
\left(  M\otimes\mathcal{W}_{D^{n}}\right)  \otimes\mathcal{W}_{D}\right)
_{\left(  M\otimes\mathcal{W}_{D^{n}}\right)  _{\pi\left(  x_{n}\right)  }}\\
&  =\left(  \left(  \mathbb{J}^{D^{n}}(\pi)\underset{M}{\times}\left(
M\otimes\mathcal{W}_{D^{n}}\right)  \right)  \otimes\mathcal{W}_{D}\right)
_{\varphi_{n}\left(  x_{n}\right)  \times\left(  M\otimes\mathcal{W}_{D^{n}%
}\right)  _{\pi\left(  x_{n}\right)  }}\\
&  \underrightarrow{\left(  \left(  \nabla,\gamma\right)  \in\mathbb{J}%
^{D^{n}}(\pi)\underset{M}{\times}\left(  M\otimes\mathcal{W}_{D^{n}}\right)
\mapsto\nabla\left(  \gamma\right)  \in E\otimes\mathcal{W}_{D^{n}}\right)
\otimes\mathrm{id}_{\mathcal{W}_{D}}}\\
&  \left(  \left(  E\otimes\mathcal{W}_{D^{n}}\right)  \otimes\mathcal{W}%
_{D}\right)  _{\left(  E\otimes\mathcal{W}_{D^{n}}\right)  _{\pi_{0}\left(
x_{n}\right)  }}\\
&  =\left(  E\otimes\mathcal{W}_{D^{n+1}}\right)  _{\pi_{0}\left(
x_{n}\right)  }%
\end{align*}
is equivalent to the composition of mappings
\begin{align*}
&  \left(  \mathbf{J}^{n}(\pi)\otimes\mathcal{W}_{D}\right)  _{x_{n}}%
\underset{M\otimes\mathcal{W}_{D}}{\times}\left(  \left(  M\otimes
\mathcal{W}_{D^{n}}\right)  \otimes\mathcal{W}_{D}\right)  _{\left(
M\otimes\mathcal{W}_{D^{n}}\right)  _{\pi\left(  x_{n}\right)  }}\\
&  =\left(  \mathbf{J}^{n}(\pi)\otimes\mathcal{W}_{D}\right)  _{x_{n}%
}\underset{M\otimes\mathcal{W}_{D}}{\times}\\
&  \left(  \left(  \left(  M\otimes\mathcal{W}_{D^{n-1}}\right)
\otimes\mathcal{W}_{D}\right)  \otimes\mathcal{W}_{D}\right)  _{\left(
\left(  M\otimes\mathcal{W}_{D^{n-1}}\right)  \otimes\mathcal{W}_{D}\right)
_{\left(  M\otimes\mathcal{W}_{D^{n-1}}\right)  _{\pi\left(  x_{n}\right)  }}%
}\\
&  =\left(  \left(  \mathbf{J}^{n}(\pi)\underset{M}{\times}\left(  \left(
M\otimes\mathcal{W}_{D^{n-1}}\right)  \otimes\mathcal{W}_{D}\right)  \right)
\otimes\mathcal{W}_{D}\right)  _{\ast}\\
\text{\lbrack}\ast &  =x_{n}\times\left(  \left(  M\otimes\mathcal{W}%
_{D^{n-1}}\right)  \otimes\mathcal{W}_{D}\right)  _{\left(  M\otimes
\mathcal{W}_{D^{n-1}}\right)  _{\pi\left(  x_{n}\right)  }}\text{]}\\
&  \underrightarrow{\left(  \mathrm{id}_{\mathbf{J}^{n}(\pi)}\times
\left\langle \pi_{M}^{M\otimes\mathcal{W}_{D^{n-1}}}\otimes\mathrm{id}%
_{\mathcal{W}_{D}},\mathrm{id}_{\left(  M\otimes\mathcal{W}_{D^{n-1}}\right)
\otimes\mathcal{W}_{D}}\right\rangle \right)  \otimes\mathrm{id}%
_{\mathcal{W}_{D}}}\\
&  \left(  \left(  \mathbf{J}^{n}(\pi)\underset{M}{\times}\left(
M\otimes\mathcal{W}_{D}\right)  \underset{M}{\times}\left(  \left(
M\otimes\mathcal{W}_{D^{n-1}}\right)  \otimes\mathcal{W}_{D}\right)  \right)
\otimes\mathcal{W}_{D}\right)  _{\ast}\\
\text{\lbrack}\ast &  =x_{n}\times\pi\left(  x_{n}\right)  \times\left(
\left(  M\otimes\mathcal{W}_{D^{n-1}}\right)  \otimes\mathcal{W}_{D}\right)
_{\left(  M\otimes\mathcal{W}_{D^{n-1}}\right)  _{\pi\left(  x_{n}\right)  }%
}\text{]}\\
&  \underrightarrow{\left(  \left(
\begin{array}
[c]{c}%
\left(  \nabla,t\right)  \in\mathbf{J}^{n}(\pi)\times\left(  M\otimes
\mathcal{W}_{D}\right)  \mapsto\\
\nabla\left(  t\right)  \in\mathbf{J}^{n-1}(\pi)\otimes\mathcal{W}_{D}%
\end{array}
\right)  \times\mathrm{id}_{\left(  \left(  M\otimes\mathcal{W}_{D^{n-1}%
}\right)  \otimes\mathcal{W}_{D}\right)  }\right)  \otimes\mathrm{id}%
_{\mathcal{W}_{D}}}\\
&  \,\left(  \left(  \left(  \mathbf{J}^{n-1}(\pi)\otimes\mathcal{W}%
_{D}\right)  \underset{M}{\times}\left(  \left(  M\otimes\mathcal{W}_{D^{n-1}%
}\right)  \otimes\mathcal{W}_{D}\right)  \right)  \otimes\mathcal{W}%
_{D}\right)  _{\ast}\\
\text{\lbrack}\ast &  =\left(  \mathbf{J}^{n-1}(\pi)\otimes\mathcal{W}%
_{D}\right)  _{\pi_{n-1}\left(  x_{n}\right)  }\times\left(  \left(
M\otimes\mathcal{W}_{D^{n-1}}\right)  \otimes\mathcal{W}_{D}\right)  _{\left(
M\otimes\mathcal{W}_{D^{n-1}}\right)  _{\pi\left(  x_{n}\right)  }}\text{]}\\
&  =\left(  \left(  \mathbf{J}^{n-1}(\pi)\underset{M}{\times}\left(
M\otimes\mathcal{W}_{D^{n-1}}\right)  \right)  \otimes\mathcal{W}_{D^{2}%
}\right)  _{\pi_{n-1}\left(  x_{n}\right)  \times\left(  M\otimes
\mathcal{W}_{D^{n-1}}\right)  _{\pi\left(  x_{n}\right)  }}\\
&  \underrightarrow{\varphi_{n-1}\times\mathrm{id}_{\left(  M\otimes
\mathcal{W}_{D^{n}}\right)  \otimes\mathcal{W}_{D}}}\,\\
&  \left(  \left(  \mathbb{J}^{D^{n-1}}(\pi)\underset{M}{\times}\left(
M\otimes\mathcal{W}_{D^{n-1}}\right)  \right)  \otimes\mathcal{W}_{D^{2}%
}\right)  _{\pi_{0}\left(  x_{n}\right)  \times\left(  M\otimes\mathcal{W}%
_{D^{n-1}}\right)  _{\pi\left(  x_{n}\right)  }}\\
&  \underrightarrow{\left(  \left(  \nabla,\gamma\right)  \in\mathbb{J}%
^{D^{n-1}}(\pi)\times\left(  M\otimes\mathcal{W}_{D^{n-1}}\right)
\mapsto\nabla\left(  \gamma\right)  \in E\otimes\mathcal{W}_{D^{n-1}}\right)
\otimes\mathrm{id}_{\mathcal{W}_{D^{2}}}}\\
&  \left(  \left(  E\otimes\mathcal{W}_{D^{n-1}}\right)  \otimes
\mathcal{W}_{D^{2}}\right)  _{\left(  E\otimes\mathcal{W}_{D^{n-1}}\right)
_{\pi_{0}\left(  x_{n}\right)  }}\\
&  =\left(  E\otimes\mathcal{W}_{D^{n+1}}\right)  _{\pi_{0}\left(
x_{n}\right)  }%
\end{align*}
Therefore (\ref{6.1.1.4'}) is no other than the composition of mappings
\begin{align*}
&  \left(  M\otimes\mathcal{W}_{D^{n+1}}\right)  _{\pi\left(  x_{n}\right)
}\\
&  \,\underrightarrow{\,\gamma\in\left(  M\otimes\mathcal{W}_{D^{n+1}}\right)
_{\pi\left(  x_{n}\right)  }\mapsto\gamma^{\left\langle n,n+1\right\rangle
}\in\left(  M\otimes\mathcal{W}_{D^{n+1}}\right)  _{\pi\left(  x_{n}\right)
}}\,\\
&  \left(  M\otimes\mathcal{W}_{D^{n+1}}\right)  _{\pi\left(  x_{n}\right)
}\\
&  =\left(  \left(  M\otimes\mathcal{W}_{D^{n}}\right)  \otimes\mathcal{W}%
_{D}\right)  _{\left(  M\otimes\mathcal{W}_{D^{n}}\right)  _{\pi\left(
x_{n}\right)  }}\\
&  \underrightarrow{\left\langle \pi_{M}^{M\otimes\mathcal{W}_{D^{n}}}%
\otimes\mathrm{id}_{\mathcal{W}_{D}},\mathrm{id}_{\left(  M\otimes
\mathcal{W}_{D^{n}}\right)  \otimes\mathcal{W}_{D}}\right\rangle }\,\\
&  \left(  M\otimes\mathcal{W}_{D}\right)  _{\pi\left(  x_{n}\right)
}\underset{M\otimes\mathcal{W}_{D}}{\times}\left(  \left(  M\otimes
\mathcal{W}_{D^{n}}\right)  \otimes\mathcal{W}_{D}\right)  _{\left(
M\otimes\mathcal{W}_{D^{n}}\right)  _{\pi\left(  x_{n}\right)  }}\\
&  \underrightarrow{\nabla_{x_{n}}\times\mathrm{id}_{\left(  M\otimes
\mathcal{W}_{D^{n}}\right)  \otimes\mathcal{W}_{D}}}\,\\
&  \left(  \mathbf{J}^{n}(\pi)\otimes\mathcal{W}_{D}\right)  _{x_{n}}%
\underset{M\otimes\mathcal{W}_{D}}{\times}\left(  \left(  M\otimes
\mathcal{W}_{D^{n}}\right)  \otimes\mathcal{W}_{D}\right)  _{\left(
M\otimes\mathcal{W}_{D^{n}}\right)  _{\pi\left(  x_{n}\right)  }}\\
&  =\left(  \mathbf{J}^{n}(\pi)\otimes\mathcal{W}_{D}\right)  _{x_{n}%
}\underset{M\otimes\mathcal{W}_{D}}{\times}\\
&  \left(  \left(  \left(  M\otimes\mathcal{W}_{D^{n-1}}\right)
\otimes\mathcal{W}_{D}\right)  \otimes\mathcal{W}_{D}\right)  _{\left(
\left(  M\otimes\mathcal{W}_{D^{n-1}}\right)  \otimes\mathcal{W}_{D}\right)
_{\left(  M\otimes\mathcal{W}_{D^{n-1}}\right)  _{\pi\left(  x_{n}\right)  }}%
}\\
&  =\left(  \left(  \mathbf{J}^{n}(\pi)\underset{M}{\times}\left(  \left(
M\otimes\mathcal{W}_{D^{n-1}}\right)  \otimes\mathcal{W}_{D}\right)  \right)
\otimes\mathcal{W}_{D}\right)  _{\ast}\\
\text{\lbrack}\ast &  =x_{n}\times\left(  \left(  M\otimes\mathcal{W}%
_{D^{n-1}}\right)  \otimes\mathcal{W}_{D}\right)  _{\left(  M\otimes
\mathcal{W}_{D^{n-1}}\right)  _{\pi\left(  x_{n}\right)  }}\text{]}\\
&  \underrightarrow{\left(  \mathrm{id}_{\mathbf{J}^{n}(\pi)}\times
\left\langle \pi_{M}^{M\otimes\mathcal{W}_{D^{n-1}}}\otimes\mathrm{id}%
_{\mathcal{W}_{D}},\mathrm{id}_{\left(  M\otimes\mathcal{W}_{D^{n-1}}\right)
\otimes\mathcal{W}_{D}}\right\rangle \right)  \otimes\mathrm{id}%
_{\mathcal{W}_{D}}}\\
&  \left(  \left(  \mathbf{J}^{n}(\pi)\underset{M}{\times}\left(
M\otimes\mathcal{W}_{D}\right)  \underset{M}{\times}\left(  \left(
M\otimes\mathcal{W}_{D^{n-1}}\right)  \otimes\mathcal{W}_{D}\right)  \right)
\otimes\mathcal{W}_{D}\right)  _{\ast}\\
\text{\lbrack}\ast &  =x_{n}\times\pi\left(  x_{n}\right)  \times\left(
\left(  M\otimes\mathcal{W}_{D^{n-1}}\right)  \otimes\mathcal{W}_{D}\right)
_{\left(  M\otimes\mathcal{W}_{D^{n-1}}\right)  _{\pi\left(  x_{n}\right)  }%
}\text{]}\\
&  \underrightarrow{\left(  \left(
\begin{array}
[c]{c}%
\left(  \nabla,t\right)  \in\mathbf{J}^{n}(\pi)\times\left(  M\otimes
\mathcal{W}_{D}\right)  \mapsto\\
\nabla\left(  t\right)  \in\mathbf{J}^{n-1}(\pi)\otimes\mathcal{W}_{D}%
\end{array}
\right)  \times\mathrm{id}_{\left(  \left(  M\otimes\mathcal{W}_{D^{n-1}%
}\right)  \otimes\mathcal{W}_{D}\right)  }\right)  \otimes\mathrm{id}%
_{\mathcal{W}_{D}}}\\
&  \left(  \left(  \left(  \mathbf{J}^{n-1}(\pi)\otimes\mathcal{W}_{D}\right)
\underset{M}{\times}\left(  \left(  M\otimes\mathcal{W}_{D^{n-1}}\right)
\otimes\mathcal{W}_{D}\right)  \right)  \otimes\mathcal{W}_{D}\right)  _{\ast
}\\
\text{\lbrack}\ast &  =\left(  \mathbf{J}^{n-1}(\pi)\otimes\mathcal{W}%
_{D}\right)  _{\pi_{n-1}\left(  x_{n}\right)  }\times\left(  \left(
M\otimes\mathcal{W}_{D^{n-1}}\right)  \otimes\mathcal{W}_{D}\right)  _{\left(
M\otimes\mathcal{W}_{D^{n-1}}\right)  _{\pi\left(  x_{n}\right)  }}\text{]}\\
&  =\left(  \left(  \mathbf{J}^{n-1}(\pi)\times\left(  M\otimes\mathcal{W}%
_{D^{n-1}}\right)  \right)  \otimes\mathcal{W}_{D^{2}}\right)  _{\pi
_{n-1}\left(  x_{n}\right)  \times\left(  M\otimes\mathcal{W}_{D^{n-1}%
}\right)  _{\pi\left(  x_{n}\right)  }}\\
&  \underrightarrow{\varphi_{n-1}\times\mathrm{id}_{\left(  M\otimes
\mathcal{W}_{D^{n}}\right)  \otimes\mathcal{W}_{D}}}\,\\
&  \left(  \left(  \mathbb{J}^{D^{n-1}}(\pi)\underset{M}{\times}\left(
M\otimes\mathcal{W}_{D^{n-1}}\right)  \right)  \otimes\mathcal{W}_{D^{2}%
}\right)  _{\pi_{0}\left(  x_{n}\right)  \times\left(  M\otimes\mathcal{W}%
_{D^{n-1}}\right)  _{\pi\left(  x_{n}\right)  }}\\
&  \underrightarrow{\left(  \left(  \nabla,\gamma\right)  \in\mathbb{J}%
^{D^{n-1}}(\pi)\times\left(  M\otimes\mathcal{W}_{D^{n-1}}\right)
\mapsto\nabla\left(  \gamma\right)  \in E\otimes\mathcal{W}_{D^{n-1}}\right)
\otimes\mathrm{id}_{\mathcal{W}_{D^{2}}}}\\
&  \left(  \left(  E\otimes\mathcal{W}_{D^{n-1}}\right)  \otimes
\mathcal{W}_{D^{2}}\right)  _{\left(  E\otimes\mathcal{W}_{D^{n-1}}\right)
_{\pi_{0}\left(  x_{n}\right)  }}\\
&  =\left(  E\otimes\mathcal{W}_{D^{n+1}}\right)  _{\pi_{0}\left(
x_{n}\right)  }%
\end{align*}
On the other hand, the composition of mappings
\begin{align*}
&  \left(  M\otimes\mathcal{W}_{D^{n+1}}\right)  _{\pi\left(  x_{n}\right)
}\\
&  =\left(  \left(  M\otimes\mathcal{W}_{D^{n}}\right)  \otimes\mathcal{W}%
_{D}\right)  _{\left(  M\otimes\mathcal{W}_{D^{n}}\right)  _{\pi\left(
x_{n}\right)  }}\\
&  \underrightarrow{\left\langle \pi_{M}^{M\otimes\mathcal{W}_{D^{n}}}%
\otimes\mathrm{id}_{\mathcal{W}_{D}},\mathrm{id}_{\left(  M\otimes
\mathcal{W}_{D^{n}}\right)  \otimes\mathcal{W}_{D}}\right\rangle }\,\\
&  \left(  M\otimes\mathcal{W}_{D}\right)  _{\pi\left(  x_{n}\right)
}\underset{M\otimes\mathcal{W}_{D}}{\times}\left(  \left(  M\otimes
\mathcal{W}_{D^{n}}\right)  \otimes\mathcal{W}_{D}\right)  _{\left(
M\otimes\mathcal{W}_{D^{n}}\right)  _{\pi\left(  x_{n}\right)  }}\\
&  \underrightarrow{\nabla_{x_{n}}\times\mathrm{id}_{\left(  M\otimes
\mathcal{W}_{D^{n}}\right)  \otimes\mathcal{W}_{D}}}\,\\
&  \left(  \mathbf{J}^{n}(\pi)\otimes\mathcal{W}_{D}\right)  _{x_{n}}%
\underset{M\otimes\mathcal{W}_{D}}{\times}\left(  \left(  M\otimes
\mathcal{W}_{D^{n}}\right)  \otimes\mathcal{W}_{D}\right)  _{\left(
M\otimes\mathcal{W}_{D^{n}}\right)  _{\pi\left(  x_{n}\right)  }}\\
&  \underrightarrow{\left(  \varphi_{n}\otimes\mathrm{id}_{\mathcal{W}_{D}%
}\right)  \times\mathrm{id}_{\left(  M\otimes\mathcal{W}_{D^{n}}\right)
\otimes\mathcal{W}_{D}}}\,\\
&  \left(  \mathbb{J}^{D^{n}}(\pi)\otimes\mathcal{W}_{D}\right)  _{\varphi
_{n}\left(  x_{n}\right)  }\underset{M\otimes\mathcal{W}_{D}}{\times}\left(
\left(  M\otimes\mathcal{W}_{D^{n}}\right)  \otimes\mathcal{W}_{D}\right)
_{\left(  M\otimes\mathcal{W}_{D^{n}}\right)  _{\pi\left(  x_{n}\right)  }}\\
&  =\left(  \left(  \mathbb{J}^{D^{n}}(\pi)\underset{M}{\times}\left(
M\otimes\mathcal{W}_{D^{n}}\right)  \right)  \otimes\mathcal{W}_{D}\right)
_{\varphi_{n}\left(  x_{n}\right)  \times\left(  M\otimes\mathcal{W}_{D^{n}%
}\right)  _{\pi\left(  x_{n}\right)  }}\\
&  \underrightarrow{\left(  \left(  \nabla,\gamma\right)  \in\mathbb{J}%
^{D^{n}}(\pi)\underset{M}{\times}\left(  M\otimes\mathcal{W}_{D^{n}}\right)
\mapsto\nabla\left(  \gamma\right)  \in E\otimes\mathcal{W}_{D^{n}}\right)
\otimes\mathrm{id}_{\mathcal{W}_{D}}}\\
&  \left(  \left(  E\otimes\mathcal{W}_{D^{n}}\right)  \otimes\mathcal{W}%
_{D}\right)  _{\left(  E\otimes\mathcal{W}_{D^{n}}\right)  _{\pi_{0}\left(
x_{n}\right)  }}\\
&  =\left(  E\otimes\mathcal{W}_{D^{n+1}}\right)  _{\pi_{0}\left(
x_{n}\right)  }\\
&  \underrightarrow{\,\gamma\in E\otimes\mathcal{W}_{D^{n+1}}\mapsto
\gamma^{\left\langle n,n+1\right\rangle }\in E\otimes\mathcal{W}_{D^{n+1}}%
}\,\,\left(  E\otimes\mathcal{W}_{D^{n+1}}\right)  _{\pi_{0}\left(
x_{n}\right)  }%
\end{align*}
is the composition of mappings
\begin{align*}
&  \left(  M\otimes\mathcal{W}_{D^{n+1}}\right)  _{\pi\left(  x_{n}\right)
}\\
&  =\left(  \left(  M\otimes\mathcal{W}_{D^{n}}\right)  \otimes\mathcal{W}%
_{D}\right)  _{\left(  M\otimes\mathcal{W}_{D^{n}}\right)  _{\pi\left(
x_{n}\right)  }}\\
&  \underrightarrow{\left\langle \pi_{M}^{M\otimes\mathcal{W}_{D^{n}}}%
\otimes\mathrm{id}_{\mathcal{W}_{D}},\mathrm{id}_{\left(  M\otimes
\mathcal{W}_{D^{n}}\right)  \otimes\mathcal{W}_{D}}\right\rangle }\,\\
&  \left(  M\otimes\mathcal{W}_{D}\right)  _{\pi\left(  x_{n}\right)
}\underset{M\otimes\mathcal{W}_{D}}{\times}\left(  \left(  M\otimes
\mathcal{W}_{D^{n}}\right)  \otimes\mathcal{W}_{D}\right)  _{\left(
M\otimes\mathcal{W}_{D^{n}}\right)  _{\pi\left(  x_{n}\right)  }}\\
&  \underrightarrow{\nabla_{x_{n}}\times\mathrm{id}_{\left(  M\otimes
\mathcal{W}_{D^{n}}\right)  \otimes\mathcal{W}_{D}}}\,\\
&  \left(  \mathbf{J}^{n}(\pi)\otimes\mathcal{W}_{D}\right)  _{x_{n}}%
\underset{M\otimes\mathcal{W}_{D}}{\times}\left(  \left(  M\otimes
\mathcal{W}_{D^{n}}\right)  \otimes\mathcal{W}_{D}\right)  _{\left(
M\otimes\mathcal{W}_{D^{n}}\right)  _{\pi\left(  x_{n}\right)  }}\\
&  =\left(  \mathbf{J}^{n}(\pi)\otimes\mathcal{W}_{D}\right)  _{x_{n}%
}\underset{M\otimes\mathcal{W}_{D}}{\times}\\
&  \left(  \left(  \left(  M\otimes\mathcal{W}_{D^{n-1}}\right)
\otimes\mathcal{W}_{D}\right)  \otimes\mathcal{W}_{D}\right)  _{\left(
\left(  M\otimes\mathcal{W}_{D^{n-1}}\right)  \otimes\mathcal{W}_{D}\right)
_{\left(  M\otimes\mathcal{W}_{D^{n-1}}\right)  _{\pi\left(  x_{n}\right)  }}%
}\\
&  =\left(  \left(  \mathbf{J}^{n}(\pi)\times\left(  \left(  M\otimes
\mathcal{W}_{D^{n-1}}\right)  \otimes\mathcal{W}_{D}\right)  \right)
\otimes\mathcal{W}_{D}\right)  _{\ast}\\
\text{ [}\ast &  =x_{n}\times\left(  \left(  M\otimes\mathcal{W}_{D^{n-1}%
}\right)  \otimes\mathcal{W}_{D}\right)  _{\left(  M\otimes\mathcal{W}%
_{D^{n-1}}\right)  _{\pi\left(  x_{n}\right)  }}\text{]}\\
&  \underrightarrow{\left(  \mathrm{id}_{\mathbf{J}^{n}(\pi)}\times
\left\langle \pi_{M}^{M\otimes\mathcal{W}_{D^{n-1}}}\otimes\mathrm{id}%
_{\mathcal{W}_{D}},\mathrm{id}_{\left(  M\otimes\mathcal{W}_{D^{n-1}}\right)
\otimes\mathcal{W}_{D}}\right\rangle \right)  \otimes\mathrm{id}%
_{\mathcal{W}_{D}}}\\
&  \left(  \left(  \mathbf{J}^{n}(\pi)\underset{M}{\times}\left(
M\otimes\mathcal{W}_{D}\right)  \underset{M}{\times}\left(  \left(
M\otimes\mathcal{W}_{D^{n-1}}\right)  \otimes\mathcal{W}_{D}\right)  \right)
\otimes\mathcal{W}_{D}\right)  _{\ast}\\
\text{\lbrack}\ast &  =x_{n}\times\pi\left(  x_{n}\right)  \times\left(
\left(  M\otimes\mathcal{W}_{D^{n-1}}\right)  \otimes\mathcal{W}_{D}\right)
_{\left(  M\otimes\mathcal{W}_{D^{n-1}}\right)  _{\pi\left(  x_{n}\right)  }%
}\text{]}\\
&  \underline{\left(  \left(
\begin{array}
[c]{c}%
\left(  \nabla,t\right)  \in\mathbf{J}^{n}(\pi)\times\left(  M\otimes
\mathcal{W}_{D}\right)  \mapsto\\
\nabla\left(  t\right)  \in\mathbf{J}^{n-1}(\pi)\otimes\mathcal{W}_{D}%
\end{array}
\right)  \times\mathrm{id}_{\left(  \left(  M\otimes\mathcal{W}_{D^{n-1}%
}\right)  \otimes\mathcal{W}_{D}\right)  }\right)  }\\
&  \underrightarrow{\otimes\mathrm{id}_{\mathcal{W}_{D}}}\,\\
&  \left(  \left(  \left(  \mathbf{J}^{n-1}(\pi)\otimes\mathcal{W}_{D}\right)
\underset{M\otimes\mathcal{W}_{D}}{\times}\left(  \left(  M\otimes
\mathcal{W}_{D^{n-1}}\right)  \otimes\mathcal{W}_{D}\right)  \right)
\otimes\mathcal{W}_{D}\right)  _{\ast}\\
\text{\lbrack}\ast &  =\left(  \mathbf{J}^{n-1}(\pi)\otimes\mathcal{W}%
_{D}\right)  _{\pi_{n-1}\left(  x_{n}\right)  }\underset{M\otimes
\mathcal{W}_{D}}{\times}\left(  \left(  M\otimes\mathcal{W}_{D^{n-1}}\right)
\otimes\mathcal{W}_{D}\right)  _{\left(  M\otimes\mathcal{W}_{D^{n-1}}\right)
_{\pi\left(  x_{n}\right)  }}\text{]}\\
&  =\left(  \left(  \mathbf{J}^{n-1}(\pi)\times\left(  M\otimes\mathcal{W}%
_{D^{n-1}}\right)  \right)  \otimes\mathcal{W}_{D^{2}}\right)  _{\pi
_{n-1}\left(  x_{n}\right)  \times\left(  M\otimes\mathcal{W}_{D^{n-1}%
}\right)  _{\pi\left(  x_{n}\right)  }}%
\end{align*}
followed by the composition of mappings
\begin{align*}
&  \left(  \left(  \mathbf{J}^{n-1}(\pi)\times\left(  M\otimes\mathcal{W}%
_{D^{n-1}}\right)  \right)  \otimes\mathcal{W}_{D^{2}}\right)  _{\pi
_{n-1}\left(  x_{n}\right)  \times\left(  M\otimes\mathcal{W}_{D^{n-1}%
}\right)  _{\pi\left(  x_{n}\right)  }}\\
&  \underrightarrow{\varphi_{n-1}\times\mathrm{id}_{\left(  M\otimes
\mathcal{W}_{D^{n}}\right)  \otimes\mathcal{W}_{D}}}\,\\
&  \left(  \left(  \mathbb{J}^{D^{n-1}}(\pi)\underset{M}{\times}\left(
M\otimes\mathcal{W}_{D^{n-1}}\right)  \right)  \otimes\mathcal{W}_{D^{2}%
}\right)  _{\pi_{0}\left(  x_{n}\right)  \times\left(  M\otimes\mathcal{W}%
_{D^{n-1}}\right)  _{\pi\left(  x_{n}\right)  }}\\
&  \underrightarrow{\left(  \left(  \nabla,\gamma\right)  \in\mathbb{J}%
^{n-1}(\pi)\times\left(  M\otimes\mathcal{W}_{D^{n-1}}\right)  \mapsto
\nabla\left(  \gamma\right)  \in E\otimes\mathcal{W}_{D^{n-1}}\right)
\otimes\mathrm{id}_{\mathcal{W}_{D^{2}}}}\\
&  \left(  E\otimes\mathcal{W}_{D^{n-1}}\right)  \otimes\mathcal{W}_{D^{2}%
}=\left(  E\otimes\mathcal{W}_{D^{n+1}}\right)  _{\pi_{0}\left(  x_{n}\right)
}\\
&  \underrightarrow{\,\gamma\in E\otimes\mathcal{W}_{D^{n+1}}\mapsto
\gamma^{\left\langle n,n+1\right\rangle }\in E\otimes\mathcal{W}_{D^{n+1}}%
}\,\,\left(  E\otimes\mathcal{W}_{D^{n+1}}\right)  _{\pi_{0}\left(
x_{n}\right)  }\text{,}%
\end{align*}
which is easily seen to be equivalent to the composition of mappings
\begin{align*}
&  \left(  \left(  \mathbf{J}^{n-1}(\pi)\times\left(  M\otimes\mathcal{W}%
_{D^{n-1}}\right)  \right)  \otimes\mathcal{W}_{D^{2}}\right)  _{\pi
_{n-1}\left(  x_{n}\right)  \times\left(  M\otimes\mathcal{W}_{D^{n-1}%
}\right)  _{\pi\left(  x_{n}\right)  }}\\
&  \underrightarrow{\mathrm{id}_{\mathbf{J}^{n-1}(\pi)\times\left(
M\otimes\mathcal{W}_{D^{n-1}}\right)  }\otimes\mathcal{W}_{\left(  d_{1}%
.d_{2}\right)  \in D^{2}\mapsto\left(  d_{2}.d_{1}\right)  \in D^{2}}}\,\\
&  \left(  \left(  \mathbf{J}^{n-1}(\pi)\times\left(  M\otimes\mathcal{W}%
_{D^{n-1}}\right)  \right)  \otimes\mathcal{W}_{D^{2}}\right)  _{\pi
_{n-1}\left(  x_{n}\right)  \times\left(  M\otimes\mathcal{W}_{D^{n-1}%
}\right)  _{\pi\left(  x_{n}\right)  }}\\
&  \underrightarrow{\varphi_{n-1}\times\mathrm{id}_{\left(  M\otimes
\mathcal{W}_{D^{n}}\right)  \otimes\mathcal{W}_{D}}}\,\\
&  \left(  \left(  \mathbb{J}^{D^{n-1}}(\pi)\underset{M}{\times}\left(
M\otimes\mathcal{W}_{D^{n-1}}\right)  \right)  \otimes\mathcal{W}_{D^{2}%
}\right)  _{\pi_{0}\left(  x_{n}\right)  \times\left(  M\otimes\mathcal{W}%
_{D^{n-1}}\right)  _{\pi\left(  x_{n}\right)  }}\\
&  \underrightarrow{\left(  \left(  \nabla,\gamma\right)  \in\mathbb{J}%
^{D^{n-1}}(\pi)\times\left(  M\otimes\mathcal{W}_{D^{n-1}}\right)
\mapsto\nabla\left(  \gamma\right)  \in E\otimes\mathcal{W}_{D^{n-1}}\right)
\otimes\mathrm{id}_{\mathcal{W}_{D^{2}}}}\\
&  \left(  \left(  E\otimes\mathcal{W}_{D^{n-1}}\right)  \otimes
\mathcal{W}_{D^{2}}\right)  _{\left(  E\otimes\mathcal{W}_{D^{n-1}}\right)
_{\pi_{0}\left(  x_{n}\right)  }}\\
&  =\left(  E\otimes\mathcal{W}_{D^{n+1}}\right)  _{\pi_{0}\left(
x_{n}\right)  }%
\end{align*}
Therefore the desired result follows from.the second condition in the item 3
of Notation \ref{n3.3}.
\end{enumerate}
\end{proof}

\begin{lemma}
\label{t6.1.2}The diagram
\[%
\begin{array}
[c]{ccc}%
\quad\quad\mathbf{J}^{n+1}(\pi) & \underrightarrow
{\ \ \ \ \ \ \ \ \ \ \ \ \varphi_{n+1}\ \ \ \ \ \ \ \ \ \ \ \ \ \ } &
\mathbb{\hat{J}}^{D^{n+1}}(\pi)\quad\quad\\
\pi_{n+1,n}\quad\downarrow &  & \downarrow\quad\widehat{\pi}_{n+1,n}\\
\quad\quad\mathbf{J}^{n}(\pi) & \overrightarrow
{\ \ \ \ \ \ \ \ \ \ \ \ \ \varphi_{n}\ \ \ \ \ \ \ \ \ \ \ \ \ \ } &
\widehat{\mathbb{J}}^{D^{n}}(\pi)\quad\quad
\end{array}
\]
is commutative.
\end{lemma}

\begin{proof}
Given $\nabla_{x_{n}}\in\mathbf{J}^{n+1}(\pi)$, $\left(  \widehat{\pi}%
_{n+1,n}\circ\varphi_{n+1}\right)  \left(  \nabla_{x_{n}}\right)  $ is, by the
very definition of $\widehat{\pi}_{n+1,n}$, the composition of mappings
\begin{align*}
&  \left(  M\otimes\mathcal{W}_{D^{n}}\right)  _{\pi\left(  x_{n}\right)
}\underrightarrow{\,\mathbf{s}_{n+1}}\,\left(  M\otimes\mathcal{W}_{D^{n+1}%
}\right)  _{\pi\left(  x_{n}\right)  }\underrightarrow{\,\varphi_{n+1}%
(\nabla_{x_{n}})}\,\\
&  \left(  E\otimes\mathcal{W}_{D^{n+1}}\right)  _{\pi_{0}\left(
x_{n}\right)  }\underrightarrow{\,\mathbf{d}_{n+1}}\,\left(  E\otimes
\mathcal{W}_{D^{n}}\right)  _{\pi_{0}\left(  x_{n}\right)  }%
\end{align*}
which is equivalent, by the very definition of $\varphi_{n+1}(\nabla_{x_{n}}%
)$, to the composition of mappings
\begin{align*}
&  \left(  M\otimes\mathcal{W}_{D^{n}}\right)  _{\pi\left(  x_{n}\right)
}\underrightarrow{\,\mathbf{s}_{n+1}}\,\left(  M\otimes\mathcal{W}_{D^{n+1}%
}\right)  _{\pi\left(  x_{n}\right)  }\\
&  =\left(  \left(  M\otimes\mathcal{W}_{D^{n}}\right)  \otimes\mathcal{W}%
_{D}\right)  _{\left(  M\otimes\mathcal{W}_{D^{n}}\right)  _{\pi\left(
x_{n}\right)  }}\\
&  \underrightarrow{\left\langle \pi_{M}^{M\otimes\mathcal{W}_{D^{n}}}%
\otimes\mathrm{id}_{\mathcal{W}_{D}},\mathrm{id}_{\left(  M\otimes
\mathcal{W}_{D^{n}}\right)  \otimes\mathcal{W}_{D}}\right\rangle }\\
&  \left(  \left(  M\otimes\mathcal{W}_{D}\right)  \underset{M\otimes
\mathcal{W}_{D}}{\times}\left(  \left(  M\otimes\mathcal{W}_{D^{n}}\right)
\otimes\mathcal{W}_{D}\right)  \right)  _{\left\{  \pi\left(  x_{n}\right)
\right\}  \times\left(  M\otimes\mathcal{W}_{D^{n}}\right)  _{\pi\left(
x_{n}\right)  }}\\
&  \underrightarrow{\nabla_{x_{n}}\times\mathrm{id}_{\left(  M\otimes
\mathcal{W}_{D^{n}}\right)  \otimes\mathcal{W}_{D}}}\\
&  \left(  \left(  \mathbf{J}^{n}(\pi)\otimes\mathcal{W}_{D}\right)
\underset{M\otimes\mathcal{W}_{D}}{\times}\left(  \left(  M\otimes
\mathcal{W}_{D^{n}}\right)  \otimes\mathcal{W}_{D}\right)  \right)  _{\left\{
\pi\left(  x_{n}\right)  \right\}  \times\left(  M\otimes\mathcal{W}_{D^{n}%
}\right)  _{\pi\left(  x_{n}\right)  }}\\
&  \underrightarrow{\left(  \varphi_{n}\otimes\mathrm{id}_{\mathcal{W}_{D}%
}\right)  \times\mathrm{id}_{\left(  M\otimes\mathcal{W}_{D^{n}}\right)
\otimes\mathcal{W}_{D}}}\\
&  \left(  \left(  \mathbb{J}^{D^{n}}(\pi)\otimes\mathcal{W}_{D}\right)
\underset{M\otimes\mathcal{W}_{D}}{\times}\left(  \left(  M\otimes
\mathcal{W}_{D^{n}}\right)  \otimes\mathcal{W}_{D}\right)  \right)  _{\left\{
\pi\left(  x_{n}\right)  \right\}  \times\left(  M\otimes\mathcal{W}_{D^{n}%
}\right)  _{\pi\left(  x_{n}\right)  }}\\
&  =\left(  \left(  \mathbb{J}^{D^{n}}(\pi)\underset{M}{\times}\left(
M\otimes\mathcal{W}_{D^{n}}\right)  \right)  \otimes\mathcal{W}_{D}\right)
_{\left\{  \pi\left(  x_{n}\right)  \right\}  \times\left(  M\otimes
\mathcal{W}_{D^{n}}\right)  _{\pi\left(  x_{n}\right)  }}\\
&  \underrightarrow{\left(  \left(  \nabla,\gamma\right)  \in\mathbb{J}%
^{D^{n}}(\pi)\times\left(  M\otimes\mathcal{W}_{D^{n}}\right)  \mapsto
\nabla\left(  \gamma\right)  \in E\otimes\mathcal{W}_{D^{n}}\right)
\otimes\mathrm{id}_{\mathcal{W}_{D}}}\\
&  \left(  \left(  E\otimes\mathcal{W}_{D^{n}}\right)  \otimes\mathcal{W}%
_{D}\right)  _{\left(  E\otimes\mathcal{W}_{D^{n}}\right)  _{\pi_{0}\left(
x_{n}\right)  }}\\
&  =\left(  E\otimes\mathcal{W}_{D^{n+1}}\right)  _{\pi_{0}\left(
x_{n}\right)  }\underrightarrow{\,\mathbf{d}_{n+1}}\,\left(  E\otimes
\mathcal{W}_{D^{n}}\right)  _{\pi_{0}\left(  x_{n}\right)  }%
\end{align*}
This is easily seen to be equivalent to $\varphi_{n}(\pi_{n+1,n}\left(
\nabla_{x_{n}}\right)  )$, which completes the proof.
\end{proof}

Lemma \ref{t6.1.1} can be strengthened as follows:

\begin{lemma}
\label{t6.1.3}We have
\[
\varphi_{n+1}(\nabla_{x_{n}})\in\mathbb{J}^{n+1}(\pi)
\]

\end{lemma}

\begin{proof}
With due regard to Lemmas \ref{t6.1.1} and \ref{t6.1.2}, we have only to show
that
\begin{align}
&  \left(  \varphi_{n+1}(\nabla_{x_{n}})\right)  \circ\left(  \mathrm{id}%
_{M}\otimes\mathcal{W}_{\left(  d_{1},...,d_{n},d_{n+1}\right)  \in
D^{n+1}\mapsto\left(  d_{1},...,d_{n}d_{n+1}\right)  \in D^{n}}\right)
\nonumber\\
&  =\left(  \mathrm{id}_{E}\otimes\mathcal{W}_{\left(  d_{1},...,d_{n}%
,d_{n+1}\right)  \in D^{n+1}\mapsto\left(  d_{1},...,d_{n}d_{n+1}\right)  \in
D^{n+1}}\right)  \circ\nonumber\\
&  \left(  \widehat{\pi}_{n+1,n}(\varphi_{n+1}(\nabla_{x_{n}}))\right)
\label{6.1.3.1}%
\end{align}
For $n=0$, there is nothing to prove. We proceed by induction on $n$. By the
very definition of $\varphi_{n+1}$, the left-hand side of (\ref{6.1.3.1}) is
the composition of mappings
\begin{align*}
&  \left(  M\otimes\mathcal{W}_{D^{n}}\right)  _{\pi\left(  x_{n}\right)  }\\
&  \underrightarrow{\mathrm{id}_{M}\otimes\mathcal{W}_{\left(  d_{1}%
,...,d_{n},d_{n+1}\right)  \in D^{n+1}\mapsto\left(  d_{1},...,d_{n}%
d_{n+1}\right)  \in D^{n}}}\\
&  \left(  M\otimes\mathcal{W}_{D^{n+1}}\right)  _{\pi\left(  x_{n}\right)
}\\
&  =\left(  \left(  M\otimes\mathcal{W}_{D^{n}}\right)  \otimes\mathcal{W}%
_{D}\right)  _{\left(  M\otimes\mathcal{W}_{D^{n}}\right)  _{\pi\left(
x_{n}\right)  }}\\
&  \underrightarrow{\left\langle \pi_{M}^{M\otimes\mathcal{W}_{D^{n}}}%
\otimes\mathrm{id}_{\mathcal{W}_{D}},\mathrm{id}_{\left(  M\otimes
\mathcal{W}_{D^{n}}\right)  \otimes\mathcal{W}_{D}}\right\rangle }\\
&  \left(  M\otimes\mathcal{W}_{D}\right)  _{\pi\left(  x_{n}\right)
}\underset{M\otimes\mathcal{W}_{D}}{\times}\left(  \left(  M\otimes
\mathcal{W}_{D^{n}}\right)  \otimes\mathcal{W}_{D}\right)  _{\left(
M\otimes\mathcal{W}_{D^{n}}\right)  _{\pi\left(  x_{n}\right)  }}\\
&  \underrightarrow{\nabla_{x_{n}}\times\mathrm{id}_{\left(  M\otimes
\mathcal{W}_{D^{n}}\right)  \otimes\mathcal{W}_{D}}}\\
&  \left(  \mathbf{J}^{n}(\pi)\otimes\mathcal{W}_{D}\right)  _{\pi\left(
x_{n}\right)  }\underset{M\otimes\mathcal{W}_{D}}{\times}\left(  \left(
M\otimes\mathcal{W}_{D^{n}}\right)  \otimes\mathcal{W}_{D}\right)  _{\left(
M\otimes\mathcal{W}_{D^{n}}\right)  _{\pi\left(  x_{n}\right)  }}\\
&  \underrightarrow{\left(  \varphi_{n}\otimes\mathrm{id}_{\mathcal{W}_{D}%
}\right)  \times\mathrm{id}_{\left(  M\otimes\mathcal{W}_{D^{n}}\right)
\otimes\mathcal{W}_{D}}}\\
&  \left(  \mathbb{J}^{D^{n}}(\pi)\otimes\mathcal{W}_{D}\right)  _{\pi\left(
x_{n}\right)  }\underset{M\otimes\mathcal{W}_{D}}{\times}\left(  \left(
M\otimes\mathcal{W}_{D^{n}}\right)  \otimes\mathcal{W}_{D}\right)  _{\left(
M\otimes\mathcal{W}_{D^{n}}\right)  _{\pi\left(  x_{n}\right)  }}\\
&  =\left(  \left(  \mathbb{J}^{D^{n}}(\pi)\underset{M}{\times}\left(
M\otimes\mathcal{W}_{D^{n}}\right)  \right)  \otimes\mathcal{W}_{D}\right)
_{\left\{  \pi\left(  x_{n}\right)  \right\}  \times\left(  M\otimes
\mathcal{W}_{D^{n}}\right)  _{\pi\left(  x_{n}\right)  }}\\
&  \underrightarrow{\left(  \left(  \nabla,\gamma\right)  \in\mathbb{J}%
^{D^{n}}(\pi)\times\left(  M\otimes\mathcal{W}_{D^{n}}\right)  \mapsto
\nabla\left(  \gamma\right)  \in E\otimes\mathcal{W}_{D^{n}}\right)
\otimes\mathrm{id}_{\mathcal{W}_{D}}}\\
&  \left(  \left(  E\otimes\mathcal{W}_{D^{n}}\right)  \otimes\mathcal{W}%
_{D}\right)  _{\left(  E\otimes\mathcal{W}_{D^{n}}\right)  _{\pi_{0}\left(
x_{n}\right)  }}\\
&  =\left(  E\otimes\mathcal{W}_{D^{n+1}}\right)  _{\pi_{0}\left(
x_{n}\right)  }%
\end{align*}
which is easily seen, by dint of Lemma \ref{t6.1.1},\ to be equivalent to the
right-hand side of (\ref{6.1.3.1}).
\end{proof}

Thus we have established the mappings $\varphi_{n}:\mathbf{J}^{n}%
(\pi)\rightarrow\mathbb{J}^{D^{n}}(\pi)$.

\section{\label{s7}From the Second Approach to the Third}

The principal objective in this section is to define a mapping $\psi
_{n}:\mathbb{J}^{D^{n}}(\pi)\rightarrow\mathbb{J}^{D_{n}}(\pi)$. Let us begin with

\begin{proposition}
\label{t7.1.1}Let $\nabla_{x}$ be a $D^{n}$-pseudotangential \textit{over }the
bundle $\pi:E\rightarrow M$ \textit{at} $x\in E$ and $\gamma\in\left(
M\otimes\mathcal{W}_{D_{n}}\right)  _{\pi(x)}$. Then there exists a unique
$\gamma^{\prime}\in\left(  E\otimes\mathcal{W}_{D_{n}}\right)  _{x}$ such
that
\begin{align*}
&  \nabla_{x}(\left(  \mathrm{id}_{M}\otimes\mathcal{W}_{(d_{1},...,d_{n})\in
D^{n}\longmapsto(d_{1}+...+d_{n})\in D_{n}}\right)  \left(  \gamma\right)  )\\
&  =\left(  \mathrm{id}_{E}\otimes\mathcal{W}_{(d_{1},...,d_{n})\in
D^{n}\longmapsto(d_{1}+...+d_{n})\in D_{n}}\right)  \left(  \gamma^{\prime
}\right)
\end{align*}

\end{proposition}

\begin{proof}
This stems easily from the following simple lemma.
\end{proof}

\begin{lemma}
\label{t7.1.2}The diagram
\[
\mathcal{W}_{D_{n}}\underrightarrow{\,\mathcal{W}_{(d_{1},...,d_{n})\in
D^{n}\longmapsto(d_{1}+...+d_{n})\in D_{n}}}\,\mathcal{W}_{D^{n}}
\begin{array}
[c]{c}%
\underrightarrow{\mathcal{W}_{\tau_{1}}}\\
\vdots\\
\underrightarrow{\mathcal{W}_{\tau_{i}}}\\
\vdots\\
\underrightarrow{\mathcal{W}_{\tau_{n-1}}}%
\end{array}
\mathcal{W}_{D^{n}}%
\]
is a limit diagram in the category of Weil algebras, where $\tau_{i}%
:D^{n}\rightarrow D^{n}$ is the mapping permuting the $i$-th and $(i+1)$-th
components of $D^{n}$ while fixing the other components.
\end{lemma}

\begin{notation}
We will denote by $\widehat{\psi}_{n}(\nabla_{x})(\gamma)$ the unique
$\gamma^{\prime}$ in the above proposition, thereby getting a function
$\widehat{\psi}_{n}(\nabla_{x}):\left(  M\otimes\mathcal{W}_{D_{n}}\right)
_{\pi(x)}\rightarrow\left(  E\otimes\mathcal{W}_{D_{n}}\right)  _{x}$.
\end{notation}

\begin{proposition}
\label{t7.1.3}For any\textbf{\ }$\nabla_{x}\in\widehat{\mathbb{J}}_{x}^{D^{n}%
}(\pi)$, we have $\widehat{\psi}_{n}(\nabla_{x})\in\widehat{\mathbb{J}}%
_{x}^{D_{n}}(\pi)$.
\end{proposition}

\begin{proof}
We have to verify the three conditions in Definition \ref{d5.1.1} concerning
the mapping $\widehat{\psi}_{n}(\nabla_{x}):\left(  M\otimes\mathcal{W}%
_{D_{n}}\right)  _{\pi(x)}\rightarrow\left(  E\otimes\mathcal{W}_{D_{n}%
}\right)  _{x}$.

\begin{enumerate}
\item To see the first condition, it suffices to show that
\begin{align*}
&  \left(  \mathrm{id}_{M}\otimes\mathcal{W}_{(d_{1},...,d_{n})\in
D^{n}\longmapsto(d_{1}+...+d_{n})\in D_{n}}\right)  \left(  \gamma\right) \\
&  =\left(  \mathrm{id}_{E}\otimes\mathcal{W}_{(d_{1},...,d_{n})\in
D^{n}\longmapsto(d_{1}+...+d_{n})\in D_{n}}\right)  \left(  \left(  \pi
\otimes\mathrm{id}_{\mathcal{W}_{D_{n}}}\right)  \left(  \widehat{\psi}%
_{n}(\nabla_{x})\left(  \gamma\right)  \right)  \right)  \text{,}%
\end{align*}
which follows from
\begin{align*}
&  \left(  \mathrm{id}_{M}\otimes\mathcal{W}_{(d_{1},...,d_{n})\in
D^{n}\longmapsto(d_{1}+...+d_{n})\in D_{n}}\right)  \left(  \left(  \pi
\otimes\mathrm{id}_{\mathcal{W}_{D_{n}}}\right)  \left(  \widehat{\psi}%
_{n}(\nabla_{x})\left(  \gamma\right)  \right)  \right) \\
&  =\left(  \pi\otimes\mathrm{id}_{\mathcal{W}_{D^{n}}}\right)  \left(
\left(  \mathrm{id}_{E}\otimes\mathcal{W}_{(d_{1},...,d_{n})\in D^{n}%
\longmapsto(d_{1}+...+d_{n})\in D_{n}}\right)  \left(  \widehat{\psi}%
_{n}(\nabla_{x})\left(  \gamma\right)  \right)  \right) \\
&  \text{\lbrack By the bifunctionality of }\otimes\text{]}\\
&  =\left(  \pi\otimes\mathrm{id}_{\mathcal{W}_{D^{n}}}\right)  \left(
\nabla_{x}(\left(  \mathrm{id}_{M}\otimes\mathcal{W}_{(d_{1},...,d_{n})\in
D^{n}\longmapsto(d_{1}+...+d_{n})\in D_{n}}\right)  \left(  \gamma\right)
)\right) \\
&  \text{\lbrack By the very definition of }\widehat{\psi}_{n}(\nabla
_{x})\text{]}\\
&  =\left(  \mathrm{id}_{M}\otimes\mathcal{W}_{(d_{1},...,d_{n})\in
D^{n}\longmapsto(d_{1}+...+d_{n})\in D_{n}}\right)  \left(  \gamma\right)
\end{align*}

\item Now we are going to deal with the second condition. It is easy to see
that the composition of mappings
\begin{align*}
&  \left(  M\otimes\mathcal{W}_{D_{n}}\right)  _{\pi\left(  x\right)
}\underrightarrow{\,\mathrm{id}_{M}\otimes\mathcal{W}_{\left(  \alpha
\cdot\right)  _{D_{n}}}}\,\left(  M\otimes\mathcal{W}_{D_{n}}\right)
_{\pi\left(  x\right)  }\underrightarrow{\,\mathcal{W}_{(d_{1},...,d_{n})\in
D^{n}\longmapsto(d_{1}+...+d_{n})\in D_{n}}}\,\\
&  \left(  M\otimes\mathcal{W}_{D^{n}}\right)  _{\pi\left(  x\right)  }%
\end{align*}
is equivalent to the composition of mappings
\begin{align*}
&  \left(  M\otimes\mathcal{W}_{D_{n}}\right)  _{\pi\left(  x\right)
}\underrightarrow{\,\mathcal{W}_{(d_{1},...,d_{n})\in D^{n}\longmapsto
(d_{1}+...+d_{n})\in D_{n}}}\,\left(  M\otimes\mathcal{W}_{D^{n}}\right)
_{\pi\left(  x\right)  }\\
&  \underrightarrow{\,\mathrm{id}_{M}\otimes\mathcal{W}_{\left(
\alpha\underset{1}{\cdot}\right)  _{D^{n}}}}\,\left(  M\otimes\mathcal{W}%
_{D^{n}}\right)  _{\pi\left(  x\right)  }...\underrightarrow{\mathrm{id}%
_{M}\otimes\mathcal{W}_{\left(  \alpha\underset{n}{\cdot}\right)  _{D^{n}}}%
}\,\\
&  \left(  M\otimes\mathcal{W}_{D^{n}}\right)  _{\pi\left(  x\right)
}\text{,}%
\end{align*}
while the composition of mappings
\begin{align*}
&  \left(  M\otimes\mathcal{W}_{D^{n}}\right)  _{\pi\left(  x\right)
}\underrightarrow{\,\mathrm{id}_{M}\otimes\mathcal{W}_{\left(  \alpha
\underset{1}{\cdot}\right)  _{D^{n}}}}\,\left(  M\otimes\mathcal{W}_{D^{n}%
}\right)  _{\pi\left(  x\right)  }...\,\underrightarrow{\mathrm{id}_{M}%
\otimes\mathcal{W}_{\left(  \alpha\underset{n}{\cdot}\right)  _{D^{n}}}}\,\\
&  \left(  M\otimes\mathcal{W}_{D^{n}}\right)  _{\pi\left(  x\right)
}\underrightarrow{\,\nabla_{x}}\,\left(  E\otimes\mathcal{W}_{D^{n}}\right)
_{x}%
\end{align*}
is equivalent to the composition of mappings
\begin{align*}
&  \left(  M\otimes\mathcal{W}_{D^{n}}\right)  _{\pi\left(  x\right)
}\underrightarrow{\,\nabla_{x}}\,\left(  E\otimes\mathcal{W}_{D^{n}}\right)
_{x}\underrightarrow{\,\mathrm{id}_{E}\otimes\mathcal{W}_{\left(
\alpha\underset{1}{\cdot}\right)  _{D^{n}}}}\,\left(  E\otimes\mathcal{W}%
_{D^{n}}\right)  _{x}...\\
&  \,\underrightarrow{\mathrm{id}_{E}\otimes\mathcal{W}_{\left(
\alpha\underset{n}{\cdot}\right)  _{D^{n}}}}\,\left(  E\otimes\mathcal{W}%
_{D^{n}}\right)  _{x}%
\end{align*}
Therefore the composition of mappings
\begin{align*}
&  \left(  M\otimes\mathcal{W}_{D_{n}}\right)  _{\pi\left(  x\right)
}\underrightarrow{\,\mathrm{id}_{M}\otimes\mathcal{W}_{\left(  \alpha
\cdot\right)  _{D_{n}}}}\,\left(  M\otimes\mathcal{W}_{D_{n}}\right)
_{\pi\left(  x\right)  }\underrightarrow{\,\mathcal{W}_{(d_{1},...,d_{n})\in
D^{n}\longmapsto(d_{1}+...+d_{n})\in D_{n}}}\,\\
&  \left(  M\otimes\mathcal{W}_{D^{n}}\right)  _{\pi\left(  x\right)
}\underrightarrow{\,\nabla_{x}}\,\left(  E\otimes\mathcal{W}_{D^{n}}\right)
_{x}%
\end{align*}
is equivalent to the composition of mappings
\begin{align*}
&  \left(  M\otimes\mathcal{W}_{D_{n}}\right)  _{\pi\left(  x\right)
}\underrightarrow{\,\mathcal{W}_{(d_{1},...,d_{n})\in D^{n}\longmapsto
(d_{1}+...+d_{n})\in D_{n}}}\,\left(  M\otimes\mathcal{W}_{D^{n}}\right)
_{\pi\left(  x\right)  }\underrightarrow{\,\nabla_{x}}\,\left(  E\otimes
\mathcal{W}_{D^{n}}\right)  _{x}\\
&  \underrightarrow{\,\mathrm{id}_{E}\otimes\mathcal{W}_{\left(
\alpha\underset{1}{\cdot}\right)  _{D^{n}}}}\,\left(  E\otimes\mathcal{W}%
_{D^{n}}\right)  _{x}...\underrightarrow{\,\mathrm{id}_{E}\otimes
\mathcal{W}_{\left(  \alpha\underset{n}{\cdot}\right)  _{D^{n}}}}\,\left(
E\otimes\mathcal{W}_{D^{n}}\right)  _{x}\text{,}%
\end{align*}
which should be equivalent in turn to
\begin{align*}
&  \left(  M\otimes\mathcal{W}_{D_{n}}\right)  _{\pi\left(  x\right)
}\underrightarrow{\widehat{\,\psi}_{n}(\nabla_{x})}\,\left(  E\otimes
\mathcal{W}_{D_{n}}\right)  _{x}\underrightarrow{\,\mathrm{id}_{E}%
\otimes\mathcal{W}_{(d_{1},...,d_{n})\in D^{n}\longmapsto(d_{1}+...+d_{n})\in
D_{n}}}\,\left(  E\otimes\mathcal{W}_{D^{n}}\right)  _{x}\\
&  \underrightarrow{\,\mathrm{id}_{E}\otimes\mathcal{W}_{\left(
\alpha\underset{1}{\cdot}\right)  _{D^{n}}}}\,\left(  E\otimes\mathcal{W}%
_{D^{n}}\right)  _{x}...\underrightarrow{\,\mathrm{id}_{E}\otimes
\mathcal{W}_{\left(  \alpha\underset{n}{\cdot}\right)  _{D^{n}}}}\,\left(
E\otimes\mathcal{W}_{D^{n}}\right)  _{x}%
\end{align*}
Since the composition of mappings
\begin{align*}
&  \left(  E\otimes\mathcal{W}_{D_{n}}\right)  _{x}\underrightarrow
{\mathcal{W}_{(d_{1},...,d_{n})\in D^{n}\longmapsto(d_{1}+...+d_{n})\in D_{n}%
}}\left(  E\otimes\mathcal{W}_{D^{n}}\right)  _{x}\underrightarrow
{\,\mathrm{id}_{E}\otimes\mathcal{W}_{\left(  \alpha\underset{1}{\cdot
}\right)  _{D^{n}}}}\,\left(  E\otimes\mathcal{W}_{D^{n}}\right)  _{x}...\\
&  \underrightarrow{\mathrm{id}_{E}\otimes\mathcal{W}_{\left(  \alpha
\underset{n}{\cdot}\right)  _{D^{n}}}}\,\left(  E\otimes\mathcal{W}_{D^{n}%
}\right)  _{x}%
\end{align*}
is equivalent to the composition of mappings
\begin{align*}
&  \left(  E\otimes\mathcal{W}_{D_{n}}\right)  _{x}\underrightarrow
{\,\mathrm{id}_{E}\otimes\mathcal{W}_{\left(  \alpha\cdot\right)  _{D_{n}}}%
}\,\left(  E\otimes\mathcal{W}_{D_{n}}\right)  _{x}\underrightarrow
{\,\mathrm{id}_{E}\otimes\mathcal{W}_{(d_{1},...,d_{n})\in D^{n}%
\longmapsto(d_{1}+...+d_{n})\in D_{n}}}\\
&  \left(  E\otimes\mathcal{W}_{D^{n}}\right)  _{x}\text{,}%
\end{align*}
the coveted result follows.

\item We are going to deal with the third condition. We have to show that the
diagram
\begin{equation}%
\begin{array}
[c]{ccc}%
\left(  M\otimes\mathcal{W}_{D_{n}}\right)  _{\pi\left(  x\right)  } &
\underrightarrow{\mathrm{id}_{M}\otimes\mathcal{W}_{\mathbf{m}_{D_{n}\times
D_{m}\rightarrow D_{n}}}} & \left(  M\otimes\mathcal{W}_{D_{n}}\right)
_{\pi\left(  x\right)  }\otimes\mathcal{W}_{D_{m}}\\%
\begin{array}
[c]{cc}%
\widehat{\psi}_{n}(\nabla_{x}) & \downarrow
\end{array}
&  &
\begin{array}
[c]{cc}%
\downarrow & \widehat{\psi}_{n}(\nabla_{x})\otimes\mathrm{id}_{\mathcal{W}%
_{D_{m}}}%
\end{array}
\\
\left(  E\otimes\mathcal{W}_{D_{n}}\right)  _{x} & \underrightarrow
{\mathrm{id}_{E}\otimes\mathcal{W}_{\mathbf{m}_{D_{n}\times D_{m}\rightarrow
D_{n}}}} & \left(  E\otimes\mathcal{W}_{D_{n}}\right)  _{x}\otimes
\mathcal{W}_{D_{m}}%
\end{array}
\label{7.1.3.1}%
\end{equation}
commutes. It is easy to see that the diagram
\[%
\begin{array}
[c]{ccc}%
\left(  E\otimes\mathcal{W}_{D_{n}}\right)  _{x} & \mathrm{id}_{E}%
\otimes\mathcal{W}_{+_{D^{n}\rightarrow D_{n}}} & \left(  E\otimes
\mathcal{W}_{D^{n}}\right)  _{x}\\%
\begin{array}
[c]{cc}%
\mathrm{id}_{E}\otimes\mathcal{W}_{\mathbf{m}_{D_{n}\times D_{m}\rightarrow
D_{n}}} & \downarrow
\end{array}
&  &
\begin{array}
[c]{cc}%
\downarrow & \mathrm{id}_{E}\otimes\mathcal{W}_{\eta}%
\end{array}
\\
\left(  E\otimes\mathcal{W}_{D_{n}}\right)  _{x}\otimes\mathcal{W}_{D_{m}} &
\mathrm{id}_{E}\otimes\mathcal{W}_{+_{D^{n}\rightarrow D_{n}}\times
\mathrm{id}_{D_{m}}} & \left(  E\otimes\mathcal{W}_{D^{n}}\right)  _{x}%
\otimes\mathcal{W}_{D_{m}}%
\end{array}
\]
commutes, where $\eta$\ stands for
\[
(d_{1},...,d_{n},e)\in D^{n}\times D_{m}\longmapsto(d_{1}e,...,d_{n}e)\in
D^{n}%
\]
so that the commutativity of the diagram in (\ref{7.1.3.1}) is equivalent to
the commutativity of the outer square of the diagram
\begin{equation}%
\begin{array}
[c]{ccc}%
\left(  M\otimes\mathcal{W}_{D_{n}}\right)  _{\pi\left(  x\right)  } &
\underrightarrow{\mathrm{id}_{M}\otimes\mathcal{W}_{\mathbf{m}_{D_{n}\times
D_{m}\rightarrow D_{n}}}} & \left(  M\otimes\mathcal{W}_{D_{n}}\right)
_{\pi\left(  x\right)  }\otimes\mathcal{W}_{D_{m}}\\%
\begin{array}
[c]{cc}%
\widehat{\psi}_{n}(\nabla_{x}) & \downarrow
\end{array}
&  &
\begin{array}
[c]{cc}%
\downarrow & \widehat{\psi}_{n}(\nabla_{x})\otimes\mathrm{id}_{\mathcal{W}%
_{D_{m}}}%
\end{array}
\\
\left(  E\otimes\mathcal{W}_{D_{n}}\right)  _{x} & \underrightarrow
{\mathrm{id}_{E}\otimes\mathcal{W}_{\mathbf{m}_{D_{n}\times D_{m}\rightarrow
D_{n}}}} & \left(  E\otimes\mathcal{W}_{D_{n}}\right)  _{x}\otimes
\mathcal{W}_{D_{m}}\\%
\begin{array}
[c]{cc}%
\mathrm{id}_{E}\otimes\mathcal{W}_{+_{D^{n}\rightarrow D_{n}}} & \downarrow
\end{array}
&  &
\begin{array}
[c]{cc}%
\downarrow & \mathrm{id}_{E}\otimes\mathcal{W}_{+_{D^{n}\rightarrow D_{n}%
}\times\mathrm{id}_{D_{m}}}%
\end{array}
\\
\left(  E\otimes\mathcal{W}_{D^{n}}\right)  _{x} & \underrightarrow
{\mathrm{id}_{E}\otimes\mathcal{W}_{\eta}} & \left(  E\otimes\mathcal{W}%
_{D^{n}}\right)  _{x}\otimes\mathcal{W}_{D_{m}}%
\end{array}
\label{7.1.3.2}%
\end{equation}
where $+_{D^{n}\rightarrow D_{n}}$\ stands for
\[
(d_{1},...,d_{n})\in D^{n}\longmapsto(d_{1}+...+d_{n})\in D_{n}%
\]
The composition of mappings
\[
\left(  M\otimes\mathcal{W}_{D_{n}}\right)  _{\pi\left(  x\right)
}\underrightarrow{\widehat{\,\psi}_{n}(\nabla_{x})}\,\left(  E\otimes
\mathcal{W}_{D_{n}}\right)  _{x}\underrightarrow{\,\mathrm{id}_{E}%
\otimes\mathcal{W}_{+_{D^{n}\rightarrow D_{n}}}}\,\left(  E\otimes
\mathcal{W}_{D^{n}}\right)  _{x}%
\]
is equal to the composition of mappings
\[
\left(  M\otimes\mathcal{W}_{D_{n}}\right)  _{\pi\left(  x\right)
}\underrightarrow{\,\mathrm{id}_{M}\otimes\mathcal{W}_{+_{D^{n}\rightarrow
D_{n}}}}\,\left(  M\otimes\mathcal{W}_{D^{n}}\right)  _{\pi\left(  x\right)
}\underrightarrow{\,\nabla_{x}}\,\left(  E\otimes\mathcal{W}_{D^{n}}\right)
_{x}%
\]
while the composition of mappings
\begin{align*}
&  \left(  M\otimes\mathcal{W}_{D_{n}}\right)  _{\pi\left(  x\right)  }%
\otimes\mathcal{W}_{D_{m}}\underrightarrow{\widehat{\,\psi}_{n}(\nabla
_{x})\otimes\mathrm{id}_{\mathcal{W}_{D_{m}}}}\,\left(  E\otimes
\mathcal{W}_{D_{n}}\right)  _{x}\otimes\mathcal{W}_{D_{m}}\\
&  \underrightarrow{\,\mathrm{id}_{E}\otimes\mathcal{W}_{+_{D^{n}\rightarrow
D_{n}}\times\mathrm{id}_{D_{m}}}}\left(  E\otimes\mathcal{W}_{D^{n}}\right)
_{x}\otimes\mathcal{W}_{D_{m}}%
\end{align*}
is equal to the composition of mappings
\begin{align*}
&  \left(  M\otimes\mathcal{W}_{D_{n}}\right)  _{\pi\left(  x\right)
}\underrightarrow{\,\mathrm{id}_{M}\otimes\mathcal{W}_{+_{D^{n}\rightarrow
D_{n}}\times\mathrm{id}_{D_{m}}}}\,\left(  M\otimes\mathcal{W}_{D^{n}}\right)
_{\pi\left(  x\right)  }\underrightarrow{\,\nabla_{x}\otimes\mathrm{id}%
_{\mathcal{W}_{D_{m}}}}\\
&  \left(  E\otimes\mathcal{W}_{D^{n}}\right)  _{x}\otimes\mathcal{W}_{D_{m}}%
\end{align*}
Since the diagram
\[%
\begin{array}
[c]{ccc}%
\left(  M\otimes\mathcal{W}_{D_{n}}\right)  _{\pi\left(  x\right)  } &
\underrightarrow{\mathrm{id}_{M}\otimes\mathcal{W}_{\mathbf{m}_{D_{n}\times
D_{m}\rightarrow D_{n}}}} & \left(  M\otimes\mathcal{W}_{D_{n}}\right)
_{\pi\left(  x\right)  }\otimes\mathcal{W}_{D_{m}}\\%
\begin{array}
[c]{cc}%
\mathrm{id}_{M}\otimes\mathcal{W}_{+_{D^{n}\rightarrow D_{n}}} & \downarrow
\end{array}
&  &
\begin{array}
[c]{cc}%
\downarrow & \mathrm{id}_{M}\otimes\mathcal{W}_{+_{D^{n}\rightarrow D_{n}%
}\times\mathrm{id}_{D_{m}}}%
\end{array}
\\
\left(  M\otimes\mathcal{W}_{D^{n}}\right)  _{\pi\left(  x\right)  } &
\underrightarrow{\mathrm{id}_{M}\otimes\mathcal{W}_{\eta}} & \left(
M\otimes\mathcal{W}_{D^{n}}\right)  _{\pi\left(  x\right)  }\otimes
\mathcal{W}_{D_{m}}\\%
\begin{array}
[c]{cc}%
\nabla_{x} & \downarrow
\end{array}
&  &
\begin{array}
[c]{cc}%
\downarrow & \nabla_{x}\otimes\mathrm{id}_{\mathcal{W}_{D_{m}}}%
\end{array}
\\
\left(  E\otimes\mathcal{W}_{D^{n}}\right)  _{x} & \underrightarrow
{\mathrm{id}_{E}\otimes\mathcal{W}_{\eta}} & \left(  E\otimes\mathcal{W}%
_{D^{n}}\right)  _{x}\otimes\mathcal{W}_{D_{m}}%
\end{array}
\]
commutes, the outer square of the diagram in (\ref{7.1.3.2}) commutes. This
completes the proof.
\end{enumerate}
\end{proof}

\begin{proposition}
\label{t7.1.4}The diagram
\[%
\begin{array}
[c]{ccc}%
\mathbb{\hat{J}}_{x}^{D^{n+1}}(\pi) & \underrightarrow{\widehat{\psi}_{n+1}} &
\mathbb{\hat{J}}_{x}^{D_{n+1}}(\pi)\\%
\begin{array}
[c]{cc}%
\widehat{\mathbf{\pi}}_{n+1,n} & \downarrow
\end{array}
&  &
\begin{array}
[c]{cc}%
\downarrow & \widehat{\mathbf{\pi}}_{n+1,n}%
\end{array}
\\
\mathbb{\hat{J}}_{x}^{D^{n}}(\pi) & \overrightarrow{\widehat{\psi}_{n}} &
\mathbb{\hat{J}}_{x}^{D_{n}}(\pi)
\end{array}
\]
commutes.
\end{proposition}

\begin{proof}
Given $\nabla_{x}\in\mathbb{\hat{J}}_{x}^{D^{n+1}}(\pi)$, the composition of
mappings
\begin{align}
&  \left(  M\otimes\mathcal{W}_{D_{n}}\right)  _{\pi\left(  x\right)
}\underrightarrow{\,\widehat{\mathbf{\pi}}_{n+1,n}\left(  \widehat{\psi}%
_{n+1}\left(  \nabla_{x}\right)  \right)  }\,\left(  E\otimes\mathcal{W}%
_{D_{n}}\right)  _{x}\underrightarrow{\,\mathrm{id}_{E}\otimes\mathcal{W}%
_{\mathbf{m}_{D_{n}\times D_{n}\rightarrow D_{n}}}}\,\nonumber\\
&  \left(  E\otimes\mathcal{W}_{D_{n}}\right)  _{x}\otimes\mathcal{W}_{D_{n}%
}\,\underrightarrow{\mathrm{id}_{E}\otimes\mathcal{W}_{+_{D^{n}\rightarrow
D_{n}}\times\mathrm{id}_{D_{n}}}}\,\left(  E\otimes\mathcal{W}_{D^{n}}\right)
_{x}\otimes\mathcal{W}_{D_{n}} \label{7.1.4.1.a}%
\end{align}
is equivalent to the composition of mappings
\begin{align*}
&  \left(  M\otimes\mathcal{W}_{D_{n}}\right)  _{\pi\left(  x\right)
}\underrightarrow{\,\widehat{\mathbf{\pi}}_{n+1,n}\left(  \widehat{\psi}%
_{n+1}\left(  \nabla_{x}\right)  \right)  }\,\left(  E\otimes\mathcal{W}%
_{D_{n}}\right)  _{x}\underrightarrow{\,\mathrm{id}_{E}\otimes\mathcal{W}%
_{\mathbf{m}_{D_{n+1}\times D_{n}\rightarrow D_{n}}}}\,\\
&  \left(  E\otimes\mathcal{W}_{D_{n+1}}\right)  _{x}\otimes\mathcal{W}%
_{D_{n}}\underrightarrow{\mathrm{id}_{E}\otimes\mathcal{W}_{+_{D^{n+1}%
\rightarrow D_{n+1}}\times\mathrm{id}_{D_{n}}}}\,\left(  E\otimes
\mathcal{W}_{D^{n+1}}\right)  _{x}\otimes\mathcal{W}_{D_{n}}\underrightarrow
{\,\mathbf{d}_{n+1}\otimes\mathrm{id}_{\mathcal{W}_{D_{n}}}}\\
&  \,\left(  E\otimes\mathcal{W}_{D^{n}}\right)  _{x}\otimes\mathcal{W}%
_{D_{n}}%
\end{align*}
which is in turn equivalent to the composition of mappings
\begin{align*}
&  \left(  M\otimes\mathcal{W}_{D_{n}}\right)  _{\pi\left(  x\right)
}\underrightarrow{\,\mathrm{id}_{M}\otimes\mathcal{W}_{\mathbf{m}%
_{D_{n+1}\times D_{n}\rightarrow D_{n}}}}\,\left(  M\otimes\mathcal{W}%
_{D_{n+1}}\right)  _{\pi\left(  x\right)  }\otimes\mathcal{W}_{D_{n}%
}\,\underrightarrow{\widehat{\psi}_{n+1}\left(  \nabla_{x}\right)
\otimes\mathcal{W}_{\mathrm{id}_{\mathcal{W}_{D_{n}}}}}\,\\
&  \left(  E\otimes\mathcal{W}_{D_{n+1}}\right)  _{x}\otimes\mathcal{W}%
_{D_{n}}\,\underrightarrow{\mathrm{id}_{E}\otimes\mathcal{W}_{+_{D^{n+1}%
\rightarrow D_{n+1}}\times\mathrm{id}_{D_{n}}}}\,\left(  E\otimes
\mathcal{W}_{D^{n+1}}\right)  _{x}\otimes\mathcal{W}_{D_{n}}\underrightarrow
{\,\mathbf{d}_{n+1}\otimes\mathrm{id}_{\mathcal{W}_{D_{n}}}}\,\\
&  \left(  E\otimes\mathcal{W}_{D^{n}}\right)  _{x}\otimes\mathcal{W}_{D_{n}}%
\end{align*}
This is to be supplanted by the composition of mappings
\begin{align*}
&  \left(  M\otimes\mathcal{W}_{D_{n}}\right)  _{\pi\left(  x\right)
}\underrightarrow{\,\mathrm{id}_{M}\otimes\mathcal{W}_{\mathbf{m}%
_{D_{n+1}\times D_{n}\rightarrow D_{n}}}}\,\left(  M\otimes\mathcal{W}%
_{D_{n+1}}\right)  _{\pi\left(  x\right)  }\otimes\mathcal{W}_{D_{n}}\\
&  \underrightarrow{\,\mathrm{id}_{M}\otimes\mathcal{W}_{+_{D^{n+1}\rightarrow
D_{n+1}}\times\mathrm{id}_{D_{n}}}}\,\,\left(  M\otimes\mathcal{W}_{D^{n+1}%
}\right)  _{\pi\left(  x\right)  }\otimes\mathcal{W}_{D_{n}}\underrightarrow
{\nabla_{x}\otimes\mathrm{id}_{\mathcal{W}_{D_{n}}}}\\
&  \left(  E\otimes\mathcal{W}_{D^{n+1}}\right)  _{x}\otimes\mathcal{W}%
_{D_{n}}\,\underrightarrow{\mathbf{d}_{n+1}\otimes\mathrm{id}_{\mathcal{W}%
_{D_{n}}}}\,\left(  E\otimes\mathcal{W}_{D^{n}}\right)  _{x}\otimes
\mathcal{W}_{D_{n}}\text{,}%
\end{align*}
which is in turn equivalent to the composition of mappings
\begin{align*}
&  \left(  M\otimes\mathcal{W}_{D_{n}}\right)  _{\pi\left(  x\right)
}\underrightarrow{\,\mathrm{id}_{M}\otimes\mathcal{W}_{\mathbf{m}%
_{_{D_{n+1}\times D_{n}\rightarrow D_{n}}}}}\,\left(  M\otimes\mathcal{W}%
_{D_{n+1}}\right)  _{\pi\left(  x\right)  }\otimes\mathcal{W}_{D_{n}%
}\underrightarrow{\,\mathrm{id}_{M}\otimes\mathcal{W}_{+_{D^{n+1}\rightarrow
D_{n+1}}\times\mathrm{id}_{D_{n}}}}\,\\
&  \left(  M\otimes\mathcal{W}_{D^{n+1}}\right)  _{\pi\left(  x\right)
}\otimes\mathcal{W}_{D_{n}}\,\underrightarrow{\mathbf{d}_{n+1}\otimes
\mathrm{id}_{\mathcal{W}_{D_{n}}}}\,\left(  M\otimes\mathcal{W}_{D^{n}%
}\right)  _{\pi\left(  x\right)  }\otimes\mathcal{W}_{D_{n}}\\
&  \underrightarrow{\,\widehat{\mathbf{\pi}}_{n+1,n}\left(  \nabla_{x}\right)
\otimes\mathrm{id}_{\mathcal{W}_{D_{n}}}}\,\left(  E\otimes\mathcal{W}_{D^{n}%
}\right)  _{x}\otimes\mathcal{W}_{D_{n}}%
\end{align*}
by Proposition \ref{t4.1.3}. This is to be supplanted by the composition of
mappings
\begin{align*}
&  \left(  M\otimes\mathcal{W}_{D_{n}}\right)  _{\pi\left(  x\right)
}\underrightarrow{\,\mathrm{id}_{M}\otimes\mathcal{W}_{\mathbf{m}_{D_{n}\times
D_{n}\rightarrow D_{n}}}}\,\left(  M\otimes\mathcal{W}_{D_{n}}\right)
_{\pi\left(  x\right)  }\otimes\mathcal{W}_{D_{n}}\underrightarrow
{\,\mathrm{id}_{M}\otimes\mathcal{W}_{+_{D^{n}\rightarrow D_{n}}%
\times\mathrm{id}_{D_{n}}}}\,\\
&  \left(  M\otimes\mathcal{W}_{D^{n}}\right)  _{\pi\left(  x\right)  }%
\otimes\mathcal{W}_{D_{n}}\underrightarrow{\,\widehat{\mathbf{\pi}}%
_{n+1,n}\left(  \nabla_{x}\right)  \otimes\mathrm{id}_{\mathcal{W}_{D_{n}}}%
}\,\left(  E\otimes\mathcal{W}_{D^{n}}\right)  _{x}\otimes\mathcal{W}_{D_{n}%
}\text{,}%
\end{align*}
which is equivalent to the composition of mappings
\begin{align*}
&  \left(  M\otimes\mathcal{W}_{D_{n}}\right)  _{\pi\left(  x\right)
}\underrightarrow{\,\mathrm{id}_{M}\otimes\mathcal{W}_{\mathbf{m}_{D_{n}\times
D_{n}\rightarrow D_{n}}}}\,\left(  M\otimes\mathcal{W}_{D_{n}}\right)
_{\pi\left(  x\right)  }\otimes\mathcal{W}_{D_{n}}\underrightarrow
{\widehat{\,\psi}_{n}\left(  \widehat{\mathbf{\pi}}_{n+1,n}\left(  \nabla
_{x}\right)  \right)  \otimes\mathrm{id}_{\mathcal{W}_{D_{n}}}}\,\\
&  \left(  E\otimes\mathcal{W}_{D_{n}}\right)  _{x}\otimes\mathcal{W}_{D_{n}%
}\,\underrightarrow{\mathrm{id}_{E}\otimes\mathcal{W}_{+_{D^{n}\rightarrow
D_{n}}\times\mathrm{id}_{D_{n}}}}\,\left(  E\otimes\mathcal{W}_{D^{n}}\right)
_{x}\otimes\mathcal{W}_{D_{n}}%
\end{align*}
This is really equivalent to the composition of mappings
\begin{align}
&  \left(  M\otimes\mathcal{W}_{D_{n}}\right)  _{\pi\left(  x\right)
}\underrightarrow{\widehat{\,\psi}_{n}\left(  \widehat{\mathbf{\pi}}%
_{n+1,n}\left(  \nabla_{x}\right)  \right)  }\,\left(  E\otimes\mathcal{W}%
_{D_{n}}\right)  _{x}\underrightarrow{\,\mathrm{id}_{E}\otimes\mathcal{W}%
_{\mathbf{m}_{D_{n}\times D_{n}\rightarrow D_{n}}}}\,\left(  E\otimes
\mathcal{W}_{D_{n}}\right)  _{x}\otimes\mathcal{W}_{D_{n}}\nonumber\\
&  \underrightarrow{\mathrm{id}_{E}\otimes\mathcal{W}_{+_{D^{n}\rightarrow
D_{n}}\times\mathrm{id}_{D_{n}}}}\,E\otimes\mathcal{W}_{D^{n}\times D_{n}}
\label{7.1.4.2.a}%
\end{align}
This just established fact that the composition of mappings in
(\ref{7.1.4.1.a}) and that in (\ref{7.1.4.2.a}) are equivalent implies the
coveted result at once. This completes the proof.
\end{proof}

\begin{proposition}
\label{t7.1.5}Let $\mathbb{D}$ be a simplicial infinitesimal space of
dimension $n$ and degree $m$. Let $\nabla_{x}$ be a $D^{n}$-pseudotangential
\textit{over }the bundle $\pi:E\rightarrow M$ \textit{at} $x\in E$ and
$\gamma\in\left(  M\otimes\mathcal{W}_{D_{n}}\right)  _{\pi\left(  x\right)
}$. Then the composition of mappings
\[
\left(  M\otimes\mathcal{W}_{D_{n}}\right)  _{\pi\left(  x\right)
}\underrightarrow{\,\mathrm{id}_{M}\otimes\mathcal{W}_{+_{\mathbb{D}%
\rightarrow D_{n}}}}\,\left(  M\otimes\mathcal{W}_{\mathbb{D}}\right)
_{\pi\left(  x\right)  }\underrightarrow{\,\nabla_{x}^{\mathbb{D}}}\,\left(
E\otimes\mathcal{W}_{\mathbb{D}}\right)  _{x}%
\]
is equivalent to the composition of mappings
\[
\left(  M\otimes\mathcal{W}_{D_{n}}\right)  _{\pi\left(  x\right)
}\,\underrightarrow{\widehat{\psi}_{n}(\nabla_{x}})\,\left(  E\otimes
\mathcal{W}_{D_{n}}\right)  _{x}\underrightarrow{\,\mathrm{id}_{E}%
\otimes\mathcal{W}_{+_{\mathbb{D}\rightarrow D_{n}}}}\,\left(  E\otimes
\mathcal{W}_{\mathbb{D}}\right)  _{x}%
\]

\end{proposition}

\begin{proof}
Let $i:D^{k}\rightarrow\mathbb{D}$ be any mapping in the standard
quasi-colimit representation of $\mathbb{D}$. The composition of mappings
\begin{align}
&  \left(  M\otimes\mathcal{W}_{D_{n}}\right)  _{\pi\left(  x\right)
}\underrightarrow{\,\mathrm{id}_{M}\otimes\mathcal{W}_{+_{\mathbb{D}%
\rightarrow D_{n}}}}\,\left(  M\otimes\mathcal{W}_{\mathbb{D}}\right)
_{\pi\left(  x\right)  }\underrightarrow{\,\nabla_{x}^{\mathbb{D}}}\,\left(
E\otimes\mathcal{W}_{\mathbb{D}}\right)  _{x}\nonumber\\
&  \underrightarrow{\mathrm{id}_{E}\otimes\mathcal{W}_{i}}\,\left(
E\otimes\mathcal{W}_{D^{k}}\right)  _{x} \label{7.1.5.1}%
\end{align}
is equivalent, by dint of Theorem \ref{t4.1.8}, to the composition of
mappings
\begin{align*}
&  \left(  M\otimes\mathcal{W}_{D_{n}}\right)  _{\pi\left(  x\right)
}\underrightarrow{\,\mathrm{id}_{M}\otimes\mathcal{W}_{\mathbf{i}%
_{D_{k}\rightarrow D_{n}}}}\left(  \,M\otimes\mathcal{W}_{D_{k}}\right)
_{\pi\left(  x\right)  }\underrightarrow{\,\mathrm{id}_{M}\otimes
\mathcal{W}_{+_{D^{k}\rightarrow D_{k}}}}\,\left(  M\otimes\mathcal{W}_{D^{k}%
}\right)  _{\pi\left(  x\right)  }\\
&  \underrightarrow{\,\nabla_{x}^{D^{k}}}\,\left(  E\otimes\mathcal{W}_{D^{k}%
}\right)  _{x}\text{,}%
\end{align*}
which is in turn equivalent, by the very definition of $\widehat{\psi}_{k}$,
to the composition of mappings
\begin{align*}
&  \left(  M\otimes\mathcal{W}_{D_{n}}\right)  _{\pi\left(  x\right)
}\underrightarrow{\,\mathrm{id}_{M}\otimes\mathcal{W}_{\mathbf{i}%
_{D_{k}\rightarrow D_{n}}}}\,\left(  \,M\otimes\mathcal{W}_{D_{k}}\right)
_{\pi\left(  x\right)  }\underrightarrow{\widehat{\,\psi}_{k}\left(
\nabla_{x}^{D^{k}}\right)  }\,\left(  E\otimes\mathcal{W}_{D_{k}}\right)
_{x}\\
&  \underrightarrow{\,\mathrm{id}_{E}\otimes\mathcal{W}_{+_{D^{k}\rightarrow
D_{k}}}}\,\left(  E\otimes\mathcal{W}_{D^{k}}\right)  _{x}\text{.}%
\end{align*}
This is indeed equivalent, by dint of Proposition \ref{t7.1.4}, to the
composition of mappings
\begin{align*}
&  \left(  M\otimes\mathcal{W}_{D_{n}}\right)  _{\pi\left(  x\right)
}\underrightarrow{\widehat{\psi}_{n}\left(  \nabla_{x}\right)  }\,\left(
E\otimes\mathcal{W}_{D_{n}}\right)  _{x}\underrightarrow{\,\mathrm{id}%
_{E}\otimes\mathcal{W}_{\mathbf{i}_{D_{k}\rightarrow D_{n}}}}\,\left(
E\otimes\mathcal{W}_{D_{k}}\right)  _{x}\\
&  \underrightarrow{\,\mathrm{id}_{E}\otimes\mathcal{W}_{+_{D^{k}\rightarrow
D_{k}}}}\,\left(  E\otimes\mathcal{W}_{D^{k}}\right)  _{x}\text{,}%
\end{align*}
which is in turn equivalent to the composition of mappings
\begin{align}
&  \left(  M\otimes\mathcal{W}_{D_{n}}\right)  _{\pi\left(  x\right)
}\underrightarrow{\widehat{\,\psi}_{n}\left(  \nabla_{x}\right)  }\,\left(
E\otimes\mathcal{W}_{D_{n}}\right)  _{x}\underrightarrow{\,\mathrm{id}%
_{E}\otimes\mathcal{W}_{+_{\mathbb{D}\rightarrow D_{n}}}}\,\left(
E\otimes\mathcal{W}_{\mathbb{D}}\right)  _{x}\nonumber\\
&  \underrightarrow{\,\mathrm{id}_{E}\otimes\mathcal{W}_{i}}\,\left(
E\otimes\mathcal{W}_{D^{k}}\right)  _{x} \label{7.1.5.2}%
\end{align}
The just established fact that the composition of mappings in (\ref{7.1.5.1})
and that in (\ref{7.1.5.2}) are equivalent implies the coveted result at once.
This completes the proof.
\end{proof}

\begin{theorem}
\label{t7.1.6}For any\textbf{\ }$\nabla_{x}\in\mathbb{J}_{x}^{D^{n}}(\pi)$, we
have $\widehat{\psi}_{n}\left(  \nabla_{x}\right)  \in\mathbb{J}_{x}^{D_{n}%
}(\pi)$.
\end{theorem}

\begin{proof}
In view of Proposition \ref{t7.1.3}, it suffices to show that $\widehat{\psi
}_{n}\left(  \nabla_{x}\right)  $ satisfies the condition in Definition
\ref{d5.1.2}. Here we deal only with the case that $n=3$ and the simple
polynomial $\rho$ at issue is $d\in D_{3}\longmapsto d^{2}\in D$, leaving the
general case safely to the reader. Since
\[
(d_{1}+d_{2}+d_{3})^{2}=2(d_{1}d_{2}+d_{1}d_{3}+d_{2}d_{3})
\]
for any $(d_{1},d_{2},d_{3})\in D^{3}$, we have the commutative diagram
\begin{equation}%
\begin{array}
[c]{ccc}%
D^{3} & \overset{\chi}{\rightarrow} & D(6)\\
+_{D^{3}\rightarrow D_{3}}\downarrow &  & \downarrow+_{D(6)\rightarrow D}\\
D_{3} & \underset{\rho}{\rightarrow} & D
\end{array}
\label{7.1.6.1}%
\end{equation}
where $\chi$ stands for the mapping
\[
(d_{1},d_{2},d_{3})\in D^{3}\mapsto\left(  d_{1}d_{2},d_{1}d_{3},d_{2}%
d_{3},d_{1}d_{2},d_{1}d_{3},d_{2}d_{3}\right)  \in D(6)
\]
Then the composition of mappings
\begin{align*}
&  \left(  M\otimes\mathcal{W}_{D}\right)  _{\pi\left(  x\right)
}\underrightarrow{\,\mathrm{id}_{M}\otimes\mathcal{W}_{\rho}}\,\left(
M\otimes\mathcal{W}_{D_{3}}\right)  _{\pi\left(  x\right)  }\,\underrightarrow
{\widehat{\psi}_{3}\left(  \nabla_{x}\right)  }\,\left(  E\otimes
\mathcal{W}_{D_{3}}\right)  _{x}\\
&  \underrightarrow{\,\mathrm{id}_{E}\otimes\mathcal{W}_{+_{D^{3}\rightarrow
D_{3}}}}\,\left(  E\otimes\mathcal{W}_{D^{3}}\right)  _{x}%
\end{align*}
is equivalent, by the very definition of $\widehat{\psi}_{3}$, to the
composition of mappings
\begin{align*}
&  \left(  M\otimes\mathcal{W}_{D}\right)  _{\pi\left(  x\right)
}\underrightarrow{\,\mathrm{id}_{M}\otimes\mathcal{W}_{\rho}}\,\left(
M\otimes\mathcal{W}_{D_{3}}\right)  _{\pi\left(  x\right)  }\underrightarrow
{\,\mathrm{id}_{M}\otimes\mathcal{W}_{+_{D^{3}\rightarrow D_{3}}}}\,\left(
M\otimes\mathcal{W}_{D^{3}}\right)  _{\pi\left(  x\right)  }\\
&  \underrightarrow{\,\nabla_{x}}\,\left(  E\otimes\mathcal{W}_{D^{3}}\right)
_{x}%
\end{align*}
which is in turn equivalent to the composition of mappings
\begin{align*}
&  \left(  M\otimes\mathcal{W}_{D}\right)  _{\pi\left(  x\right)
}\underrightarrow{\,\mathrm{id}_{M}\otimes\mathcal{W}_{+_{D(6)\rightarrow D}}%
}\,\left(  M\otimes\mathcal{W}_{D(6)}\right)  _{\pi\left(  x\right)
}\underrightarrow{\,\mathrm{id}_{M}\otimes\mathcal{W}_{\chi}}\,\left(
M\otimes\mathcal{W}_{D^{3}}\right)  _{\pi\left(  x\right)  }\\
&  \underrightarrow{\,\nabla_{x}}\,\left(  E\otimes\mathcal{W}_{D^{3}}\right)
_{x}%
\end{align*}
with due regard to the commutative diagram in (\ref{7.1.6.1}). By Theorem
\ref{t4.1.8}, this is equivalent to the composition of mappings
\begin{align*}
&  \left(  M\otimes\mathcal{W}_{D}\right)  _{\pi\left(  x\right)
}\underrightarrow{\,\mathrm{id}_{M}\otimes\mathcal{W}_{+_{D(6)\rightarrow D}}%
}\,\left(  M\otimes\mathcal{W}_{D(6)}\right)  _{\pi\left(  x\right)
}\underrightarrow{\,\nabla_{x}^{D(6)}}\,\left(  E\otimes\mathcal{W}%
_{D(6)}\right)  _{x}\\
&  \underrightarrow{\,\mathrm{id}_{E}\otimes\mathcal{W}_{\chi}}\,\left(
E\otimes\mathcal{W}_{D^{3}}\right)  _{x}%
\end{align*}
which is in turn equivalent by Proposition \ref{t7.1.5}\ to the composition of
mappings
\begin{align*}
&  \left(  M\otimes\mathcal{W}_{D}\right)  _{\pi\left(  x\right)
}\underrightarrow{\widehat{\,\mathbb{\psi}}_{1}(\pi_{3,1}(\nabla_{x}%
))}\,\left(  E\otimes\mathcal{W}_{D}\right)  _{x}\underrightarrow
{\,\mathrm{id}_{E}\otimes\mathcal{W}_{+_{D(6)\rightarrow D}}}\,\left(
E\otimes\mathcal{W}_{D(6)}\right)  _{x}\\
&  \underrightarrow{\,\mathrm{id}_{E}\otimes\mathcal{W}_{\chi}}\,\left(
E\otimes\mathcal{W}_{D^{3}}\right)  _{x}%
\end{align*}
Since
\[
\widehat{\mathbb{\psi}}_{1}(\widehat{\pi}_{3,1}(\nabla_{x}))=\widehat{\pi
}_{3,1}(\widehat{\mathbb{\psi}}_{3}(\nabla_{x}))
\]
by Proposition \ref{t7.1.4}\ and the commutativity of the diagram
(\ref{7.1.6.1}), this is equivalent to the composition of mappings
\begin{align*}
&  \left(  M\otimes\mathcal{W}_{D}\right)  _{\pi\left(  x\right)
}\underrightarrow{\,\pi_{3,1}(\widehat{\mathbb{\psi}}_{3}(\nabla_{x}%
))}\,\left(  E\otimes\mathcal{W}_{D}\right)  _{x}\underrightarrow
{\,\mathrm{id}_{E}\otimes\mathcal{W}_{\rho}}\,\left(  E\otimes\mathcal{W}%
_{D_{3}}\right)  _{x}\\
&  \underrightarrow{\,\mathrm{id}_{E}\otimes\mathcal{W}_{+_{D^{3}\rightarrow
D_{3}}}}\,\left(  E\otimes\mathcal{W}_{D^{3}}\right)  _{x}\text{,}%
\end{align*}
which completes the proof.
\end{proof}

\begin{notation}
Thus the mapping $\widehat{\mathbb{\psi}}_{n}:\widehat{\mathbb{J}}^{D^{n}}%
(\pi)\rightarrow\widehat{\mathbb{J}}^{D_{n}}(\pi)$ is naturally restricted to
a mapping $\psi_{n}:\mathbb{J}^{D^{n}}(\pi)\rightarrow\mathbb{J}^{D_{n}}(\pi)$.
\end{notation}

\end{document}